\documentclass[12pt,a4paper]{article}
\usepackage{amsmath}

\usepackage{adjustbox}
\usepackage[T1]{fontenc}
\usepackage{ShortcutStyle}
\usepackage[normalem]{ulem}
\usepackage{tikz-cd}
\usepackage{amssymb,latexsym}
\usepackage{mathrsfs}
\usepackage{psfrag}
\usepackage{enumitem}
\allowdisplaybreaks

\usepackage{soul}

\usepackage{comment}

\usepackage[all]{xy}
\usepackage{color}
\newtheorem{prop}{Proposition}[section]
\newtheorem{theo}[prop]{Theorem}
\newtheorem{cor}[prop]{Corollary}

\newtheorem{defn}[prop]{Definition} 
\newtheorem{rem}[prop]{Remark}
\newtheorem{rems}[prop]{Remarks}
\newtheorem{lem}[prop]{Lemma}

\newenvironment{proof}
 {\begin{trivlist} \item[\hskip \labelsep {\bf Proof}\hspace*{3 mm}]}
 {\hfill$\Box$\end{trivlist}}
\newenvironment{acknow}
 {\begin{trivlist} \item[\hskip \labelsep {\bf Acknowledgments.}]}
 {\end{trivlist}}

\usepackage{hyperref}
\hypersetup{colorlinks,%
citecolor=blue,%
filecolor=black,%
linkcolor=blue,%
urlcolor=black,%
pdftex}

\date{}

\begin{document}

\title{On $k$-folding map-germs and hidden symmetries of surfaces in the Euclidean 3-space}
\author{Guillermo Pe\~nafort Sanchis and  Farid Tari}

\maketitle
\begin{abstract}
Let $M$ be a smooth surface in $\mathbb R^3$  (or a complex surface in $\mathbb C^3$) and $k\geq 2$ be an integer. 
At any point on $M$ and for any  plane in $\mathbb R^3$, we construct a holomorphic map-germ $(\C^2,0)\to(\C^3,0)$ of the form $F_k(x,y)= (x,y^k,f(x,y))$, called a $k$-folding map-germ.
We study in this paper the local singularities of $k$-folding map-germs and relate them to the extrinsic differential geometry of $M$. More precisely, we

\begin{itemize}[itemsep=-5mm]
\item stratify the jet space of $k$-folding map-germs  
so that the strata of codimension $\le 4$ correspond to topologically equivalent $\mathcal A$-finitely determined germs;\\
\item obtain the  topological classification of $k$-folding map-germs on generic surfaces in $\mathbb R^3$ (or $\mathbb C^3$);\\
\item generalise the work of Bruce-Wilkinson on folding maps ($k=2$);\\ 
\item recover, in a unified way, results obtained by considering the contact of surfaces with lines, planes and spheres;\\
\item discover new robust features on smooth surfaces in $\mathbb R^3$.
\end{itemize}

\end{abstract}

\renewcommand{\thefootnote}{\fnsymbol{footnote}}

\section{Introduction }\label{sec:intro}
The aim of this work is to study $k$-folding map-germs on complex surfaces in $\mathbb C^3$ 
and relate them to the extrinsic differential geometry of surfaces in $\mathbb R^3$. 

	The standard Whitney fold of order $k$ with respect to the plane $\pi_0:y=0$ in $\mathbb C^3$  is the map
	$\omega_k\colon \C^3\to \C^3,$ given by
$$
	\omega_k(x,y,z)=(x,y^k,z).
$$	

The map $\omega_k$ `folds' the space $\C^3$ along the plane $\pi_0$, gluing the points
	$(x,y,z),$ $(x,\xi y,z),\dots,(x,\xi^{k-1}y,z),$ where $\xi=e^{2\pi i/k}$ is a primitive $k^{th}$-root of unity. 
	The Whitney fold of order $k$ with respect to any plane $\pi$, denoted  by $\omega_k^{\pi}$, is defined similarly in \S \ref{sec:Prelim}.

	Let $M$ be a complex surface in $\C^3$. We call the restriction  of  $\omega_k^{\pi}$ to $M$ the $k$-folding map on $M$ with respect to $\pi$.
	As our study is local, given point $p$ on $M$ and a plane $\pi$ in $\mathbb C^3$ through $p$, we choose a coordinate system so that 
	$M$ is locally the graph of a function $z=f(x,y)$ and $\pi=\pi_0:y=0$ (see Remarks \ref{rems:Fkpi}(4)). Then the germ at $p_0$ of the $k$-folding map is represented in standard form by the map-germ 
	$F_k: (\mathbb C^2,0)\to (\mathbb C^3,0)$, given by
\begin{equation}\label{eq:kfolding_map_standard}
	F_k(x,y)=(x,y^k,f(x,y)).	
\end{equation}
	
For an analytic (resp. smooth) surface $M\subset \R^3$, the $k$-folding map at a point $p$ on $M$ is constructed by complexifying $M$ (resp. 
a certain jet of a parametrisation of $M$) at $p$.
The singularities of a $k$-folding map-germ encode the local symmetries of $M$ with respect to the (complex) reflection group of order $k$ whose hyperplane arrangement consists of the single plane $\pi$.

The study of $2$-folding map-germs on surfaces in $\R^3$ was carried out by  Bruce and Wilkinson 
\cite{Bruce84, BruceWilk,wilkinson} (see \cite{BGT, GiblinTariRef, GiblinTariPerp,IFRT, IzumiyaTakaTari, wilkinson} for more work  on the subject),  without resorting to complexification. The real map-germs in \cite{Bruce84, BruceWilk,wilkinson} 
are called folding map-germs and our $2$-folding map-germs are their complexifications.  
Complexifying does not give  extra information when $k=2$.
For $k\geq 3$, per contra,  $k$-folding maps reveal a great deal of new geometric information. 
The local symmetries captured by these map-germs cannot be seen in the real case, which is why we call them \emph{hidden symmetries} of $M\subset \mathbb R^3$. The loci of their singularities 
are visible on $M$ and capture extrinsic geometric information of the surface.

Bruce and Wilkinson showed that
folding maps capture the sub-parabolic and ridge curves, as well as 
umbilic points and other special points on these curves: these are robust features of the surface  
(i.e., they are special geometric features that can be traced on an evolving surface; see \S\ref{sec:surfaces} for details). 
Passing to the complex setting, we  
show that the singularities of $k$-folding maps, $k\ge 2$, capture in a unified way, known robust features obtained by 
considering the contact of the surface 
with lines, planes and spheres  (parabolic, sub-parabolic, ridge and flecnodal curves, umbilic points, $B_3$, $C_3$ and $S_3$-points, $A^*_2$-points, cusps of Gauss (gulls-points) and butterfly-points).   
Our approach also reveals a new robust feature on surfaces: when $k$ is divisible by $3$, 
we obtain a new curve, called the {\it $H_3$-curve}. 
We also obtain new special points  on previously known curves as well as on the $H_3$-curve.
This motivates the following question: can the $H_3$-curve
be obtained via the contact of the surface with some special geometric object? Further work is also required for understanding the link 
between local (hidden) symmetries of a surface and its contact with lines and planes.

The paper is organised as follows. In \S \ref{sec:Prelim}, we set notation and give some preliminaries. 
In \S \ref{sec:TopInvariants}, we obtain formulae for the invariants $C,$ $T,$ $\mu({\mathcal D})$ and $r({\mathcal D})$ of $k$-folding map-germs. These 
are respectively, 
the number of cross-caps, the  number of triple points, the Milnor number and the number of branches 
of the double point curve. These invariants 
determine the finite $\mathcal A$-determinacy and the topological class of a $k$-folding map-germ. 
In  \S \ref{sec:Classification}, we produce a stratification of the $l$-jet space of $k$-folding map-germs in standard form 
which is identified with the $l$-jet space of  germs of functions $J^{l}(2,1)$.
The stratification results are summarised as follows.

\begin{theo}\label{MainThm}
	For any integer $k\ge 2$, there is a stratification $\mathcal S_k$ of $J^{11}(2,1)$ such that, 
	for any stratum $S$ in $\mathcal S_k$ of codimension $\le 4$, 
	all $k$-folding map-germs in standard form with $11$-jets in $S$ are finitely $\mathcal A$-determined and are pairwise topologically equivalent.
\end{theo}

We relate in  \S \ref{sec:hiddensymmetries} 
the stratification of the jet space to the extrinsic differential geometry of surfaces in $\mathbb R^3$. 
After clarifying what it means for a surface to be {generic}, 
 we deduce the following result about the topological classes of $k$-folding map-germs.

\begin{theo}\label{ClassThm}
	Let $k\geq 3$ be an integer and let $M$ be a generic smooth  surface in $\R^3$ {\rm (}or a complex surface in $\C^3${\rm )}. 
	Then, at any point $p$ on $M$ and for any plane $\pi$ through $p$, 
the $k$-folding map-germ at $p$ with respect to $\pi$
	is finitely $\mathcal A$-determined and is topologically equivalent to one of the following map-germs:
$$
\begin{array}{cl}
	{\bf M}^k_0	&				(x,y)\mapsto(x,y^k,y),\\
	{\bf M}^k_1	&				(x,y)\mapsto(x,y^k,xy+y^2),\\
{\bf M}^{k}_{l}&	(x,y)\mapsto(x,y^k,y^2+y^3+x^ly),\  l=2,3,4,\\
{\bf N}^k_l&	(x,y)\mapsto(x,y^k,y^2+x^2y+y^{2l-1}),\  l=3,4,\\
{\bf O}^{k}_4 &	(x,y)\mapsto (x,y^k,y^2+x^3y+xy^3),\\
		{\bf P}^k_l&					(x,y)\mapsto(x,y^k,xy+y^3+y^{3l-1}),\ l=2,3,4,\\
	{\bf Q}^k_3&					(x,y)\mapsto(x,y^k,xy+y^4+y^{5}+y^{6}),\\
	{\bf Q}^k_4&	(x,y)\mapsto(x,y^k,xy+y^4+y^{6}+y^{8}),\\
	\widetilde{{\bf Q}}^k_4&(x,y)\mapsto(x,y^k,xy+y^4+y^{5}+y^{7}),\\
		{\bf R}^k_4	&				(x,y)\mapsto(x,y^k,xy+y^5+y^6+y^7),\\
		{\bf U}^k_{3}	&			(x,y)\mapsto(x,y^k,x^2y+2xy^2+y^3+y^{4}),\\
		{\bf U}^k_{4}&	(x,y)\mapsto(x,y^k,x^2y+2xy^2+y^3+y^{8}),\\
		{\bf V}^{k,j,j'}_{4}&				(x,y)\mapsto(x,y^k,x^2y+xy^2+{a_{j,j'}}y^3+{b_{j,j'}}x^3y+y^4),\\
		{\bf W}^{k,j}_{4}&					(x,y)\mapsto(x,y^k,x^2y+xy^2+{c_j}y^3+4x^2y^2+y^4),\\
		{\bf W}^{3p,p}_4& (x,y)\mapsto(x,y^k,x^2y+y^3+y^4+y^5),\\
		{\bf X}_4^k &					(x,y)\mapsto (x,y^k,xy^2+y^3+x^3y+y^4),\\
		{\bf Y}^{k}_{4}&					(x,y)\mapsto(x,y^k,-xy^2+x^2y+y^4+y^5),
\end{array}
$$	
where the constants $a_{j,j'},b_{j,j'}$ {\rm(}resp. $c_{j}${\rm)} are as in {\rm Proposition \ref{prop:a11=a21=a22=0(1)} (}resp. {\rm \ref{prop:a11=a21=a22=0(2)})} and 
$j,j'=1,\ldots,k-1$, with $j\ne j'$.
\end{theo}

It is worth noting that $3$-folding map-germs can have  $\mathcal A$-simple singularities 
and their corresponding strata in $\mathcal S_k$ are $\mathcal A$-constant. 
For $k\ge 4$, none of the $k$-folding map-germs are 
$\mathcal A$-simple, except for immersions and for the $C_3$-singularity of $F_4$, and the strata of $\mathcal S_k$ give rise to moduli of finitely determined map-germs with constant invariants $C,T,\mu(\mathcal D)$ and $r(\mathcal D)$. 

The  robust features captured by $k$-folding map-germs on a generic surface are sketched  in Figure \ref{fig:NewRobust}. An interesting finding is that, having studied symmetries of infinitely many orders (for any $k\ge 2$), 
we obtain  a finite collection of robust features that occur along curves and a finite collection 
of special points on these curves if we discard the ${\bf V}^{k,j,j'}$ and ${\bf W}^{k,j}$-points. 

\begin{figure}
	\begin{center}
	\includegraphics[scale=0.8]{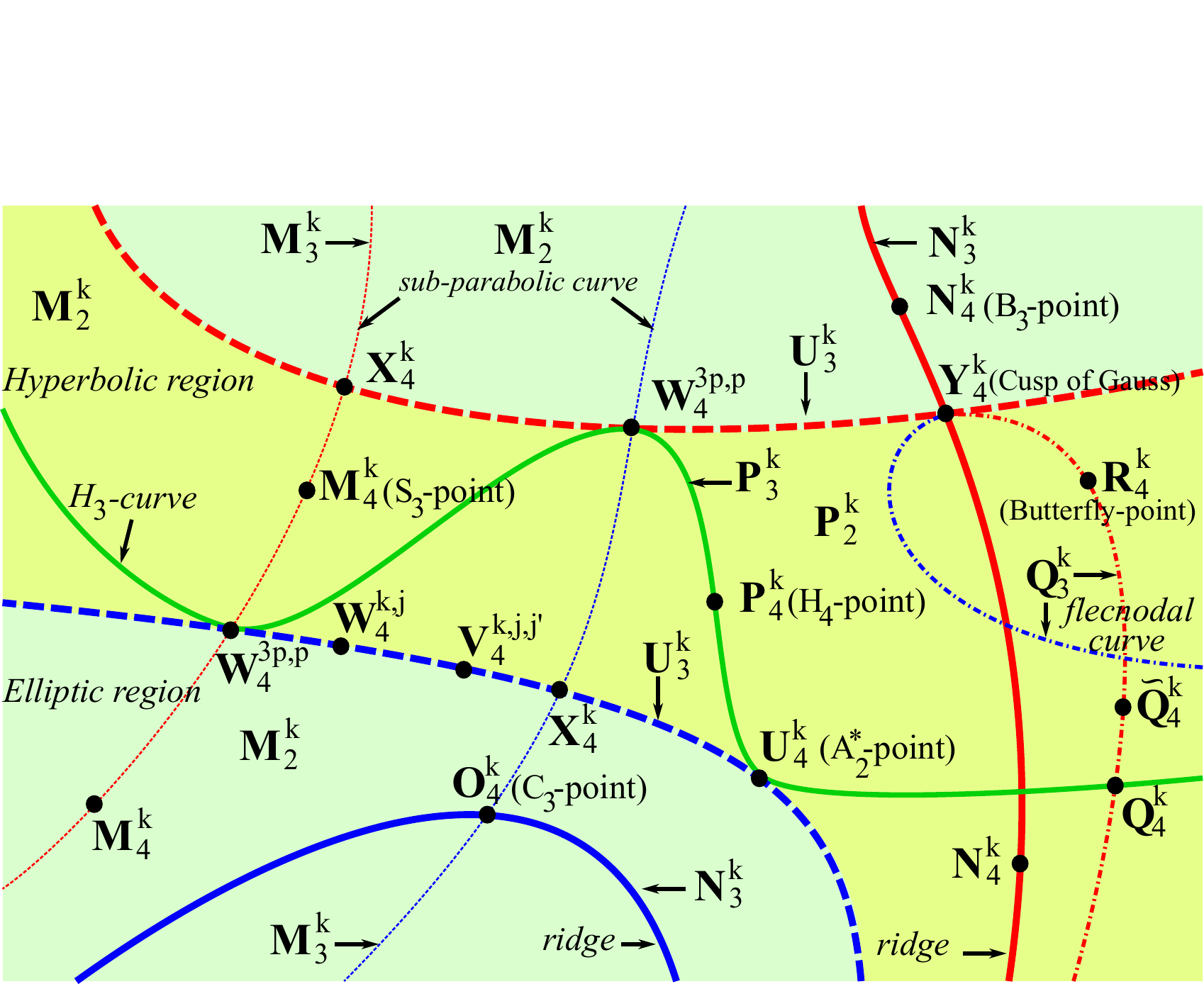}
	\end{center}
	\caption{Robust features  captured by $k$-folding map-germs away from umbilic points for $k\ge 4$, even and divisible by $3$ (see Theorem \ref{theo:umb} and Remarks \ref{rems:umb} for the robust curves at umbilic points). }
	\label{fig:NewRobust}
\end{figure}

\section{Preliminaries}\label{sec:Prelim}

We introduce here $k$-folding map-germs and notation from  singularity theory  that 
are needed in the paper. We start with the singularity theory notation, see for example \cite{Mond_Nuno, wallsurvey} for more details.

\subsection{Singularities of map-germs}

We deal with germs $F\colon (M,p)\to (N,F(p))$ of holomorphic maps 
between complex manifolds. Taking coordinate chartes,  this is the same 
as dealing with map-germs $(\mathbb C^n,0)\to (\mathbb C^p,0)$.

Let ${\mathcal O}_n$ be the local ring of germs of holomorphic 
functions $(\mathbb{C}^n,0 ) \to \mathbb{C}$ and ${\mathfrak m}_n$  its maximal ideal
(which is the subset of germs that vanish at the origin). Denote by
$\cO(n,p)$  the $\cO_n$-module of  holomorphic map-germs $(\C^n,0)\to \C^p$, so $\cO(n,p)=\bigoplus_p\cO_n$.  

Let ${\mathcal R}$ (resp. ${\mathcal L}$) be the group of bi-holomorphic germs $(\C^n,0)\to (\C^n,0)$ (resp. $(\C^p,0)\to (\C^p,0)$).
The group ${\mathcal A}={\mathcal R}\times{\mathcal L}$ of right-left equivalence
acts on ${\mathfrak m}_n.{\mathcal O}(n,p)$ by $(h_1,h_2).G=h_2\circ G\circ
h_1^{-1}$. Two germs $H,G$ are said to be $\mathcal A$-equivalent, and write $H\sim_{\mathcal A} G$, if 
$H=(h_1,h_2).G$ for some $(h_1,h_2)\in \mathcal A$.

The $l$-jet space of map-germs in $\mathfrak m_n\cdot\cO(n,p)$ is
by definition
$$J^l(n,p)={\mathfrak m_n\cdot\mathcal O}(n,p)/{\mathfrak m}_n^{l+1}\cdot{}{\mathcal O}(n,p).$$
Given a germ $G\in\mathfrak m_n\cdot{\mathcal O}(n,p)$, we identified its $l$-jet $j^lG$ with 
its Taylor polynomial of degree $l$ at the origin. 
Let $\mathcal A_{l}$ be the subgroup of $\mathcal A$ whose elements
have $l$-jets the
germ of the identity. The group $\mathcal A_l$ is a normal subgroup of $\mathcal
A$. Define
$\mathcal A^{(l)}=\mathcal A/\mathcal A_l$. The elements of $\mathcal A^{(l)}$ are the
$l$-jets of the elements of $\mathcal A$. The action of $\mathcal A$ on ${\mathfrak m}_n.{\mathcal O}(n,p)$
induces an action
of the jet group $\mathcal A^{(l)}$ on $J^l(n,p)$ as follows. For $j^lG\in J^l(n,p)$
and $j^l(h_1,h_2)\in \mathcal A^{(l)}$, $j^l(h_1,h_2).j^rG=j^l((h_1,h_2).G).$

A germ $G$ is said to be finitely $\mathcal A$-determined if there exist an integer $l$ 
such that $G\sim_{\mathcal A}H$ for any $H$ with $j^lH=j^lG$; $j^lG$ is then said to be a sufficient jet of $G$. The germ $G$ is then said to be $l$-$\mathcal A$-determined. The least $l$ satisfying this property is called the degree of determinacy of $G$. 

There are classifications of finitely determined map-germs for various pairs $(n,p)$. 
When $p=1$, there is Arnold's extensive list of the $\mathcal R$-classification of germs of functions (\cite{ArnoldEtal}). For $(n,p)=(2,2)$, 
classifications were carried out by several authors, the most extensive ones are given in \cite{Goryunov, rieger}. 
Here we need only the singularities of  $\mathcal A_e$-codimension $\le 2$, 
which we reproduce in Table \ref{tab:othgprojeSurfaR3Algb}. 
For $(n,p)=(2,3)$, Mond \cite{mond} produced an extensive list of finitely $\mathcal A$-determined map-germs. 
We use in this paper the following singularities from \cite{mond}:
\begin{center}
\begin{tabular}{cl}
	Immersion & $(x,y,0)$ \\
	Cross-cap & $(x,y^2,xy)$ \\
	$S_{k}$ & $(x, y^2, y^3 + x^{k+1}y)$, $k\ge 1$ \\
	$B_{k}$ & $(x, y^2, x^2y + y^{2k+1})$, $k\ge 2$\\
	$C_{3}$ & $(x, y^2, xy^3 + x^{3}y)$\\
	$H_k$ & $(x,xy+y^{3k-1},y^3)$, $k\ge 2$\\
	$X_4$ & $(x,y^3,x^2y+xy^2+y^{4})$
\end{tabular}
\end{center}

The notion of a {\it simple germ} is defined in  \cite{ArnoldEtal} as follows.
Let $X$ be a manifold and $\mathcal G$ a Lie group acting on $X$. The
modality of a point $g\in X$ under the action of $\mathcal G$ on $X$ is the
least number $m$ such that a sufficiently small neighbourhood of $g$
may be covered by a finite number of $m$-parameter families of
orbits. The point $g$ is said to be {\it simple} if its modality is
$0$, that is, a sufficiently small neighbourhood intersects only a
finite number of orbits. The modality of a finitely $\mathcal A$-determined
map-germ is the modality of a sufficient jet in the jet-space under
the action of the jet-group.

We also need the notion of topological equivalence. We say that two germs $H,G\in {\mathfrak m}_n\cdot{\cO(n,p)}$ are topologically equivalent if 
$H=h_2\circ G\circ h_1^{-1}$ for some germs of homeomorphisms $h_1$ and $h_2$ of, respectively, the source and target.

\subsection{Reflections and $k$-folding maps}

In all this paper, we fix  the inner product  $\langle a,b\rangle=\sum_ia_i\overline {b_i}$ in $\C^3$. 

Let $\pi$ be an element of  the affine Grassmannian ${\rm Graff}(2,3)$
of planes in  $\mathbb C^3$. 
A plane $\pi$ has equation $\left <q,v\right>=d$, where $v$ is a fixed non-zero vector orthogonal to  $\pi$ and $d$ is a fixed scalar.
However, any non-zero scalar multiple of $(d,v)$ gives an equation of $\pi$, so $\pi$ is identified with the 
 class $\overline{(d,v)}\in \mathbb C P^3$ of $(d,v)\in \mathbb C^4$.

Let $\pi:\left <q,v\right>=d$ be a plane in $\mathbb C^3$. The orthogonal projection of a point $p\in \mathbb C^3$ to $\pi$ along the vector $v$
is the point $q= p+\lambda v\in \pi$ with $\lambda=(d-\left <p,v\right>)/\left <v,v\right>$.

Consider the map $\omega_k^{\pi}\colon\C^3\to\C^3$ given by
\[
\omega_k^{\pi}(p)=q+\lambda^k v=p+\lambda v+\lambda^k v.
\]

If we take $(d',v')=(\alpha d,\alpha v)$, $\alpha\in \mathbb C\setminus0$, as another representative of $\pi=\overline{(d,v)}$, then 
$$
\begin{array}{rcl}
p+\lambda' v'+\lambda'^k v'&=&p+\frac{\alpha(d-\left <p,v\right>))}{\alpha^2\left <v,v\right>}\alpha v+
	\frac{\alpha^k(d-\left <p,v\right>)^k)}{\alpha^{2k}\left <v,v\right>}\alpha v\\
	&=&p+\lambda v+\lambda^k \alpha^{1-k} v\\
	&=&q+\lambda^k \alpha^{1-k} v.
\end{array}
$$

Clearly, the map $\omega_k^{\pi}$ depends on the points on the line $(\alpha d, \alpha v)\in \mathbb C^4$ and not merely on 
the class of the line $\overline{(d,v)}\in \mathbb CP^3$. 
However, all these maps  are $\mathcal L$-equivalent: 
the bi-holomorphic map $q-\lambda v\mapsto q-\lambda \alpha^{\frac{1-k}{k}}v$
composed (on the left) with the map $\omega_k^{\pi}$ with $\pi$ represented by $(d,v)$ gives  the map $\omega_k^{\pi}$ with $\pi$ represented by  $(\alpha d, \alpha v)$. Therefore, the   $\mathcal L$-class of $\omega_k^{\pi}$ depends only on $\pi$.

\begin{defn}The \emph{Whitney fold of order $k$} {\rm(}$k$-fold for short{\rm)} with respect to a plane $\pi\in  {\rm Graff}(2,3)$ 
is the  $\mathcal L$-class of the map $\omega^k_{\pi}$. 
We still denote by $\omega_k^{\pi}$ any representative of $\omega_k^{\pi}$ obtained by choosing a representative $(p,v)$ of 
$\pi=\overline{(d,v)}\in \mathbb CP^3$.
\end{defn}

A $k$-fold may be viewed as generalisation of the Whitney fold $(x,y^2,z)\mapsto (x,y^2,z)$. While the Whitney fold folds the space along the plane $\{y=0\}$ and identifies the points $(x,y,z)$ and $(x,-y,z)$, 
the Whitney fold of order $k$ with respect to a plane $\pi$ represented by $(d,v)$ is a generically a $k$-to-one branched cover, 
ramified along $\pi$, and identifies $k$-tuples of points
$q-\lambda v, \  q-\xi\lambda v,\ \dots\ ,\   q-\xi^{k-1}\lambda v,$
where $\xi=e^{2\pi i/k}$ is a primitive $k^{th}$-root of unity and $q\in \pi$.

The map $\omega_k^{\pi}$ can also be viewed as the quotient map associated to the action of the cyclic group $\mathbb Z/k\mathbb Z$, regarded as a complex reflection group whose hyperplane arrangement consist of the single plane $\pi$. 
We regard  $\mathbb Z/k\mathbb Z$ as the group generated by the order $k$ complex reflection 
$q-\lambda v\mapsto q-\xi\lambda v$. Observe that, even though the plane $\pi$ does not determine $\omega_k^{\pi}$ uniquely (it depends on the choice of a representative of $\pi=\overline{(d,v)}\in \mathbb CP^3$), the action of $\Z/k\Z$ on $\C^3$ is determined uniquely by $\pi$.

Given any subset $X\subseteq \C^3$, $\omega_k^{\pi}(X)$ encodes the order $k$ reflectional symmetries of $X$ with respect to $\pi$. See  \cite{GuillReflec} for a recent work on singular maps related to reflection groups.

\begin{defn} \label{def:k-foldingComplex}
Let $M\subset\C^3$ be a complex surface, $p$ a point on $M$ and $k\geq 2$ an integer. 
 Given $\pi\in  {\rm Graff}(2,3)$, the $k$-folding map-germ on $M$ at $p$ with respect to $\pi$ 
 is the $\mathcal A$-class of the restriction of $\omega_k^{\pi}$ to $M$ at $p$.
We denote any representative of the class by
$F_k^{\pi}\colon (M,p)\to( \C^3,\omega_k^{\pi}(p)).$
\end{defn}

\begin{rems}\label{rems:Fkpi}
{\rm 
1. All the map-germs $F_k^{\pi}$ with $\pi$ represented by $(\alpha d, \alpha v)$, $\alpha\in \mathbb C$, 
are $\mathcal  A$-equivalent as the maps $\omega_k^{\pi}$ are $\mathcal L$-equivalent. 
Thus, the $\mathcal A$-class of $F_k^{\pi}$ depends only on $\pi$ 
and not on the choice of a representative $(p,v)$ of $\pi=\overline{(d,v)} \in \mathbb CP^3$.
In all the paper, we work with a representative of the $\mathcal A$-class of $F_k^{\pi}$. 

2. If $p\notin \pi$, then $F_k^{\pi}$ is the germ of an immersion. Thus, to obtain any meaningful local geometric information about the surface $M$ we should  take 
the plane $\pi$ passing through the point $p\in M$.

3. The image of $F_k^{\pi}$ is the image by $\omega_k^{\pi}$ of the germ $(M,p)$, so for $p\in \pi\cap M$, 
 $F_k^{\pi}$ captures order $k$ local symmetries of $M$ with respect to $\pi$. The aim of this paper is to understand how these 
 local symmetries are captured by the $\mathcal A$-singularities of  $F_k^{\pi}$.

4. Let $p_0\in \pi\cap M$ and $(d,v)$ a representative of $\pi$.
If $v\notin T_{p_0}M$, then $F^{\pi}_k$ is a germ of an immersion and is $\mathcal A$-equivalent to $(x,y)\mapsto(x,y^k,y)$.
Suppose that $v\in T_{p_0}M$. We choose a coordinates system in $\C^3$ so that $p_0$ is the origin,
the $z$-axis is along a normal vector to $M$ at $p_0$, the $y$-axis along $v$ and the $x$-axis orthogonal to the previous two axes.
Then we can take $M$ locally at $p_0$ as the graph $z=f(x,y)$ of some holomorphic map $f$ in a neighbourhood $U$ of the origin. In 
this coordinate system, we have $\pi=\pi_0:y=0$. 
Consequently, the $k$-folding map-germ on $M$ at $p_0$ is the germ $F_k=F_k^{\pi_0}:(\C^2,0)\to (\C^3,0)$, given in standard form 
$F_k(x,y)=(x,y^k,f(x,y))$. In view of this, we shall always take a given $k$-folding map-germ in standard form \eqref{eq:kfolding_map_standard}.

5. Definition \ref{def:k-foldingComplex} is adapted as follows for the real case. 
When $M$ is an analytic surface in $\R^3$, denote by  $M_{\C,p}$ 
its local complexification   at $p$  and by $\pi_{\mathbb C}$ the  complexification of $\pi$. 
The $k$-folding map-germ on $M$ at $p$ with respect to $\pi$ is then defined as
the $\mathcal A$-class of the restriction of $\omega_k^{\pi_{\mathbb C}}$ to $M_{\C,p}$ at $p$.
When $M$ is a smooth surface,  we consider
the $k$-folding map-germ of a given jet of {\rm(}a parametrisation of{\rm)} $M$ at $p$.
}
\end{rems}


\section{Topological invariants}\label{sec:TopInvariants}

We recall the definitions of some key $\mathcal A$-invariants of map-germs $(\mathbb C^2,0)\to (\mathbb C^3,0)$. These are the Milnor number of the double point curve $\mu(\mathcal D)$, the number of cross-caps $C$ and the number of triple points $T$. 
We give formulae for computing these invariants for $k$-folding map-germs, 
and use the invariants to study the finite $\mathcal A$-determinacy and  topological equivalence 
of these germs.

\subsection{The double point curve}\label{subs:D(F)}

We start by recalling the definition of the double and triple point spaces of a corank one map-germ $F\colon (\C^n,0)\to(\C^{n+1},0)$ from  \cite{Marar_Mond}. 
Any such germ can be written in a suitable coordinate system in the form 
$
F(x,y)=(x,f_{n}(x,y),f_{n+1}(x,y)),
$
 with $x=(x_1,\dots,x_{n-1})\in(\C^{n-1},0)$ and $y\in (\C,0)$. 

Given  $h\in \cO_n$, the iterated divided differences of $h$ are defined as 
$$
\begin{array}{rcl}
h[x,y,y']&=&\dfrac{h(x,y')-h(x,y)}{y'-y}\in \cO_{n+1},\\
h[x,y,y',y'']&=&\dfrac{h[x,y,y'']-h[x,y,y']}{y''-y'}\in\cO_{n+2}.
\end{array}
$$

The multiple point ideals of a map-germ $F$ as above are defined as 
$$
\begin{array}{rcl}
I^2(F)&=&\langle f_{n}[x,y,y'],f_{n+1}[x,y,y']\rangle\subseteq \cO_{n+1},\\
I^3(F)&=&\langle f_{n}[x,y,y'],f_{n}[x,y,y',y''],f_{n+1}[x,y,y'],f_{n+1}[x,y,y',y'']\rangle\subseteq \cO_{n+2}.
\end{array}
$$

The double and triple point spaces of $F$ are, respectively,
$$
\begin{array}{rcl}
{\mathcal D}^2(F)&=&V(I^2(F)),\\
{\mathcal D}^3(F)&=&V(I^3(F)).
\end{array}
$$

By counting variables and generators, it follows that ${\mathcal D}^2(F)$ (resp. ${\mathcal D}^3(F)$) 
is a complete intersection whenever it has dimension $n-1$ (resp. $n-2$).

The double point space ${\mathcal D}^2(F)$, as a subset of $(\C^{n-1}\times \C\times \C,0)$,  consists of points $(x,y,y')$ 
such that either $y\neq y'$ and $F(x,y)=F(x,y')$ or $y'=y$ and $F$ is singular at $(x,y)$.

To define the source double point space, we assume that $F$ is finite. Then, the projection $\pi \colon {\mathcal D}^2(F)\to \C^{n-1}\times \C$ given by $(x,y,y')\mapsto (x,y)$ is also finite. As a consequence, the image of  $\pi$ can be given a complex structure as the $0$-th Fitting ideal $\mathcal F_0(\pi_*\cO_{{\mathcal D}^2(F)})$ 
of the push forward module $\pi_*\cO_{{\mathcal D}^2(F)}$ (see \cite{Mond_Pellikaan} for details). The source double point space ${\mathcal D}(F)$ is defined as the projection $\pi({\mathcal D}^2(F))$ endowed with this complex space structure, that is,
\[
{\mathcal D}(F)=V(\mathcal F_0(\pi_*\cO_{{\mathcal D}^2(F)}))\subseteq \C^n.
\]

To compute the source double point space of a $k$-folding map-germ, we need the following result.

\begin{lem} Let $Z$ be a germ of an $n$-dimensional Cohen-Macaulay space and let $h_1,\dots,h_r$ in $\cO_Z$ be regular elements. Write $X_j=V(h_j)$ and $X=V(h_1\dots h_r)$. Let $\phi\colon Z\to (\C^n,0)$ be a germ of a morphism of complex spaces such that the restrictions $\phi\vert_{X_j}\colon X_j\to(\C^n,0)$ are finite. Then 
$\mathcal F_0((\phi\vert_{X})_*\cO_{X})=\prod_{j=1}^r\mathcal F_0((\phi\vert_{X_j})_*\cO_{X_j}).$
\label{lem:0-fitting}
\end{lem}

\begin{proof}It is enough to prove the statement for the case $r=2$. 
	We can assume that $X_1$ and $X_2$ have no common irreducible component as topological spaces. 
	Indeed, consider the two subspaces $\mathcal X_1=V(h_1-t)$ and $\mathcal X_2=V(h_2)$ of $Z\times (\C,0)$ and the map $\phi\times \operatorname{Id}\colon Z\times (\C,0)\to (\C^{n+1},0)$. The  spaces 
	$\mathcal X_1$ and $\mathcal X_2$  have no common irreducible component. Moreover, 
	if the statement holds for $\phi\times \operatorname{Id}$, then it holds for $\phi$. 
	This is a consequence of the fact that Fitting ideals commute with base change (see Lemma 1.2 in \cite{Mond_Pellikaan}).
Now consider the disjoint union $X_1\sqcup X_2\subseteq Z\sqcup Z$ and the commutative diagram
\[
\begin{tikzcd}[row sep=1em,column sep=0.3em]
X_1\sqcup X_2 \arrow[rr,"\alpha"]	\arrow[dr,swap,"\psi"]&	&  X\arrow[dl,"\phi\vert_X"] \\
	& (\C^n,0)&
\end{tikzcd}
\]

The map $\alpha$ is generically a local isomorphism (that is, a local isomorphism on a Zariski open and dense subset)
because $X_1$ and $X_2$ are assumed to have no common component. Moreover, both $X$ and $X_1\sqcup X_2$ are Cohen Macaulay spaces, which implies that the ideals $\mathcal F_0(\psi_*\cO_{X_1\sqcup X_2})$ and $\mathcal F_0((\phi\vert_{X})_*\cO_{X})$ are principal. Since  $\alpha$ is generically a local isomorphism, they are necessarily equal. The statement then follows from the equalities
$$
\mathcal F_0(\psi_*\cO_{X_1\sqcup X_2})=\mathcal F_0((\phi\vert_{X_1})_*\cO_{X_1}\oplus (\phi\vert_{X_2})_*\cO_{X_2}))
=\mathcal F_0((\phi\vert_{X_1})_*\cO_{X_1})\cdot \mathcal F_0((\phi\vert_{X_2})_*\cO_{X_2}).
$$
\end{proof}

\begin{theo} \label{theo:decom_D(F_k)}
For  a $k$-folding map-germ $F_k(x,y)=(x,y^k,f(x,y))$, the double point space 
${\mathcal D}$ is the zero locus $V(\lambda)$, where $\lambda=\prod_{j=1}^{k-1}\lambda_j$ and 
\[
\lambda_j=\frac{f(x,y)-f(x,\xi^j y)}{(1-\xi^j)y},
\] 
for $1\le j\le k-1.$ We have thus a decomposition
${\mathcal D}=\bigcup_{j=1}^{k-1}{\mathcal D}_j$, with ${\mathcal D}_j=V(\lambda_j)$.
\end{theo}

\begin{proof}
The double point space ${\mathcal D}^2(F_k)$ is the intersection of the zero loci of the divided differences 
\[
\frac{((y')^k-y^k)}{y'-y}\quad {\rm and}\quad \frac{f(x,y)-f(x,y')}{y-y'}.
\]

Since $((y')^k-y^k)/{(y'-y)}=\prod_{j=1}^{k-1}(y'-\xi^j y)$, we conclude that, as a set, the space ${\mathcal D}^2(F_k)$ is the union of the spaces 
\[
{\mathcal D}^2_j=\left\{(x,y,\xi^jy)\in (\mathbb C^{n-1}\times \mathbb C\times\mathbb C,0) \mid \frac{f(x,y)-f(x,\xi^j y)}{(1-\xi^j)y}=0\right\},
\]
for $j=1,\ldots, k-1$. 
Each of the sets ${\mathcal D}^2_j$ projects to $V(\lambda_j)$, which shows that ${\mathcal D}(F_k)=V(\lambda)$ as sets.

To show the equality as complex spaces, observe that the possible dimension of ${\mathcal D}^2(F_k)$ is one or two. 
If ${\mathcal D}^2(F_k)$ has dimension two, then some branch ${\mathcal D}^2_j$ has dimension two. Therefore, the corresponding function 
$\lambda_j$ 
is identically zero, which in turn implies $\lambda =0$. Since $F_k$ is finite by construction, the projection ${\mathcal D}^2(F_k)$ is finite, hence the image of ${\mathcal D}^2_j$ is a germ of a two dimensional analytic closed subset of $(\C^2,0)$, so is equal to  $(\C^2,0)$. This implies that ${\mathcal D}(F_k)= (\C^2,0)=V(0)$. 

Suppose now that ${\mathcal D}^2(F_k)$ has dimension one. This implies that the functions 
$\prod_{j=1}^{k-1}(y'-\xi^j y)$ and $f[x,y,y']$
form a regular sequence. 
Applying Lemma \ref{lem:0-fitting} with $Z=V(f[x,y,y'])$ gives
${\mathcal D}(F_k)=V\left(\prod_{j=1}^{k-1}\mathcal F_0((\pi\vert_{{\mathcal D}^2_j})_*\cO_{{\mathcal D}^2_j}))\right)$,
where the ${\mathcal D}^2_j$, $j=1,\ldots, k-1$, are given the natural complex space structure. 
Each of the morphisms $\pi\vert_{{\mathcal D}^2_j}\colon {\mathcal D}^2_j\to \C^2$ consists of forgetting the third coordinate of the tuple $(x,y,\xi^j y)$, and this implies $\cF_0((\pi\vert_{{\mathcal D}^2_j})_*\cO_{{\mathcal D}^2_j}))=\langle \lambda_j\rangle$.
\end{proof}

Now we introduce some results to we use to check finite $\mathcal A$-determinacy of a 
$k$-folding map-germ $F_k$ and topological triviality in families of such germs. The first of these results was proven in \cite{Marar_Mond} for corank one map-germs, then extended to arbitrary corank in \cite{Marar_Nuno_Penafort}.
\begin{theo}\label{theo:mu(D)FinitDet}
A finite map-germ $F \colon(\mathbb C^2,0)\to (\mathbb C^3,0)$ 
is finitely $\mathcal A$-determined if, and only if, its double point curve ${\mathcal D}$ is reduced.
\end{theo}

The decomposition ${\mathcal D}=\bigcup_{j=1}^{k-1}{\mathcal D}_j$ in Theorem \ref{theo:decom_D(F_k)} 
can be used to compute $\mu({\mathcal D})$, making it easier to apply Theorem \ref{theo:mu(D)FinitDet} 
 (and Theorem \ref{theo:BobPe} below). 
We denote by  ${\mathcal D}_j\cdot {\mathcal D}_{j'}$ the intersection multiplicity of two distinct branches of the double point curve.
 Clearly, ${\mathcal D}_j\cdot {\mathcal D}_{j'}={\mathcal D}_{j'}\cdot {\mathcal D}_{j}$, and  
$$
{\mathcal D}_j\cdot {\mathcal D}_{j'}=\dim_\C\frac{\cO_2}{\langle \lambda_j,\lambda_{j'}\rangle}.
$$

 \begin{prop}\label{prop:Mukfoldingmaps} 
 A $k$-folding map-germ $F_k$ is finitely $\mathcal A$-determined 
 if, and only if, the Milnor numbers $\mu({\mathcal D}_j)$, $j=1,\ldots,k-1$, and the intersection multiplicities ${\mathcal D}_j\cdot {\mathcal D}_{j'}$ of all pairs ${\mathcal D}_j$ and ${\mathcal D}_j'$, with 
$j'\neq j$, are finite. In that case, 
\[
\mu({\mathcal D}(F_k))=\sum_{j=1}^{k-1}\mu({\mathcal D}_j)+2
\sum_{\substack{j,j'=1\\ j<j'}}^{k-1}{\mathcal D}_j\cdot {\mathcal D}_{j'}
-k+2.
\]
\end{prop}

\begin{proof}
By Theorem \ref{theo:mu(D)FinitDet}, $F_k$ is finitely $\mathcal A$-determined if, and only if,  
$\mu({\mathcal D}(F_k))$ is finite, equivalently, ${\mathcal D}(F_k)$ has an isolated singularity. This occurs if, and only if, every branch ${\mathcal D}_j$ has an isolated singularity and no pair of branches 
${\mathcal D}_j$ and ${\mathcal D}_{j'}$, with $j\ne j'$, have a common component. 
Using the formula 
$\mu=2\delta-r+1$ for plane curves (see \cite{Milnor})
and  the property $\delta(X\cup Y)=\delta (X)+\delta(Y)+X\cdot Y$, we get 
\begin{align*}
\mu({\mathcal D}(F_k))	&=2\delta({\mathcal D}(F_k))-r({\mathcal D}(F_k))+1\\
					&=\sum_{j=1}^{k-1}(2\delta({\mathcal D_j})-r(D_j)+1)-k+2+2\sum_{j,j'=1, j<j'}^{k-1}{\mathcal D}_j\cdot {\mathcal D}_{j'}\\
					&=\sum_{j=1}^{k-1}\mu({\mathcal D}_j)+2\sum_{j,j'=1, j<j'}^{k-1}{\mathcal D}_j\cdot {\mathcal D}_{j'}-k+2.
\end{align*}
 \end{proof}

\begin{rems}
{\rm 
1. Suppose that ${\mathcal D}_j$ is a germ of a regular curve  
parametrised by a regular map-germ  $\alpha\colon (\C,0)\to(\C^2,0)$. 
Then, ${\mathcal D}_j\cdot {\mathcal D}_{j'}=\mathrm{ord}(h_{j'}\circ\alpha),$
which is the degree of the first non zero term in the Taylor expansion
 of $\lambda_{j'}(\alpha(t))$.

2. If both ${\mathcal D}_j$ and ${\mathcal D}_{j'}$ are regular curves, we refer to ${\mathcal D}_j\cdot {\mathcal D}_{j'}$ as the \emph{order of contact between ${\mathcal D}_j$ and ${\mathcal D}_{j'}$}. We have ${\mathcal D}_j\cdot {\mathcal D}_{j'}=1$ if, and only if, the two curves intersect transversally.
Suppose they are tangential and parametrised, respectively, by $t\mapsto(t,\gamma_j(t))$ 
and $t\mapsto(t,\gamma_{j'}(t))$. Then
${\mathcal D}_j\cdot {\mathcal D}_{j'}=\mathrm{ord}(\gamma_j-\gamma_{j'}).$

3. Let $F_k$ be a finitely $\mathcal A$-determined $k$-folding map-germ.  
Then any pair of branches ${\mathcal D}_j$ and ${\mathcal D}_j'$, with $j\ne j'$, 
cannot have any common irreducible component, otherwise ${\mathcal D}(F_k)$ would fail to be reduced. 
Hence,  we have
$r({\mathcal D})=\sum_{j=1}^{k-1}r({\mathcal D}_j).$
}
\end{rems}

The defining functions $\lambda_j$ of the branches ${\mathcal D}_j$ of the 
double point curve play a major role in our study of finite $\mathcal A$-determinacy and topological equivalence of $k$-folding map-germs. 
We take $F_k(x,y)=(x,y^k,f(x,y))$ and write,  for any given integer $p\ge 1$, 
\begin{equation}\label{Taylorf}
j^pf(x,y)=\sum_{q=1}^p\sum_{s=0}^qa_{qs}x^{q-s}y^s.
\end{equation}
Then,
\begin{equation}\label{eq:Taylorlambdaj}
j^{p-1}\lambda_j=\sum_{q=1}^{p}\sum_{s=1}^{q}\vartheta_{sj} a_{qs}x^{q-s}y^{s-1},
\end{equation}
with
$$
\vartheta_{sj}=\frac{1-\xi^{sj}}{1-\xi^j}=1+\xi^j+\dots+\xi^{(s-1)j}.
$$

The constants $\vartheta_{sj}$ play a significant role in determining the singularity type of the germs $\lambda_j$ and 
in computing ${\mathcal D}_j\cdot {\mathcal D}_{j'}$. 
The following properties are needed in \S \ref{sec:Classification}.

\begin{lem}\label{lemVarthetaZeroOrOne}
The numbers $\vartheta_{sj}$ satisfy the following properties:
\begin{itemize}
\item[{\rm (1)}] 
$\vartheta_{sj}=0$ if, and only if, $k\mid sj$. 
\item[{\rm (2)}]
$\vartheta_{sj}=1$ if, and only if, $k\mid (s-1)j$.
\item[{\rm (3)}]
If $\vartheta_{sj}=\vartheta_{sj'}$, then $\vartheta_{sj}$ is either $0$ or $1$.
\end{itemize}

\label{lem:vartheta_js}
\end{lem}

\begin{proof}
We observe that $\vartheta_{0j}=0$, $\vartheta_{1j}=1$ and, for any integers $m,n$, we have
$\vartheta_{mj}=\vartheta_{(m+n)j}$ if, and only if, $k\mid nj$. 
For (1) we take $m=0$ and $n=s$, and for   (2) we take $m=1$ and $n=s-1$.

For (3), 
we observe that the constants $\vartheta_{sj}$ lie in the images of the curves 
$\gamma_n: S^1\subset \mathbb C\to \mathbb C$ given by 
$\gamma_n(z)={(1-z^{n+1})}/{(1-z)}=1+z+\ldots+z^n.$
We show that the self intersection points of the curves $\gamma_n$ are $0$ and $1$ (for $n\ge 3$).

Write $z=e^{i\theta}$, with $\theta\in [0,2\pi)$. Then $\gamma_n(\theta)=x+iy$, with $(x,y)\in\mathbb R^2$, gives
$
1-e^{i(n+1)\theta}=(x+iy)(1-e^{i\theta}).
$
Therefore,
$$
\begin{array}{rcl}
\cos((n+1)\theta)&=&1-x+x\cos(\theta)-y\sin(\theta),\\
\sin((n+1)\theta)&=&-y+y\cos(\theta)+x\sin(\theta).
\end{array}
$$
Now the identity $\cos((n+1)\theta)^2+\sin((n+1)\theta)^2=1$ gives
\begin{equation}
(1-\cos(\theta))(x^2+y^2-x)-y\sin(\theta)=0.
\label{eq:mobius}
\end{equation}

Suppose that $y=0$. Then equation \eqref{eq:mobius} becomes $(1-\cos(\theta))x(x-1)=0$, so $x=0$ or $1$ or $\theta=0$.

When $x=0$, we have $1-e^{i\theta}\ne 0$, so $ 1-e^{i(n+1)\theta}=0$. That gives $\theta=
\frac{2\pi j}{n+1}, j=1,\ldots, n$. 
Therefore, $\gamma_n$ passes $n$-times through the origin.

When $x=1$, we get $ e^{i n\theta}=1$, so $\theta=\frac{2\pi j}{n}, j=1,\ldots, n-1$. 
Therefore, $\gamma_n$ passes $(n-1)$-times through the point $1$.

When $\theta=0$, we have $ \gamma_n(0)=n+1$ and the curve has no self-intersections at that point.
See Figure \ref{fig:PlotGamma_n} for the cases $n=5,6,7.$

Suppose now that $y\ne 0$. Then equation \eqref{eq:mobius} can be written as
$
\cot({\theta}/{2})={(x^2+y^2-x)}/{y}.
$ 
This shows that, for any $(x,y)\in \mathbb R^2$ with $y\ne 0$, there is at most one $\theta\in [0,2\pi)$ satisfying
 $\gamma_n(\theta)=x+iy$. Therefore, the only self-intersection points of $\gamma_n$ are $0$ and $1$.
\end{proof}

\begin{figure}[htp]
\begin{center}
\includegraphics[scale=0.27]{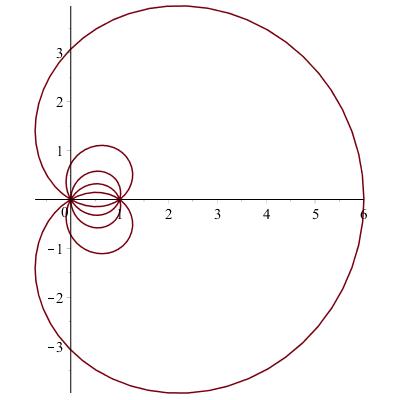}
\includegraphics[scale=0.27]{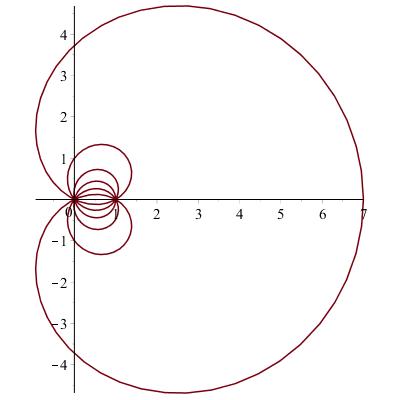}
\includegraphics[scale=0.27]{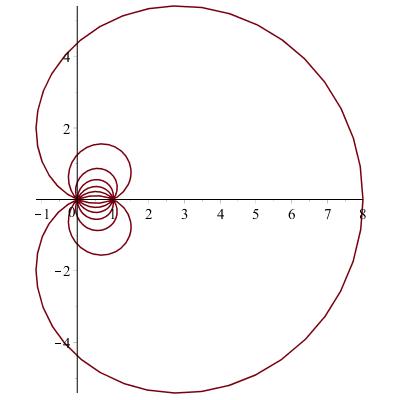}
\end{center}
\caption{Maple plots of the curves $\gamma_n$ for $n=5,6,7$ (from left to right).}
\label{fig:PlotGamma_n}
\end{figure}

\begin{rem}
{\rm 
As $k$ cannot divide $j$, the condition $k\mid sj$ in Lemma \ref{lem:vartheta_js} (1)
can also be written as
$d=\gcd(k,s)\neq 1\text{ and }j\in\{\frac{k}{d},\dots,\frac{(d-1)k}{d}\}.$
Of course, the same applies to the condition $k\mid (s-1)j$ in Lemma \ref{lem:vartheta_js} (2). }
\end{rem}


\subsection{Cross-caps and triple points}\label{subs:C and T}

The number of cross-caps and triple points are invariants of finitely $\mathcal A$-determined map-germs $F\colon (\C^2,0)\to (\C^3,0)$ that can be described using stable deformations. 
A stable mapping $U\to (\C^3,0)$ with $U$ an open neighbourhood of the origin in  $\mathbb C^2$ 
exhibits only regular points, transverse double points along curves, cross-caps and isolated transverse triple points. Every  stable deformation $F_t$ of $F$ 
exhibits the same number $C$ of cross-caps and $T$ of triple points \cite{mondRemarks}. For a corank one map-germ,
 these are given by the formulae  
\[
C=\dim_\C\frac{\cO_2}{JF},\quad T=\frac{1}{6}\dim_\C\frac{\cO_4}{I^3(F)},
\]
where $JF$ is the ideal generated by the $2\times 2$ minors of the differential matrix of $F$  \cite{mondRemarks} (the formula for $C$ holds without the corank one assumption).

For $ j,j'\in \{1,\dots,k-1\}$, with $ j\neq j'$, we set 
$$\lambda_{j,j' }=\frac{\lambda_j-\lambda_{j'}}{y}$$ 
and define 
$$
T_{j,j'}=\dim_\C\frac{\cO_2}{\langle \lambda_j,\lambda_{j,j'}\rangle}.
$$
Observe that $T_{j,j'}=T_{j',j}$ for all $j\neq j'$.

\begin{prop} The number of cross-caps and of triple points of a finitely $\mathcal A$-determined $k$-folding map-germ $F_k$ are  given by
\[
C=\dim_{\C}\dfrac{\cO_{2}}{\langle y^{k-1},\frac{\partial{f}}{\partial y}(x,y)\rangle}
\quad \mbox{and} \quad 
T=\frac{1}{3}\sum_{1\leq j<j'\leq k-1}T_{j,j'}.
\]
\end{prop}

\begin{proof}
The formula for $C$ is trivial.  
For $T$, we claim that 
\[
J=\langle \frac{y'^k-y^k}{y'-y},\frac{\frac{y''^k-y^k}{y''-y}-\frac{y'^k-y^k}{y'-y}}{y''-y}\rangle\subseteq I^3(F_k)
\]
 is a radical ideal of $\cO_{4}$ whose zero locus is
\[
Z=\{(x,y,y',y'')\mid y'=\xi^j y, y''=\xi^{j'}  \text{for some }j,j'\in \{1,\dots,k-1\} \text{ with } j\neq j'\}.
\]

Clearly, $Z$ is contained in the zero locus $V(J)$. 
The converse inclusion follows by considering the degree of the generators of $J$ and by the fact that $V(J)$ is generically reduced. 
Since $V(J)$ is also a complete intersection, it is reduced. Thus $J$ is a radical ideal.

The irreducible decomposition of $Z$ consists of the branches
\[
Z_{j,j'}=\{(x,y,\xi^j y,\xi^{j'}y)\mid x,y\in \C\}.
\]
Thus, we may decompose ${\mathcal D}^3(F_k)$, as a set, as the union of the (not necessarily irreducible) branches 
\[
{\mathcal D}^3_{j,j'}(F_k)=V(y'-\xi^jy,y''-\xi^{j'}y,f[x,y,y'],f[x,y,y',y'']).
\]

Eliminating $y'$ and $y''$, gives ${\mathcal D}^3_{j,j'}(F_k)=\langle \lambda_j,\lambda_{j,j'}\rangle$. Therefore, we only need to 
show that
$\dim_\C\frac{\cO_4}{I^3(F_k)}=\sum_{j\neq j'}\dim_\C\frac{\cO_2}{\langle\lambda_{j},\lambda_{j,j'}\rangle}$.
It is clear that the same decomposition of ${\mathcal D}^3(F_k)$ applies to an unfolding 
$\tilde F_k(x,y,t)=(x,y^k,f_t(x,y),t)$ of $F_k$, and fixing a representative and a nonzero parameter $\delta$, 
the same holds for $(\tilde F_k)_{\delta}$. This gives a decomposition of ${\mathcal D}^3(({\tilde F_k})_{\delta})$ into the branches 
${\mathcal D}^3_{j,j'}(({\tilde F_k})_{\delta})$.

It is possible to choose an unfolding where the branches  ${\mathcal D}^3_{j,j'}(({\tilde F_k})_{\delta})$ are pairwise disjoint. 
Indeed, it follows from their defining equations that two spaces ${\mathcal D}^3_{j,j'}(({\tilde F_k})_{\delta})$ and ${\mathcal D}^3_{s,s'}(({\tilde F_k})_{\delta})$ 
can only intersect on $\{y=0\}$, so it is enough to find an unfolding where ${\mathcal D}^3_{j,j'}(({\tilde F_k})_{\delta})\cap \{y=0\}$ is empty. 
Again, it follows from the defining equations that  a point $(x,0,0,0)\in {\mathcal D}^3_{j,j'}(({\tilde F_k})_{\delta})$ if  
$\frac{\partial f_{\delta}}{\partial x}(x,0)=\frac{\partial f^2_{\delta}}{\partial x^2}(x,0)=0$, 
a condition that can be avoided by choosing a suitable deformation $f_t$ of $f$. 

We take now an unfolding as above. Since $Z$ is reduced, ${\mathcal D}^3(({\tilde F_k})_{\delta})$  is isomorphic to the union of the ${\mathcal D}^3_{j,j'}(({\tilde F_k})_{\delta})$ as complex spaces. The equality we need to show follows now from the constancy of the numbers involved under continuous deformations. 
Defining $\lambda_{j}^{\delta}$ and $\lambda_{j,j'}^{\delta}$ in the obvious way, we obtain
\begin{align*}
\dim_\C\frac{\cO_4}{I^3(F_k)}&=\sum_{(x,y,y',y'')}\dim_\C\frac{\cO_4}{I^3(({\tilde F_k})_{\delta})_{(x,y,y',y'')}}\\
					&=\sum_{(x,y)}\sum_{j\neq j'}\dim_\C\frac{\cO_2}{\langle\lambda_{j}^{\delta},\lambda_{j,j'}^{\delta}\rangle}
					=\sum_{j\neq j'}\dim_\C\frac{\cO_2}{\langle\lambda_{j},\lambda_{j,j'}\rangle}.
\end{align*}

Since $T_{j,j'}=T_{j',j}$, we add only the numbers $T_{j,j'}$, with $j<j'$, and replace $1/6$ by  $1/3$.
  \end{proof}

\subsection{Topological triviality}
 We say that two subspaces $S$ and $S'$ of a topological space $X$ have the \emph{same topological type} 
 if there is a homeomorphism $X\to X$ restricting to an homeomorphism $S\to S'$. 
 Milnor showed that two isolated hypersurface singularities with same topological type have the same Milnor number.
In the case of our invariants, a similar result can be obtained using the results in \cite{NemethiPinter} and \cite{Siersma} 
(see \cite{Bobadilla_Sampaio_Penafort} for details).

\begin{prop}\label{propTopInvs}Let $F,G:(\mathbb C^2,0)\to (\mathbb C^3,0)$ be finitely $\mathcal A$-determined map-germs. If $F$ and $G$ are topologically equivalent, then
$\mu({\mathcal D}(F))=\mu({\mathcal D}(G))$,
$r({\mathcal D}(F))=r({\mathcal D}(G))$,
$C(F)=C(G)$ and 
$T(F)=T(G)$.
\end{prop}

The Milnor number of the double point curve is enough to determine the topological triviality of families of finitely $\mathcal A$-determined map-germs.

\begin{prop}[Corollary 40 \cite{Bobadilla_Pe}]\label{theo:BobPe}
A family of finitely $\mathcal A$-determined map-germs 
$G_t\colon(\mathbb C^2,0)\to(\mathbb C^3,0)$ 
is topologically trivial if, and only if, $\mu({\mathcal D}(G_t))$ is constant along the parameter $t$.
\end{prop}

Using the upper semi-continuity of the numbers involved, Proposition \ref {theo:BobPe} can be combined with Proposition \ref{prop:Mukfoldingmaps}
to yield the following result.

\begin{cor}\label{propTopTrivAndInvs}
Let $F_k^t=(x,y^k,f_t(x,y))$ be a family of finitely $\mathcal A$-determined $k$-folding map-germs. 
The following statements are equivalent:
\begin{itemize}
\item[{\rm (1)}] The family $F_k^t$ is topologically trivial.
\item[{\rm (2)}]  The numbers $\mu({\mathcal D}_j)$ and ${\mathcal D}_j\cdot {\mathcal D}_{j'}$ are constant along the family $F_k^t$.
\item[{\rm (3)}]  The numbers $C, \mu({\mathcal D}_j), {\mathcal D}_j\cdot {\mathcal D}_{j'}$ and $T_{j,j'}$ are constant along the family $F_k^t$.
\end{itemize}
\end{cor}


\section{The jet space stratification}\label{sec:Classification}
In this section, we study the singularities of $k$-folding map-germs which we take 
in standard form $F_k(x,y)=(x,y^k,f(x,y))$ (see Remarks \ref{rems:Fkpi}(4)). We identify the set of such germs  
with the set $\mathcal O_2$ of germs, at the origin, of holomorphic functions $f$.
For each $k$, we obtain a stratification $\mathcal S_k$ of $J^{11}(2,1)$ (and hence of $J^{l}(2,1)$ for $l\ge 11$). 
The stratification consists of the  strata of codimension $\leq 4$ stated in Theorem \ref{MainThm} 
together with the complement of their union (i.e., the union of strata of codimension $\ge 5$).
Every stratum of codimension $\leq 4$ of $\mathcal S_k$ consists of finitely $\mathcal A$-determined and 
pairwise topologically equivalent 
$k$-folding map-germs. The jet level $l=11$ is determined by the conditions defining the strata of $\mathcal S_k$ 
which involve the coefficients of $f$ in \eqref{Taylorf} up to degree $11$ (see Tables \ref{tab:k=3}, \ref{tab:ClassifPr}, \ref{tab:ClassifAs}, \ref{tab:ClassifPar}).

As pointed out in the introduction, the case $k=2$ was studied in \cite{Bruce84, BruceWilk,wilkinson}.
The stratification $\mathcal S_2$ can be recovered from the results in this paper. The different $\cA$-classes 
obtain in \cite{Bruce84} correspond to different topological classes. 
This follows by analysing the invariants $C,T,\mu(\mathcal D)$ and $r(\mathcal D)$. 
We shall suppose here that $k\ge 3$ and 
write the jets of $f$ as in (\ref{Taylorf}).

It is clear that $F_k$ is an immersion if and only if $a_{11}\neq 0$. All immersions are $\mathcal A$-finitely determined and 
pairwise topologically equivalent. 
Moreover, only immersions have $\mathcal D=\emptyset$, hence $r(\mathcal D)=0$.  
We define $a_{11}\neq 0$ as the open stratum of $\mathcal S_k$ (corresponding to the jets of all immersions) and choose
\[{\bf M}^k_0\colon(x,y)\mapsto(x,y^k,y)\]
as a normal form for topological equivalence of  germs in this stratum.
 
 The strata corresponding to singular germs are organized  into four branches according to the following result.
  
\begin{lem}\label{lem:Class2-jets}
For $k\ge 3$, every singular $k$-folding map-germ is $\mathcal A$-equivalent to a $k$-folding map-germ whose $2$-jet is equal to
$(x,0,xy+y^2), (x,0,y^2), (x,0,xy)$, or $(x,0,0).$
\end{lem}

\begin{proof}
As we are assuming $F_k$ to be singular 
at the origin, we have $a_{11}=0$. 
Then 
 $F_k$ is $\mathcal A$-equivalent to a germ whose $2$-jet is $(x,0,a_{21}xy+a_{22}y^2)$. Depending on the coefficients $a_{21}$ and $a_{22}$, the $2$-jet can be taken to one of the following forms:

$$
\begin{array}{lcll}
(x,0,xy+y^2) &\iff &a_{11}=0, a_{21}a_{22}\ne 0&\mbox{\rm (Branch 1)}\\
(x,0,y^2) &\iff &a_{11}=a_{21}=0,\, a_{22}\ne 0&\mbox{\rm (Branch 2)}\\
(x,0,xy) &\iff &a_{11}=a_{22}=0,\, a_{21}\ne 0&\mbox{\rm (Branch 3)}\\
(x,0,0) &\iff& a_{11}=a_{21}=a_{22}=0&\mbox{\rm (Branch 4)}
\end{array}
$$
\end{proof}

We have the following about $\mathcal A$-simplicity of germs of $k$-folding maps.

\begin{prop}\label{prop:modality}
There are no $\mathcal A$-simple $k$-folding map-germs for $k\ge 5$. 
\end{prop}
\begin{proof}
It is enough to show that the orbit of a map-germ $F_k$ with a 2-jet $(x,0,xy+y^2)$ is not simple 
as the orbits of germs in the remaining branches in Lemma \ref{lem:Class2-jets} are adjacent to it. 
For such a germ, we have $j^kF_k\sim_{\mathcal A^{(k)}}(x,(y-\frac{1}{2}x)^k,y^2)$. The result follows by Theorem 1:1 in \cite{mond} as there 
are no $\mathcal A$-simple germs of the form $(x,y^2,f(x,y))$ with $j^4f\equiv 0$.

When $k=4$ and for $F_4$ in Branch 1, we have $j^4F_4\sim_{\mathcal A^{(4)}}(x,xy^3+x^3y,y^2)$. This is a $C_3$-singularity and is 
$\mathcal A$-simple \cite{mond}. 
For $F_4$ in Branch 2, we have $j^4F_4\sim_{\mathcal A^{(4)}}(x,0,y^2)$ so it leads to non $\mathcal A$-simple germs. By adjacency, the germs in Branches 3 and 4 also lead to non $\mathcal A$-simple germs.
Therefore, the $C_3$-singularity is the only $\mathcal A$-simple singularity of $4$-folding map-germs.

The case $k=3$ is treated in \S\ref{subs:casek=3} where there are several $\mathcal A$-simple singularities of $3$-folding map-germs.
\end{proof}

\begin{rems}
{
\rm 
1. The degree of $\mathcal A$-determinacy of a singular germ $F_k$ is greater or equal to $k$ (the germ
$(x,y)\mapsto (x,0,f(x,y))$ is not finitely $\mathcal A$-determined for any $f$).

2. For germs in Branch 1 or Branch 2, we have 
$F_k\sim_{\mathcal A} (x,yp(x,y^2),y^2)$ for some germ $p$. 
One can study the germ $p$, as was done in \cite{mond}, instead of $F_k$, but this blurs the order of the original $k$-folding map-germ. 
Also, the approach in \cite{mond} of reducing the action of $\mathcal A$ on the set of $2$-folding map-germs 
to the action of a subgroup of  $\mathcal K$ on $\mathfrak  m_2$ does not extend to germs of $k$-folding maps for $k\ge 3$.
}
\end{rems}


\subsection{The case $k= 3$}\label{subs:casek=3}

When $k=3$ we get several $\mathcal A$-simple map-germs. For this reason, we  treat this case separately. 

\begin{theo}\label{theo:k=3SingReflmaps}
The only $\mathcal A$-simple singularities a $3$-folding map-germ $F_3$ can have 
are those of type $S_{2l-1}$, $l\ge 2$, or $H_s$, $s\ge 3$. In the real case, the $S_{2l-1}$-singularities are of type
$S_{2l-1}^-$. The strata of $\mathcal S_3$ of codimension $\le 4$ are given in  {\rm Table \ref{tab:k=3}}.
\end{theo}

\begin{table}[tp]
\footnotesize{
\begin{center}
\caption{The strata of $\mathcal S_3$ of codimension $\le 4$.}
\begin{tabular}{clc}
\hline 
Normal form & Defining equations and open conditions & Codim\\
\hline 
\text{Immersion} &$a_{11}\neq0$&0\\
$S_1$ &$a_{11}=0,a_{22}a_{21}\ne 0$&1\\
$S_3$ &$a_{11}=a_{21}=0,a_{22}a_{31}\ne 0$&2\\
$S_5$ &$a_{11}=a_{21}=a_{31}=0,a_{22}a_{41}\ne 0$&3\\
$S_7$ &$a_{11}=a_{21}=a_{31}=a_{41}=0,a_{22}a_{51}\ne 0$&4\\
$H_2$ &$a_{11}=a_{22}=0,a_{21}\ne 0, CndH_2\ne 0$&2\\
$H_3$ &$a_{11}=a_{22}=CndH_2=0,a_{21}\ne 0, CndH_3\ne 0$&3\\
$H_4$ &$a_{11}=a_{22}=CndH_2=CndH_3=0,a_{21}\ne 0,CndH_4\ne 0$&4\\
$X_4$&$a_{11}=a_{21}=a_{22}=0, a_{31}a_{32}a_{44}\ne 0$&4\\
${\bf U}_4^{3}$&$a_{11}=a_{21}=a_{22}=a_{44}=0, a_{31}a_{32}\ne 0, CndUm_8\ne 0$&4\\
${\bf X}_4^{3}$&$a_{11}=a_{21}=a_{22}=a_{31}=0, a_{32}a_{41}a_{44}\ne 0$&4\\
${\bf W}_4^{3,1}$&$a_{11}=a_{21}=a_{22}=a_{32}=0, a_{31}a_{44}\ne 0, a_{31}a_{55}-a_{42}a_{44}\ne 0$&4\\
\hline
\end{tabular}\label{tab:k=3}
$
\begin{array}{rcl}
CndH_2&=&a_{32}a_{44}-a_{55}a_{21}
\\
CndH_3&=&a_{88}a_{21}^3-(a_{77}a_{32}+a_{44}a_{65})a_{21}^2+a_{44}(a_{42}a_{44}+
a_{32}a_{54})a_{21}-a_{31}a_{44}^2a_{32}.
\\
CndH_4&=&a_{11,11}a_{21}^5
-(a_{44}a_{98}+a_{10,10}a_{32}+a_{77}a_{65})a_{21}^4\\
&&+(a_{44}a_{87}a_{32}+a_{32}a_{54}a_{77}+a_{44}a_{54}a_{65}+2a_{44}a_{77}a_{42}+a_{44}^2a_{75})a_{21}^3\\
&&-(a_{52}a_{44}^2+a_{44}a_{31}a_{65}+2a_{42}a_{44}a_{54}+a_{44}a_{64}a_{32}+2a_{77}a_{32}a_{31}+a_{54}^2a_{32})a_{44}a_{21}^2\\
&&+(2a_{44}a_{31}a_{42}+a_{44}a_{41}a_{32}+3a_{32}a_{54}a_{31})a_{44}^2a_{21}-2a_{31}^2a_{32}a_{44}^3
\\
CndUm_8&=& a_{55}(a_{31}a_{55}-a_{54}a_{32})+a_{77}a_{32}^2 \, \, \mbox{\rm  (see Table \ref{tab:ClassifPar})}
\end{array}
$
\end{center}
}
\end{table}

\begin{proof}
We can write $f(x,y)=f_0(x,y^3)+yf_1(x,y^3)+y^2f_2(x,y^3)$, for some germs of holomorphic functions $f_i$, $i=0,1,2$. 
Then $F_3\sim_{\mathcal A} (x,y^3,yf_1(x,y^3)+y^2f_2(x,y^3)).$ 

Suppose that $a_{10}=0$ ($F_3$ is singular)  and $a_{22}\ne 0$. Then 
$$
F_3\sim_{\mathcal A}(x,y^3,y(g(x)+y^3h(x,y^3))+y^2(a_{22}+k(x,y^3)))
$$
for some germs of  holomorphic functions $g\in \mathfrak m_1$, $h\in \mathcal O_2$ and $k\in \mathfrak  m_2$.
We can make successive changes of coordinates in the target so that
$
j^pF_3\sim_{\mathcal A^{(p)}}(x,y^3,yL(x)+a_{22}y^2)
$
for any $p\ge 3$. It is not difficult to show that $F_3$ is finitely $\mathcal A$-determined if, and only if, 
$ord(L)=ord(g)=ord( f_y (x,0))$ is finite. 
Suppose that this is the case and denote by $l$ that order. Then $j^lL(x)=a_{l1}x^l$, $a_{l1}\ne 0$, and 
the change of coordinates $y\mapsto y-(a_{l1}/2a_{22})x^l$ in the source yields 
$j^{2l+1}F_3\sim_{\mathcal A^{(2l+1)}}(x,y^3-3(a_{l1}/2a_{22})^2yx^{2l},a_{22}y^2)$.  
This is an $S_{2l-1}$-singularity, and since it is $(2l+1)$-$\mathcal A$-determined, we have 
$F_3\sim_{\mathcal A} (x,y^3-yx^{2l},y^2)$. In the real case, this is an $S_{2l-1}^{-}$-singularity.

The above calculations show, in particular, that we do not get the simple singularities $B_l^{\pm}$, $C_l^{\pm}$ and $F_4$ 
whose 2-jets are $\mathcal A^{(2)}$-equivalent to $(x,0,y^2)$.

Similar calculations show that when $a_{11}=a_{22}=0$ and $a_{21}\ne 0$, we get an $H_l$-singularity when the singularity of $F_3$ is finitely $\mathcal A$-determined.

The remaining cases are studied in the same way as in the case $k\ge 4$ (Table \ref{tab:ClassifPar}). We get three strata  with topological normal forms  ${\bf U}_4^{3}, {\bf X}_4^{3}$ and ${\bf W}_4^{3,1}$ (the germs in these strata are not $\mathcal A$-simple) together with the stratum represented by the topological normal form ${\bf U}^3_3$, which is topologically equivalent to Mond's singularity $X_4$ (see \S\ref{sec:Prelim} and \cite{mond}).
\end{proof}

\begin{rem}
{\rm 
The invariants associated to the simple singularities can be found in \cite{mondRemarks}; those 
associated to ${\bf U}_4^{3}, {\bf X}_4^{3}, {\bf W}_4^{3,1}$ are given in Table \ref{tab:InvariantsParabolicBranch}.
}
\end{rem}


\subsection{The case $k\ge 4$}\label{case:kge4}

We consider here the case when $k\ge 4$, which we divide into the four branches according to the $\mathcal A^{(2)}$-orbits 
in Lemma \ref{lem:Class2-jets}. In all that follows, $j\in\{1,\ldots,k-1\}$, and subindices of singularities indicate the codimension of the stratum.

\subsubsection{Branch 1: $a_{11}=0$, $a_{21}a_{22}\neq 0$}\label{subsec:Branch1}

\begin{theo}\label{theo:a11=0}
Any germ $F_k$ of a $k$-folding map satisfying $a_{11}=0$ and $a_{21}a_{22}\neq 0$ is 
finitely $\mathcal A$-determined and is topologically equivalent to 
\[
{\bf M}^k_1\colon(x,y)\mapsto(x,y^k,xy+y^2).
\]
The invariants take the values  $\mu({\mathcal D})=(k-2)^2,\, C=k-1,\, T=0$ and $r({\mathcal D})=k-1$ and the double point curve of $F_k$ is the union of $k-1$ regular curves intersecting transversally.
\end{theo}

\begin{proof}
The functions defining the branches ${\mathcal D}_j$ of the double point curve (see Theorem \ref{theo:decom_D(F_k)}) 
are given by 
$
\lambda_j=a_{21}x+(1+\xi^j)a_{22}y+O(2),
$
 where $O(l)$ denotes a remainder of order $l$. 
Clearly, all of the branches ${\mathcal D}_j$ are regular curves. As the scalars $1+\xi^j$ are pairwise distinct, the space
${\mathcal D}={\mathcal D}(F_k)=\bigcup_j{\mathcal D}_j$ consists of $k-1$ regular curves intersecting transversally at the origin.

We have $\mu({\mathcal D}_j)=0$ and ${\mathcal D}_j\cdot {\mathcal D}_{j'}=1$, and from Proposition \ref{prop:Mukfoldingmaps} we obtain  $\mu({\mathcal D})=(k-2)^2$. By Theorem \ref{theo:mu(D)FinitDet}, 
any germ $F_k$ satisfying the conditions in the statement of the theorem is finitely $\mathcal A$-determined. 
These germs form a stratum of codimension 1 defined by $\{a_{11}=0,a_{21}a_{22}\neq 0\}$. Since this stratum is path connected, 
we conclude by Theorem \ref{theo:BobPe} that it consists of topologically equivalent germs. 
We choose for a topological model the germ ${\bf M}^k_1$ given in the statement of the theorem.

By Proposition \ref{propTopTrivAndInvs}, it is enough 
to compute $C$ and $T_{j,j'}$ for ${\bf M}^k_1$ as these invariants are constant along the stratum. 
We have 
$C=\dim_\C {\cO_2}/{\langle y^{k-1},x+2y\rangle}=k-1$ and $T_{j,j'}=\dim_\C {\cO_2}/{\langle x+(1+\xi^j)y,1\rangle}=0.$
\end{proof}

\begin{rem}
{\rm 
The singularity ${\bf M}^4_1$ is the $\mathcal A$-simple singularity $C_3$ {\rm (}see the proof of {\rm Proposition \ref{prop:modality})}.
}
\end{rem}

The proofs for the cases in the remaining branches follow by similar arguments used in the proof of Theorem \ref{theo:a11=0} (except for the calculations of $T$). To avoid repetition, we highlight only key differences in each case.
The notation for the conditions that define the strata are those indicated in the tables.


\subsubsection{Branch 2: $a_{11}=a_{21}=0$, $a_{22}\ne 0$}\label{secCase1}

\begin{theo}\label{theo:Branch1}
The strata of codimension $\le 4$ of finitely $\mathcal A$-determined $k$-folding map-germs in the branch $a_{11}=a_{21}=0$, $a_{22}\ne 0$ are those given in {\rm Table \ref{tab:ClassifPr}}.
The invariants associated to the germs in each stratum are given in {\rm Table \ref{TableInvariantsCase1}}.
\end{theo}

\begin{proof}
The result follows from Propositions \ref{prop:a11=a21=0}, 
 \ref{prop:a11=a21=0a31ne0,keven} and \ref{prop:a11=a21=a31=a33=0,keven}.
\end{proof}

\begin{table}[htp]
\footnotesize{
\begin{center}
\caption{Strata of codimension $\le 4$ in Branch 2.}
\begin{tabular}{llc}
\hline 
Name & Defining equations and open condition& Codim\\
&together with $a_{11}=a_{21}=0$, $a_{22}\ne0$&\\
\hline 
${\bf M}^{k}_2$, $2\nmid k$& $a_{31}\ne 0$ & 2\\ 
${\bf M}^{k}_2$, $2\mid k$& $a_{33}a_{31}\ne 0$ & 2\\
${\bf M}^{k}_3$, $2\nmid k$& $a_{31}=0$, $a_{41}\ne 0$ & 3\\
${\bf M}^{k}_3$, $2\mid k$& $a_{31}=0$, $a_{33}a_{41}\ne 0$ & 3\\
${\bf M}^{k}_4$, $2\nmid k$& $a_{31}=a_{41}=0$, $a_{51}\ne 0$ & 4\\
${\bf M}^{k}_4$, $2\mid k$& $a_{31}=a_{41}=0$, $a_{33}a_{51}\ne 0$ & 4\\
${\bf N}^{k}_3$, $2\mid k$& $a_{33}=0$, 
$a_{31}\ne 0$, ${CndNA}_3\ne 0$ & 3\\
${\bf N}^{k}_4$, $2\mid k$& $a_{33}={CndNA}_3=0$, $a_{31}\ne 0$, 
${CndNA}_5\ne 0$ & 4\\
${\bf O}^{k}_4$, $2\mid k$& $a_{31}=a_{33}=0$, $a_{41}a_{43}\ne 0$&4\\
\hline
\end{tabular}
$
\begin{array}{l}
{CndNA}_3=a_{43}^2-4a_{31}a_{55}\\
{CndNA}_5=8a_{31}^3a_{77}-4a_{65}a_{43}a_{31}^2+ 2a_{53}a_{43}^2a_{31}-
a_{41}^2a_{43}^3
\end{array}
$
\label{tab:ClassifPr}
\end{center}
}
\end{table}

For germs in this branch, the map-germ $(x,y)\mapsto(x,f(x,y))$ is finite and generically two-to-one. 
Therefore, $T=0$ for any finitely $\mathcal A$-determined map-germ in this branch. 

The germs of the functions defining the double point branch ${\mathcal D}_j$ is given by
\[
\lambda_j={\vartheta_{2j}a_{22}y}+a_{31}x^2+\vartheta_{2j}a_{32}xy+\vartheta_{3j}a_{33}y^2+O(3).
\]

The branch ${\mathcal D}_j$ is thus regular if, and only if, 
$\vartheta_{2j}\ne 0$. 
By  Lemma \ref{lem:vartheta_js},  $\vartheta_{2j}=0$ when $k$ is even and $j=k/2$.

\begin{table}[tp]
	\begin{center}
		\caption{Topological invariants of germs in the strata in Table \ref{tab:ClassifPr}.}
		\footnotesize{
			\begin{tabular}{lcccc}
				\hline
				Name	&$C$ &	$T$ &	$\mu({\mathcal D})$ &	$r({\mathcal D})$
				\\
				\hline
				${\bf M}^{k}_{2}$, $2\nmid k$ & $2k-2$& 0&$2 (k-1)(k-2)-k+2$&$k-1$
				\\
				${\bf M}^{k}_{2}$, $2\mid k$& $2k-2$& 0&$(k-1)(k-2)+3-k$&$k$
				\\
				${\bf M}^{k}_{3}$, $2\nmid k$ & $3k-3$& 0&$3 (k-1)(k-2)-k+2$&$k-1$
				\\
				${\bf M}^{k}_{3}$, $2\mid k$& $3k-3$& 0&$(k-1)(k-2)+4-k$&$k-1$
				\\
				${\bf M}^{k}_{4}$, $2\nmid k$& $4k-4$& 0&$4 (k-1)(k-2)-k+2$&$k-1$
				\\
				${\bf M}^{k}_{4}$, $2\mid k$& $4k-4$& 0&$(k-1)(k-2)+5-k$&$k$
				\\
				${\bf N}^{k}_3$, $2\mid k$&$2k-2$& 	$0$		
				& 	$(2k-3)(k-2)+3$		
				& 	$k$
				\\
				${\bf N}^{k}_4$, $2\mid k$&$2k-2$& 	$0$		
				& 	$(2k-3)(k-2)+5$		
				& 	$k$
				\\
				${\bf O}^{k}_4$, $2\mid k$
				&	$3k-3$
				& 	$0$		
				& 	$(3k-4)(k-2)+4$
				& 	$k+1$\\
				\hline
			\end{tabular}
		}
		\label{TableInvariantsCase1}
	\end{center}
\end{table}

\begin{prop}\label{prop:a11=a21=0}
Any $k$-folding map-germ $F_k$ satisfying $a_{11}=a_{21}=\ldots=a_{l1}=0$, $a_{22}a_{(l+1)1}\neq 0$, 
for some $l\ge 2$, and $a_{33}\ne 0$ when $k=2p$, is finitely $\mathcal A$-determined and is topologically equivalent to 
$$
{\bf M}^{k}_{l}\colon(x,y)\mapsto(x,y^k,y^2+y^3+x^ly).
$$

The invariants $C$, $T$, $\mu(\mathcal D)$ and $r(\mathcal D)$ are as in {\rm Table \ref{tab:ClassifPr}}.
All the double point branches are regular curves except for the branch 
$\mathcal D_p$, when $k=2p$, which has an $A_{l-1}$-singularity.
We have ${\mathcal D}_j\cdot {\mathcal D}_{j'}=l$,  for all $j\ne j'$. 
 \end{prop}

\begin{proof}
Fix an index $j$ and assume that $2\nmid k$ or that $k=2p$ but $j\neq p$. Then ${\mathcal D}_j$ is a regular curve and can be parametrised by 
$t\mapsto(t,\gamma_j(t))$, with  
$\gamma_j(t)=-\frac{a_{l+1,1}}{(1+\xi^j)a_{22}}t^{l}+O(l+1).$ 
Clearly, any two distinct branches have order of contact equal to  $l$. 

Suppose now that $k=2p$. As $\vartheta^2_{p}=0$, the coefficients of $x^sy$ in $\lambda_p$ vanish for all $s\ge 1$. Moreover, since $\vartheta^3_{p}=1$, the function $\lambda_p$ is of the form 
$\lambda_p=a_{l+1,1}x^l+a_{33}y^2+y^3h(y).$
This implies that $\lambda_p$ is $\mathcal R$-equivalent to $y^2+x^{l}$ if and only if $a_{33}\neq 0$, in which case it has an $A_{l-1}$-singularity. As for the contact between branches, 
$ord(\lambda_p(t,\frac{-a_{l+1,1}}{(1+\xi^j)a_{22}}t^{l}+O(l+1))=l,$ 
hence 
${\mathcal D}_j\cdot {\mathcal D}_{p}=l$. This determines $\mu(\mathcal D)$, and hence the topological triviality and the 
constancy of the invariants along the stratum. 
\end{proof}

\begin{rems}
	{\rm 
		1. Branch 1 can be considered as a particular case of the strata ${\bf M}^k_l$ in 
		Proposition \ref{prop:a11=a21=0} (the condition $a_{33}\ne 0$ is not needed when $l=1$).
		
		2. When $2\nmid k$, the term $y^3$ in ${\bf M}^k_l$ is irrelevant for topological equivalence. 
		We include it to represent both $k$ even and $k$ odd  by the same map-germ. 
		We do this for all subsequent topological normal forms. 
	}
\end{rems}

\begin{prop}\label{prop:a11=a21=0a31ne0,keven}
Suppose that $k=2p$. 
Any  $k$-folding map-germ $F_k$  satisfying $a_{11}=a_{21}=a_{33}=0$, $a_{22}a_{31}\ne 0$ and the additional
conditions in {\rm (a)} or {\rm (b)} below is finitely $\mathcal A$-determined and is topologically equivalent to one of the germs
$$
{\bf N}^k_l\colon(x,y)\mapsto(x,y^k,y^2+x^2y+y^{2l-1}), \, l=3,4.
$$

The invariants $\mu({\mathcal D(F_k)}), C,T, r({\mathcal D}(F_k))$  are as in {\rm Table \ref{TableInvariantsCase1}.}
The branches ${\mathcal D}_j$ are regular curves for all $j\ne p$ and
 ${\mathcal D}_j\cdot {\mathcal D}_{j'}=2$ for all $j\neq j'$. 

{\rm (a)} If ${CndNA}_3\ne 0$, then the branch ${\mathcal D}_p$ has an $A_{3}$-singularity and the map-germ is topologically equivalent to ${\bf N}^k_3$. 

{\rm (b)}  If ${CndNA}_3=0$ and ${CndNA}_5\ne 0$, then ${\mathcal D}_p$ has an $A_{5}$-singularity and the map-germ is topologically equivalent to ${\bf N}^k_4$. 
\end{prop}

\begin{proof}
We have 
$
\lambda_{p}=
a_{31}x^2+a_{41}x^3+a_{43}xy^2+a_{51}x^4+ a_{53}x^2y^2+a_{55}y^4 +O(5).
$
It has an $A_3$-singularity if, and only if, ${CndNA}_3\ne 0$. 
When ${CndNA}_3=0$, we need to consider the $7$-jet of $\lambda_p$. A calculation 
shows that $\lambda_p$ has an $A_5$-singularity if, and only if, ${CndNA}_5\ne 0$.
In both cases, we have  $ord(\lambda_p(t,-{a_{31}}/{((1+\xi^j)a_{22})}t^{2}+O(3)))=2$ for all $j\ne p$.
\end{proof}

\begin{prop}\label{prop:a11=a21=a31=a33=0,keven}
Suppose that $k=2p$. Any $k$-folding map-germ $F_k$ satisfying  $a_{11}=a_{21}=a_{31}=a_{33}=0$ and $a_{22}a_{41}a_{43}\ne 0$ is finitely $\mathcal A$-determined and is topologically equivalent to 
$$
{\bf O}^{k}_4\colon (x,y)\mapsto (x,y^k,y^2+x^3y+xy^3).
$$

The codimension of the stratum is $4$ and the invariants
$\mu({\mathcal D}(F_k)), C,T,$ $r({\mathcal D}(F_k))$ are as in {\rm Table \ref{TableInvariantsCase1}.}
The branches ${\mathcal D}_j,$ $j\ne p$, are regular curves and 
${\mathcal D}_p$ has a $D_{4}$-singularity. 
We have \mbox{${\mathcal D}_j\cdot {\mathcal D}_{j'}=2$} for the distinct regular branches 
and 
\mbox{${\mathcal D}_j\cdot {\mathcal D}_{p}=3$}, for $j\ne p$. 

\end{prop}

\begin{proof}
The result follows from the fact that $\lambda_{p}=a_{41}x^3+a_{43}xy^2+O(4).$
\end{proof}


\subsubsection{Branch 3: $a_{11}=a_{22}=0$, $a_{21}\ne 0$}\label{secCase2}

\begin{theo}\label{theo:Branch2}
The strata of codimension $\le 4$ of finitely $\mathcal A$-determined $k$-folding map-germs in the branch $a_{11}=a_{22}=0$, $a_{21}\ne 0$ are those given in {\rm Table \ref{tab:ClassifAs}}. 
The invariants of the germs in each stratum are given in {\rm Table \ref{tab:invAs_strata}}.
\end{theo}

\begin{proof} The result follows from 
Propositions \ref{prop:a_{11}=a_{22}=0a_{21}ne0},
\ref{prop:a_{11}=a_{22}=a_{33}=0a_{21}ne0,4ndivk} and \ref{prop:a_{11}=a_{22}=a_{33}=a{44}=0a_{21}ne0,4ndivk}. 
\end{proof}

\begin{table}[tp]
\begin{center}
\caption{Strata of codimension $\le 4$ in Branch 3.}
\footnotesize{
\begin{tabular}{llc}
\hline 
Name & Defining equations and open conditions& Codim\\
&together with $a_{11}=a_{22}=0$, $a_{21}\ne0$&\\
\hline 
${\bf P}^k_2$, $3\nmid k$&  $a_{33}\ne 0$&2\\
${\bf P}^{k}_{2}$, $3\mid k$& $a_{33}\ne 0$, $CndH_2\ne 0$&2\\
${\bf P}^{k}_{3}$, $3\mid k$& $a_{33}\ne 0$, $CndH_2=0$, $CndH_3\ne 0$&3\\
${\bf P}^{k}_{4}$, $3\mid k$& $a_{33}\ne 0$, $CndH_2=CndH_3=0$,  $CndH_4\ne 0$&4
\\
${\bf Q}^{k}_{3}$, $3\nmid k$, $4\nmid k$& $a_{33}=0$, $a_{44}\ne 0$  &3\\         
${\bf Q}^{k}_{3}$, $3\mid k$,  $4\nmid k$& $a_{33}=0$, $a_{44}\ne 0, CndH_2\ne 0$  &3\\
${\bf Q}^{k}_{3}$, $3\nmid k$, $4\mid k$& $a_{33}=0$, $a_{44}\ne 0$, 
$CndQm_5\ne 0$  &3\\    
${\bf Q}^{k}_{3}$, $12\mid k$& $a_{33}=0$, $a_{44}\ne 0$, 
$CndH_2\ne 0$, $CndQm_5\ne 0$  &3\\
${\bf Q}^{k}_{4}$, $3\mid k$,  $4\nmid k$& $a_{33}=CndH_2=0$, $a_{44}\ne 0$, $CndH_3\ne 0$  &4\\
${\bf Q}^{k}_{4}$, $12\mid k$& $a_{33}=CndH_2=0$, $CndH_3\ne 0$, $CndQm_5\ne 0$ &4\\
$\widetilde{\bf Q}^{k}_{4}$, $3\nmid k$, $4\mid k$& $a_{33}=CndQm_5=0$, $a_{44}\ne 0$, $CndQm_6\ne 0$ &4\\         
$\widetilde{\bf Q}^{k}_{4}$, $12\mid k$& $a_{33}=CndQm_5=0$, $CndH_2\ne 0$, $CndQm_6\ne 0$ &4
\\
${\bf R}^{k}_{4}$, $4\nmid k$, $5\nmid k$& $a_{33}=a_{44}=0$, $a_{55}\ne 0$& 4
\\                                           
${\bf R}^{k}_{4}$, $4\mid k$, $5\nmid k$& $a_{33}=a_{44}=0$, $a_{55}\ne 0$, $CndQm_5\ne 0$& 4
\\
${\bf R}^{k}_{4}$, $4\nmid k$, $5\mid k$ & $a_{33}=a_{44}=0$, $a_{55}\ne 0$, $CndRm_5\ne 0$& 4
\\
${\bf R}^{k}_{4}$, $20\mid k$& $a_{33}=a_{44}=0$, $a_{55}\ne 0$, $CndQm_5\ne 0$, $CndRm_5\ne 0$ & 4
\\       
\hline                  
\end{tabular}
$
\begin{array}{rcl}
CndQm_5&=&a_{32}a_{55} - a_{21}a_{66}\\
CndQm_6&=&a_{43} a_{55}-a_{21} a_{77}\\
CndRm_5&=&a_{32}a_{66}-a_{21}a_{77}
\end{array}
$
}
\label{tab:ClassifAs}
\end{center}
\end{table}

\begin{table}[tp]
\begin{center}
\caption{Topological invariants of germs in the strata in Table \ref{tab:ClassifAs}.}
\footnotesize{
\begin{tabular}{lccccc}
\hline
	Name 
&	$C$
&	$T$
&	$\mu(\mathcal D)$
&	$r(\mathcal D)$
\\
\hline
${\bf P}^k_2$, $3\nmid k$ & $k-1$ & $\frac{(k-1)(k-2)}{6}$      & $(2k-3)(k-2)$&	 $k-1$\\   
${\bf P}^k_2$, $3\mid k$   & $k-1$ & $\frac{(k-1)(k-2)+4}{6}$	 & $(2k-3)(k-2)+4$& $k-1$\\
${\bf P}^k_3$, $3\mid k$   & $k-1$ & $\frac{(k-1)(k-2)+10}{6}$	 & $(2k-3)(k-2)+10$& $k-1$\\
${\bf P}^k_4$, $3\mid k$   & $k-1$ & $\frac{(k-1)(k-2)+16}{6}$	 & $(2k-3)(k-2)+16$& $k-1$\\
${\bf Q}^{k}_{3}$, $3 \nmid k,$ $4\nmid k$ & $k-1$ &	$\frac{(k-1)(k-2)}{3}$	& $(3k-4)(k-2)$ &$k-1$\\  
${\bf Q}^{k}_{3}$, $3\mid k,$ $4\nmid k$ & $k-1$ &	$\frac{(k-1)(k-2)+1}{3}$& $(3k-4)(k-2)+2$&$k-1$\\		    
${\bf Q}^{k}_{3}$, $3\nmid k,$ $4\mid k$ & $k-1$ &$\frac{(k-1)(k-2)+6}{3}$& $(3k-4)(k-2)+12$&$k-1$\\		    
${\bf Q}^{k}_{3}$, $12\mid k$ & $k-1$ &		$\frac{(k-1)(k-2)+7}{3}$ &	$(3k-4)(k-2)+14$&$k-1$\\	 	
${\bf Q}^{k}_{4}$, $3\mid k$, $4\nmid k$& $k-1$& $\frac{(k-1)(k-2)+4}{3}$&	$(3k-4)(k-2)+8$ &$k-1$\\	    
${\bf Q}^{k}_{4}$, $12\mid k$& $k-1$& $\frac{(k-1)(k-2)+10}{3}$ & $(3k-4)(k-2)+20$&$k-1$\\
$\widetilde{\bf Q}^{k}_{4}$, $3\nmid k,4\mid k$ &	$k-1$ &	$\frac{(k-1)(k-2)+9}{3}$&$(3k-4)(k-2)+18$&$k-1$\\	   
$\widetilde{{\bf Q}}^{k}_{4}$, $12\mid k$ & $k-1$& $\frac{(k-1)(k-2)+10}{3}$& $(3k-4)(k-2)+20$&$k-1$\\
${\bf R}^{k}_{4}$, $4 \nmid k,$ $5\nmid k$& $k-1$&$\frac{(k-1)(k-2)}{2}$&	$(4k-5)(k-2)$&	$k-1$\\	    
${\bf R}^{k}_{4}$, $4\mid k,$ $5\nmid k$ & $k-1$&	$\frac{(k-1)(k-2)+2}{2}$& $(4k-5)(k-2)+6$&$k-1$\\		    
${\bf R}^{k}_{4}$, $4\nmid k,$ $5\mid k$ & $k-1$&	$\frac{(k-1)(k-2)+8}{2}$& $(4k-5)(k-2)+24$&$k-1$\\		    
${\bf R}^{k}_{4}$, $20\mid k$ & $k-1$& $\frac{(k-1)(k-2)+10}{2}$& $(4k-5)(k-2)+30$&$k-1$
\\ 
\hline
\end{tabular}
}
\label{tab:invAs_strata}
\end{center}
\end{table}

For any finitely $\mathcal A$-determined $k$-folding map-germ in this branch, we have $C=\dim_\C {\cO_2}/{\langle y^{k-1},a_{21}x+O(2)\rangle}=k-1.$

The double point branches are regular curves that can be parametrised by $t\mapsto(\gamma_j(t),t)$, with
\[
\begin{array}{rcl}
\gamma_j(t)&=&-\frac{1}{a_{21}}\vartheta_{3j}a_{33}t^2
-\frac{1}{a_{21}^2}(\vartheta_{2j}\vartheta_{3j}a_{32}a_{33}-\vartheta_{4j}a_{21}a_{44})t^3\\
&&
-\frac{1}{a_{21}^3}(\vartheta_{2j}^2 \vartheta_{3j} a_{32}^2 a_{33}-\vartheta_{2j} \vartheta_{4j} a_{32} a_{44} a_{21}
+\vartheta_{3j}^2 a_{33}(a_{31}a_{33}-a_{43} a_{21})+\vartheta_{5j} a_{55} a_{21}^2)
t^4\\
&&+O(5).
\end{array}
\]

The strata are determined by the contact between the branches of the double point curve which depend on  
$\vartheta_{sj}$ as well as on the coefficients $a_{pq}$. We start with the case $a_{33}\ne 0$, where the  strata depend on the divisibility of $k$ by $3$.

\begin{prop}\label{prop:a_{11}=a_{22}=0a_{21}ne0}
Suppose that $a_{11}=a_{22}=0$ and $a_{21}a_{33}\neq 0$. Any $k$-folding map-germ $F_k$ in case {\rm (a)} or satisfying the additional conditions 
in {\rm (b)}
is finitely $\mathcal A$-determined and is topologically equivalent to one of the map-germs
$$
{\bf P}^k_l\colon(x,y)\mapsto(x,y^k,xy+y^3+y^{3l-1}),\text{ for }l=2,3,4.
$$

 The invariants $\mu({\mathcal D}), C,T, r(\mathcal D)$ are as in \mbox{\rm Table \ref{tab:invAs_strata}.} 
We have contact ${\mathcal D}_j\cdot {\mathcal D}_{j'}=2$, except for ${\mathcal D}_p\cdot {\mathcal D}_{2p}$ when $k=3p$ which is given in {\rm (b)}. 

\begin{itemize}
\item[{\rm (a)}] If $3\nmid k$, then $F_k$ is topologically equivalent to ${\bf P}^k_2$. 
\item[{\rm (b)}] If $k=3p$, then the strata are as follows:
	\begin{itemize}
	\item[{\rm (b1)}]  If $CndH_2\ne 0$, then ${\mathcal D}_p\cdot {\mathcal D}_{2p}=4$ and $F_k$ is topologically equivalent to ${\bf P}^k_2$. 
	
	\item[{\rm (b2)}] If $CndH_2=0$ and $CndH_3\ne 0$, then ${\mathcal D}_p\cdot {\mathcal D}_{2p}=7$ and $F_k$ is topologically equivalent to ${\bf P}^k_3$.
	
	\item[{\rm (b2)}] If $CndH_2=CndH_3=0$ and $CndH_4\ne 0$,  then ${\mathcal D}_p\cdot {\mathcal D}_{2p}=10$ and $F_k$ is topologically equivalent to ${\bf P}^k_4$.
	\end{itemize}
\end{itemize}

\end{prop}

\begin{proof}
If $3\nmid k$, then by Lemma \ref{lem:vartheta_js} we have $\vartheta_{3j}\ne 0$ and $\vartheta_{3j}\ne \vartheta_{3j'}$
for all $j\neq j'$. This implies ${\mathcal D}_j\cdot{\mathcal D}_{j'}=2$ for all $j\neq j'$.

If $k=3p$, then by Lemma \ref{lemVarthetaZeroOrOne} 
the equality $\vartheta_{3j}=\vartheta_{3j'}$ holds only when $\{j,j'\}=\{p,2p\}$. 
Again, we obtain ${\mathcal D}_j\cdot {\mathcal D}_{j'}=2$ for all $j\neq j'$ with $\{j,j'\}\neq \{p,2p\}$. 

We have $\vartheta_{3p}=\vartheta_{3(2p)}=0$, $\vartheta_{4p}=\vartheta_{4(2p)}=1$ and $\vartheta_{5p}=\vartheta_{2(2p)}$  (Lemma \ref{lemVarthetaZeroOrOne}).
Using the parametrisations of ${\mathcal D}_p$ and ${\mathcal D}_{2p}$, we get  ${\mathcal D}_p\cdot {\mathcal D}_{2p}=4$ if, and only if,  $a_{32}a_{44}-a_{21}a_{55}\ne 0$, equivalently,  
$CndH_2\ne 0$. 

When $CndH_2=0$, the exceptional branches ${\mathcal D}_p$ and ${\mathcal D}_{2p}$ are parametrised by 
$x=-\frac{a_{44}}{a_{21}}y^3+\beta_sy^7+O(9) $, $s=1,2$, with 
$\beta_1-\beta_2\ne 0$ if, and only if, $CndH_3\ne 0$. Then, ${\mathcal D}_p\cdot {\mathcal D}_{2p}=7$. 

When $CndH_2=CndH_3=0$, the exceptional branches are parametrised by 
$x=-\frac{a_{44}}{a_{21}}y^3+\beta_sy^{10}+O(11)$, $s=1,2$, with 
$\beta_1-\beta_2\ne 0$ if, and only if, $CndH_4\ne 0$. 
Then, ${\mathcal D}_p\cdot {\mathcal D}_{2p}=10$. 

The values of $T$ can be computed using the models ${\bf P}^k_l$. 
We have 
$\lambda_j=x+\vartheta_{3j}y^2+\vartheta_{(3l-1)j}y^{3l-2}$
and
$\lambda_{j,j'}=(\vartheta_{3j}-\vartheta_{3j'})y+(\vartheta_{3l-1,j}-\vartheta_{3l-1,j'})y^{3l-3}.$
By Lemma \ref{lemVarthetaZeroOrOne},  
$T_{j,j'}=\dim_\C\frac{\cO_2}{\langle x,y\rangle}=1$ when $3\nmid k$ or when $k=3p$ but $\{j,j'\}\neq \{p,2p\}$. 
For $k=3p$, 
we show that $\vartheta_{3l-1,p}\neq \vartheta_{3l-1,2p}$ so $T_{p,2p}=3l-3$.
\end{proof}

\begin{rem}
{\rm 
The singularity ${\bf P}^4_2\colon (x,y)\mapsto (x,y^4,xy+y^3+y^5)$ is topologically equivalent to the singularity $T_4\colon (x,y)\mapsto (x,y^4,xy+y^3)$ in \cite{mond}. Observe that according to Proposition \ref{prop:a_{11}=a_{22}=0a_{21}ne0}, if $k$ is not divisible by $3$, then the $y^{3l-1}$ term can be removed from the expression of ${\bf P}^k_l$, without changing the topological class of the germ.
}
\end{rem}

\begin{prop}\label{prop:a_{11}=a_{22}=a_{33}=0a_{21}ne0,4ndivk}
	Suppose that $a_{11}=a_{22}=a_{33}=0$ and $a_{21}a_{44}\neq 0$. 
	Any $k$-folding map-germ $F_k$  in case {\rm (a)} or satisfying the additional conditions 
in  
	 {\rm (b)}, {\rm (c)} or {\rm (d)}    
	is finitely $\mathcal A$-determined and is topologically equivalent to one of the following map-germs:
	\[
	{\bf Q}^k_3\colon(x,y)\mapsto(x,y^k,xy+y^4+y^{5}+y^{6})
	\]
	\[ 
	{\bf Q}^k_4\colon(x,y)\mapsto(x,y^k,xy+y^4+y^{6}+y^{8})
	\]
	\[ 
	\widetilde{{\bf Q}}^k_4\colon(x,y)\mapsto(x,y^k,xy+y^4+y^{5}+y^{7})
	\]

	The invariants  are as in {\rm Table \ref{tab:invAs_strata}.} 
We have	${\mathcal D}_j\cdot {\mathcal D}_{j'}=3$ except when $j$ and $j'$ are in the sets $J$ or $J'$ below.
	
	\begin{itemize}
		\item[{\rm (a)}] If $3\nmid k$ and $4\nmid k$, then  there are no exceptional branches and the germs in this stratum are topologically equivalent to ${\bf Q}_3^k$. 
		
		\item[{\rm (b)}] If $k=3p$ and $4\nmid k$, then $J=\{p,2p\}$.  
		\begin{itemize}
			\item[{\rm (b1)}] If $CndH_2\ne 0$, then ${\mathcal D}_p\cdot {\mathcal D}_{2p}=4$ and $F_k$ 
			is topologically equivalent to ${\bf Q}_3^k$.
				
				\item[{\rm (b2)}] If $CndH_2=0$ and $CndH_3\ne 0$,  then ${\mathcal D}_p\cdot {\mathcal D}_{2p}=7$ and $F_k$ 
				is topologically equivalent to ${\bf Q}_4^k$.
			\end{itemize}

			\item[{\rm (c)}] If $k=4p$, $3\nmid k$, then $J=\{p,2p,3p\}$. 
			\begin{itemize}
				
				\item[{\rm (c1)}] If $CndQm_5\ne 0$, then ${\mathcal D}_{j}\cdot {\mathcal D}_{j'}=5$, for all $j,j'\in J, j\ne j'$,
				and $F_k$ is topologically equivalent to  ${\bf Q}_3^k$.
				
				\item[{\rm (c2)}]If $CndQm_5=0$ and $CndQm_6\ne 0$, then ${\mathcal D}_{j}\cdot {\mathcal D}_{j'}=6$, for all $j,j'\in J, j\ne j'$,
				and $F_k$ is topologically equivalent to $\widetilde{\bf Q}_4^k$.
			\end{itemize}
			
			\item[{\rm (d)}] If $k=12p$, the exceptional contact between double point branches occurs when the indices are in $J=\{4p,8p\}$ 
			or $J'=\{3p,6p,9p\}$. There are three strata:
			
			\begin{itemize}
				\item[{\rm (d1)}] If $CndH_2\ne 0$ and $CndQm_5\ne 0$, then ${\mathcal D}_{j}\cdot {\mathcal D}_{j'}=4$ (resp. ${\mathcal D}_{j}\cdot {\mathcal D}_{j'}=5$) for all distinct pairs with $j,j'$ in $J$ (resp. $J'$), and $F_k$ is topologically equivalent to ${\bf Q}_3^k$. 
				
				\item[{\rm (d2)}] If $CndH_2=0, CndH_3\ne 0 $ and $ CndQm_5\ne 0$, then ${\mathcal D}_{j}\cdot {\mathcal D}_{j'}=7$ (resp. ${\mathcal D}_{j}\cdot {\mathcal D}_{j'}=5$) for all distinct pairs with $j,j'$ in $J$ (resp. $J'$), and $F_k$ is 
			 topologically equivalent to ${\bf Q}_4^k$. 
				
				\item[{\rm (d3)}] If $CndQm_5=0, CndQm_6\ne 0,$ and $ CndH_2\ne 0$, then 
				 ${\mathcal D}_{j}\cdot {\mathcal D}_{j'}=4$ (resp. ${\mathcal D}_{j}\cdot {\mathcal D}_{j'}=6$)  for all distinct pairs with $j,j'$ in $J$ (resp. $J'$), and $F_k$ is 
				 topologically equivalent to $\widetilde{\bf Q}_4^k$.
			\end{itemize}
			
		\end{itemize}
	\end{prop}

\begin{proof}
The branches $\mathcal D_j$ can be parametrised by $t\mapsto (\gamma_j(t),t)$ 
with 
\[
\begin{array}{rcl}
\gamma_j(t)&=&-\frac{1}{a_{21}}\vartheta_{4j}a_{44}t^3
-\frac{1}{a_{21}^2}(\vartheta_{2j}\vartheta_{4j}a_{32}a_{44}-\vartheta_{5j}a_{21}a_{55})t^4\\
&&
-\frac{1}{a_{21}^3}
(\vartheta_{2j}^2 \vartheta_{4j} a_{32}^2 a_{44}-\vartheta_{3j} \vartheta_{4j} a_{43} a_{44} a_{21}
-\vartheta_{2j} \vartheta_{5j} a_{21}a_{32}a_{55}+\vartheta_{6j} a_{66} a_{21}^2)
t^5+O(6).
\end{array}
\]

(a) If $3\nmid k$, $4\nmid k$, then $\vartheta_{4j}\ne \vartheta_{4j'}$ 
for all $j,j'$ with $j\ne j'$ (Lemma \ref{lem:vartheta_js}). Therefore,  
${\mathcal D}_j\cdot {\mathcal D}_{j'}=3$ for all distinct pairs. 

(b) If $k=3p$ and $4\nmid p$, then $\vartheta_{4j}\ne \vartheta_{4j'}$ 
for all distinct pairs with $j$ or $j'$ not in $J=\{p,2p\}$. For such pairs, 
${\mathcal D}_j\cdot {\mathcal D}_{j'}=3$.

Using the fact that $\vartheta_{6p}=\vartheta_{6(2p)}=\vartheta_{3p}=\vartheta_{3(2p)}=0$,
$\vartheta_{4p}=\vartheta_{4(2p)}=1$
 and $\vartheta_{5p}=\vartheta_{2p}\neq\vartheta_{5(2p)}=\vartheta_{2(2p)}$, the parametrisation 
 of the exceptional branch $\mathcal D_p$
becomes
\[
\gamma_p(t)=-\frac{a_{44}}{a_{21}}t^3
-\frac{\vartheta_{2p}(a_{32}a_{44}-a_{21}a_{55})}{a_{21}^2}t^4-\frac{\vartheta_{2p}^2a_{32} (a_{32}a_{44}-a_{21}a_{55})}{a_{21}^3}
t^5+O(6).
\]

A parametrisation of $\mathcal D_{2p}$ is obtained by replacing
$\vartheta_{2p}$ by $\vartheta_{2(2p)}$ in the expression of $\gamma_p$. 
Therefore, ${\mathcal D}_p\cdot {\mathcal D}_{2p}=4$ if, and only if, 
$a_{32}a_{44}-a_{21}a_{55}\ne 0$, i.e., $CndH_2\ne 0$.

When $CndH_2= 0$, we have 
\[
\gamma^1_{p}=-\frac{a_{44}}{a_{21}}t^3-\frac{a_{31} a_{44}^2 - a_{21} a_{44}a_{54} + a_{21}^2 a_{77} }{a_{21}^3}t^6+\\
\frac{\vartheta^2_{p}CndH_3}{a^4_{21}}t^7+O(8).
\]

For $\gamma_{2p}$, we replace $\vartheta^2_{p}$ by $\vartheta^2_{2p}$ in $\gamma_{p}$. 
It follows that ${\mathcal D}_p\cdot {\mathcal D}_{2p}=7$ if, and only if, 
$CndH_3\ne 0$.

\smallskip

(c)  If $k=4p$ and $3\nmid p$, we have
${\mathcal D}_j\cdot {\mathcal D}_{j'}=3$ for 
$j$ or $j'$ not in $\{p,2p,3p\}$. 

The parametrisation of $\mathcal D_p$ becomes 
$
\gamma_{p}(t)={a_{55}}/{a_{21}}t^4+
{\vartheta^2_{p}(a_{32}a_{55} - a_{21}a_{66})}/{a_{21}^2} t^5+O(6),
$
and similarly for $\gamma_{2p}$ and $\gamma_{3p}$ replacing $p$ by $2p$ and $3p$ respectively. 
Consequently, the branches $\mathcal D_p$, $\mathcal D_{2p}$, $\mathcal D_{3p}$ have pairwise order of contact $5$ if, and only if, 
$a_{32}a_{55} - a_{21}a_{66}\ne 0$, i.e., $CndQm_5\ne 0.$

When $CndQm_5=0$, 
$\gamma_{p}=-{a_{55}}/{a_{21}}t^4-
\vartheta_{2p}\vartheta_{3p}({a_{21} a_{77} - a_{43} a_{55}})/
{a_{21}^2}t^6+O(7)$, with similar ajustements as above for $\gamma_{2p}$ 
and $\gamma_{3p}$.
Therefore, the three exceptional branches 
have pairwise order of contact $6$ if, and only if, $a_{21} a_{77}-a_{43} a_{55}\ne 0$, i.e., $CndQm_6\ne 0$.

(d)  This follows by Lemma \ref{lem:vartheta_js} and (b) and (c) above.

The contact between the branches determines $\mu(\mathcal D)$, the topological types and their associated strata. 
It remains to compute $T$ for each normal form. 

For ${\bf Q}^k_3$ we have 
 $\lambda_j=x+\vartheta_{4j}y^3+\vartheta_{5j}y^{5}+\vartheta_{6j}y^5,$ and 
$
\lambda_{j,j'}=(\vartheta_{4j}-\vartheta_{4j'})y^2+(\vartheta_{5j}-\vartheta_{5j'})y^3+(\vartheta_{6j}-\vartheta_{6j'})y^4.
$
Using the properties of $\vartheta_{sj}$ in Lemma \ref{lemVarthetaZeroOrOne}, 
we have $T_{j,j'}=2$ unless $k=3p$ and $j,j'\in \{p,2p\}$, or 
$k=4p$ and $j,j'\in \{p,2p,3p\}$. 
In the first case we get $T_{j,j'}=3$ and in the second 
 $T_{j,j'}=4$. 
 
 The invariant for the germ ${\bf Q}^k_4$  differs from ${\bf Q}^k_3$ only when $k$ is divisible by $3$. 
 For $k=3p$,  we have $T_{p,2p}=6$ and $T_{j,j'}=4$ if $k=4q$ and $j,j'\in\{q,2q,3q\}$. 
 All other indices $j$ and $j'$ give $T_{j,j'}=2$. 
 
 Similarly, for the germ $\widetilde {\bf Q}^k_4$ and for $k=4p$, we get $T_{j,j'}=5$ if $j,j'\in\{p,2p,3p\}$.  
If $k=3q$, then $T_{q,2q}=3$. All other indices $j$ and $j'$ give $T_{j,j'}=2$. 
\end{proof}

\begin{rem}\label{rem:Unequivalent}
{\rm 
When $k$ is divisible by $12$, the germs ${\bf Q}^k_4$ and $\widetilde{\bf Q}^k_4$ have the same invariants $C,T,\mu(D)$ and $r(D)$ but they are not topologically equivalent as their associated sets of contacts between double points branches are distinct (see \cite{Zariski}).
}
\end{rem}

\begin{prop}\label{prop:a_{11}=a_{22}=a_{33}=a{44}=0a_{21}ne0,4ndivk}
Suppose that $a_{11}=a_{22}=a_{33}=a_{44}=0$ and $a_{21}a_{55}\neq 0$. 
Any $k$-folding map-germ $F_k$ in case {\rm (a)} or satisfying the additional conditions 
in {\rm (b)} or {\rm (c)}   
is finitely $\mathcal A$-determined and is topologically equivalent to  
\[{\bf R}^k_4\colon(x,y)\mapsto(x,y^k,xy+y^5+y^6+y^7).\]

The invariants associated to the germs in the stratum are as in {\rm Table \ref{tab:invAs_strata}}.
We have ${\mathcal D}_j\cdot {\mathcal D}_{j'}=4$ except for the exceptional pairs of branches below.

\begin{itemize}
\item[{\rm (a)}] If $4\nmid k$ and $5\nmid k$, there are no additional conditions and no exceptional branches. 

\item[{\rm (b)}] If $k=4p$, $5\nmid k$ and $CndQm_5\ne 0$, then  ${\mathcal D}_j\cdot {\mathcal D}_{j'}=5$ 
for all distinct pairs with $ j,j'$ in $J=\{p,2p,3p\}$.
 
 \item[{\rm (c)}] If $k=5p$, $4\nmid k$ and $CndRm_5\ne 0$, ${\mathcal D}_j\cdot {\mathcal D}_{j'}=5$  for all distinct pairs with $ j,j'$ in $J=\{p,2p,3p,4p\}$.

\item[{\rm (d)}] If $k=20p$, $CndQm_5\ne 0$ and $CndRm_5\ne 0$, then ${\mathcal D}_j\cdot {\mathcal D}_{j'}=5$ for all distinct pairs with $ j,j'$ in $J=\{5p,10p,15p\}$ or in $J'=\{4p,8p,12p,16p\}$.
 \end{itemize}

\end{prop}

\begin{proof} 
The branches $\mathcal D_j$ can be parametrised by $t\mapsto (\gamma_j(t),t)$ 
with 
\[
\begin{array}{rcl}
\gamma_j(t)&=&-\frac{1}{a_{21}}\vartheta_{5j}a_{55}t^4+
\frac{1}{a_{21}^2}(\vartheta_{2j}\vartheta_{5j}a_{32}a_{55} - \vartheta_{6j}a_{21}a_{66})t^5+\\
&&\frac{1}{a_{21}^3} (a_{21} (\vartheta_{2j}\vartheta_{6j}a_{32} a_{66} - \vartheta_{7j}a_{21} a_{77}) - ((\vartheta_{2j})^2a_{32}^2 - \vartheta_{3j}a_{21} a_{43})\vartheta_{5j} a_{55})t^6+O(7).
\end{array}
\]

(a) If $4\nmid k$ and $5\nmid k$, then $\vartheta_{5j}\ne \vartheta_{5j'}$ 
for all $j,j'$ with $j\ne j'$ (by Lemma \ref{lem:vartheta_js}). Therefore,  
${\mathcal D}_j\cdot {\mathcal D}_{j'}=4$ for all distinct pairs. 

(b) If $k=4p$ and $5\nmid k$, then  $\vartheta_{5p}=1$, 
$\vartheta_{6p}=\vartheta_{2p}$, and we can write 
$\gamma_p(t)=-a_{55}/a_{21}t^4+\vartheta_{2p}CndQm_5/a_{21}^2t^5+O(6)$. 
We get similarly $\gamma_{2p}(t) $ and $\gamma_{3p}(t)$ 
by substituting  
$\vartheta_{2p}$ by, respectively, $\vartheta_{2(2p)}$ and $\vartheta_{2(3p)}$ in $\gamma_p$.
The result follows as $\vartheta_{2s}\ne \vartheta_{2s'}$ for 
$s,s'\in \{p,2p,3p\}$, $s\ne s'$.

(c) If $k=5p$ and $4\nmid k$,  then  $\vartheta_{5p}=0$, 
$\vartheta_{6p}=1$, $\vartheta_{7p}=\vartheta_{2p}$, so
$\gamma_p(t)=-a_{66}/a_{21}t^5+\vartheta_{2p}CndRm_5/a_{21}^2t^6+O(7)$. The expressions 
$\gamma_{sp}$, $s=2,3,4,$ are obtaining by substituting  
$\vartheta_{2p}$ by $\vartheta_{2(sp)}$ in $\gamma_p$.

(d) The case $k=20p$ follows in a similarly way to that of Proposition \ref{prop:a_{11}=a_{22}=a_{33}=0a_{21}ne0,4ndivk}(d). 

The calculations for $T$ are similar to those in the proof of Proposition \ref{prop:a_{11}=a_{22}=a_{33}=0a_{21}ne0,4ndivk}. We obtain $T_{j,j'}=4$ if $k=4p$ and $j,j'\in\{p,2p,3p\}$, and $T_{j,j'}=5$ if $k=5q$ and $j,j'\in\{q,2q,3q,4q\}$. All remaining indices $j$ and $j'$ give $T_{j,j'}=3$.
\end{proof}

\begin{rem}\label{remCompareQ43AndR44}
{\rm
The singularities ${\bf Q}^4_3$ and ${\bf R}^4_4$ 
have the same invariants, but the contact of their double point branches shows that they are not topologically equivalent.
}
\end{rem}

\subsubsection{Branch 4: $a_{11}=a_{21}=a_{22}=0$}\label{secCase3}

\begin{theo}\label{theo:Branch3}
The strata of codimension $\le 4$ of finitely $\mathcal A$-determined $k$-folding map-germs in the branch $a_{11}=a_{22}=a_{21}=0$ are those given in {\rm Table \ref{tab:ClassifPar}}. 
The invariants of the germs in each stratum are given in {\rm Table \ref{tab:InvariantsParabolicBranch}}.
\end{theo}

\begin{proof} The result follows from 
Propositions \ref{prop:a11=a21=a22=0(0)}, \ref{prop:a11=a21=a22=0(1)}, \ref{prop:a11=a21=a22=0(2)},
\ref{prop:a11=a21=a22=0(4)}, \ref{prop:a11=a21=a22=0(5)}. 
\end{proof}

\begin{table}[htp]
\begin{center}
\caption{strata of codimension $\leq 4$ in Branch 4.}
\footnotesize{
\begin{tabular}{llc}
\hline 
Name & Defining equations and open conditions & Codim\\
& together with $a_{11}=a_{21}=a_{22}=0$&\\
\hline 
${\bf U}^{k}_{3}$, $3\nmid k$ &$a_{31}a_{33}\ne 0$, $\Delta_j^k\ne 0$, $\Omega_{j,k-j}\ne 0$&3
\\    
${\bf U}^{k}_{3}$, $3\mid k$ &$a_{31}a_{33}\ne 0$, $a_{44}\ne 0$, $\Delta_j^k\ne 0$, $\Omega_{j,j'}^k\ne 0$&3
\\    
${\bf U}^{k}_{4}$, $3\mid k$ &$a_{31}a_{33}\ne 0$, $a_{44}=0$, $\Delta_j^k\ne 0$, $\Omega_{j,j'}^k\ne 0$, $CndUm_8\ne 0$&4
\\    
${\bf V}_4^{k,j,j'}$, (*)&$a_{31}a_{33}\ne 0$, $\Omega_{j,j'}^k=0$, $a_{32}\ne 0$, $a_{44}\ne 0$, ${CndVm_5}_{j,j'}\ne 0$ &4
\\     
${\bf W}_4^{k,j}$, $3\nmid k$,  (**)&$a_{31}a_{33}\ne 0$,  $\Delta_{j}^k=0$, $CndWA_2\ne 0$&4
\\    
${\bf W}_4^{k,j}$, $3\mid k$, (**), (***)&$a_{31}a_{33}\ne 0$,  $\Delta_{j}^k=0$, $CndWA_2\ne 0$, $a_{44}\ne 0$&4\\
${\bf W}_4^{3p,p}$, $k=3p$&$a_{31}a_{33}\ne 0$,  $a_{32}=0$, $a_{44}\ne 0$, ${Cnd\widetilde{W}m}_8\ne 0$&4
\\    
${\bf X}_4^k$, $2\nmid k$, $3\nmid k$&$a_{32}a_{33}\ne 0$,  $a_{31}=0$, $a_{41}\ne 0$&4
\\
${\bf X}_4^k$, $2\mid k$, $3\nmid k$&$a_{32}a_{33}\ne 0$,  $a_{31}=0$, $a_{41}\ne 0$&4
\\
${\bf X}_4^k$, $2\nmid k$, $3\mid k$&$a_{32}a_{33}\ne 0$,  $a_{31}=0$, $a_{41}\ne 0$, $a_{44}\ne 0$&4
\\
${\bf X}_4^k$, $6\mid k$&$a_{32}a_{33}\ne 0$,  $a_{31}=0$, $a_{41}\ne 0$, $a_{44}\ne 0$&4
\\
${\bf Y}_4^k$, $2\nmid k$&$a_{32}a_{31}\ne 0$,  $a_{33}=0$, $a_{44}\ne 0$&4
\\
${\bf Y}_4^{k}$, $2\mid k$&$a_{32}a_{31}\ne 0$,  $a_{33}=0$, ${CndYA}_3\ne 0$, ${CndYm_6}_j\ne 0$, $a_{44}\ne 0$&4
\\
\hline
\end{tabular}

$
\begin{array}{rcl}
\Delta_j^k&=&(a_{32}^2-4a_{31}a_{33})\xi^{2j}+2(a_{32}^2-2a_{31}a_{33})\xi^j+a_{32}^2-4a_{31}a_{33}\\
\Omega_{j,j'}^k&=&a_{31}a_{33}(1+\xi^j+\xi^{j'})^2-a_{32}^2(\xi^j+\xi^{j'}+\xi^{j+j'})\\
CndUm_8&=&a_{55}(a_{31}a_{55}-a_{54}a_{32})+a_{77}a_{32}^2\\
{CndVm_5}_{j,j'}&=&
a_{32}^2a_{44}\beta_{j,j'} +a_{32}a_{33}(2a_{33}a_{42}-a_{32}a_{43})\alpha_{j,j'}+a_{41}a_{33}^3,\, \mbox{\rm with}\\
&& \alpha_{j,j'} =\frac{\xi^j+\xi^{j'}+\xi^{j+j'}}{(1+\xi^j+\xi^{j'})^2},\ \beta_{j,j'}=\frac{\xi^{j+3j'}+\xi^{3j+j'}+\xi^{3j'}+2\xi^{2j+j'}+2\xi^{j+2j'}+\xi^{3j}+2\xi^{j+j'}+\xi^{j}+\xi^{j'} }{(1+\xi^{j}+\xi^{j'})^4}. 
\\
CndWA_2&=&(a_{41}a_{33}+a_{31}a_{43})a_{32}^3-2a_{31}(a_{42}a_{33}+2a_{31}a_{44})a_{32}^2+8a_{31}^3a_{33}a_{44}\\
{Cnd\widetilde{W}m}_8&=&a_{31}a_{55}-a_{42}a_{44}\\
{CndYA}_3&=&a_{43}^2-4a_{31}a_{55}\\
{CndYm_6}_j&=& (a_{31}a_{44}^2+a_{55}a_{32}^2-a_{43}a_{44}a_{32})\xi^{4j}+
a_{44}(2a_{31}a_{44}-a_{32}a_{43})\xi^{2j}+a_{31}a_{44}^2\\
(*) &&\mbox{$j<j'$; $(j,j')\ne (p,2p)$ when $k=3p$}\\
(**) &&\mbox{If $2\mid k$ then $j\ne k/2$}\\
(***) &&\mbox{\rm If $k=3p$, then $j\ne p,2p$}
\end{array}
$
}
\label{tab:ClassifPar}
\end{center}
\end{table}

\begin{table}[htp]
\begin{center}
\caption{Topological invariants of germs in strata in Table \ref{tab:ClassifPar}.}
\label{tab:InvariantsParabolicBranch}
\footnotesize{
\begin{tabular}{lcccc}
\hline
	Name &	$C$ &	$T $&	$\mu(\mathcal D)$ &	$r(\mathcal D)$ 
	\\
\hline
$ {\bf U}^{k}_{3}$, $3\nmid k$&$2k-2$& $\frac{(k-1)(k-2)}{3}$&$4(k-1)(k-2)+1$	&$2k-2$\\		    
${\bf U}^{k}_{3}$, $3\mid k$ & $2k-2$& $\frac{(k-1)(k-2)+1}{3}$&$4(k-1)(k-2)+3$&$2k-2$\\
${\bf U}^{k}_{4}$, $3\mid k$	
&	$2k-2$		
&	$\frac{(k-1)(k-2)+4}{3}$
&	$4(k-1)(k-2)+9$
&	$2k-2$
\\

${\bf V}^{k,j,j'}_{4},3\nmid k$	
&	$2k-2$
&	$\frac{(k-1)(k-2)+3}{3}$
&	$4(k-1)(k-2)+7$
&	$2k-2$
\\
${\bf V}^{k,j,j'}_{4},3\mid k$	
&	$2k-2$
&	$\frac{(k-1)(k-2)+4}{3}$
&	$4(k-1)(k-2)+9$
&	$2k-2$

\\
${\bf W}^{k,j}_{4},3\nmid k$	
&	$2k-2$
&	$\frac{(k-1)(k-2)}{3}$
&	$4(k-1)(k-2)+3$
&	$2k-4$
\\
${\bf W}^{k,j}_{4},3\mid k$
&	$2k-2$
&	$\frac{(k-1)(k-2)+1}{3}$
&	$4(k-1)(k-2)+5$
&	$2k-4$

\\
${\bf W}_4^{3p,p}$, $k=3p$
&	$2k-2$		
&	$\frac{(k-1)(k-2)+4}{3}$
&	$4(k-1)(k-2)+11$
&	$2k-4$
\\
${\bf X}^{k}_{4}$, $2\nmid k,3\nmid k $&	$3k-3$&	$\frac{(k-1)(k-2)}{3}$	&$5(k-1)(k-2)+1$	& $2k-2$\\   
${\bf X}^{k}_{4}$, $2\mid k,3\nmid k$ & $3k-3$&	$\frac{(k-1)(k-2)}{3}$	&$5(k-1)(k-2)+2$	&	$2k-3$\\	    
${\bf X}^{k}_{4}$, $2\nmid k,3\mid k$ & $3k-3$&$\frac{(k-1)(k-2)+1}{3}$&$5(k-1)(k-2)+3$	&$2k-2$\\			    
${\bf X}^{k}_{4}$, $6\mid k $&   $3k-3$&$\frac{(k-1)(k-2)+1}{3}$&$5(k-1)(k-2)+4$&$2k-3$\\

${\bf Y}^{k}_{4}$, $2\nmid k$&$2k-2$&$\frac{(k-1)(k-2)}{2}$&$5(k-1)(k-2)+1$&$2k-2$\\
${\bf Y}^{k}_{4}$, $2\mid k$&$2k-2$&$\frac{k(k-2)}{2}$&$5(k-1)(k-2)+3(k-1)$&$2k-2$\\
\hline
\end{tabular}
}
\end{center} 
\end{table}

 The functions defining the double point branches $\mathcal D_j$ have the following initial terms:
$$
\lambda_j(x,y)=a_{31}x^2+a_{32}\vartheta_{2j} xy+a_{33}\vartheta_{3j} y^2+O(3).
$$
 Consequently, the branches $\mathcal D_j$ are singular. A branch $\mathcal D_j$ has an $A_1$-singularity unless the discriminant $\Delta_j^k$ of the quadratic part of $\lambda_j$ vanishes. 
 We have 
\begin{equation}\label{eq:Delta_j}
\Delta_j^k=(a_{32}^2-4a_{31}a_{33})\xi^{2j}+2(a_{32}^2-2a_{31}a_{33})\xi^j+a_{32}^2-4a_{31}a_{33}.
\end{equation}

An $A_1$-singularity is a transverse intersection of two regular curves.  
Two branches $\mathcal D_j$ and $\mathcal D_{j'}$ with an $A_1$-singularity 
may have one or both of their components being tangential 
 (that is, the tangent cones of the two branches have a non-trivial intersection). 
 Taking the resultant of $j^2\lambda_j$ and $j^2\lambda_{j'}$ with respect to one of the variables,  we find that this 
 happens if, and only if, $a_{31}a_{33}=0$ or 
\begin{equation}\label{eq:Omega}
\Omega^k_{j,j'}=
a_{31}a_{33}(1+\xi^j+\xi^{j'})^2
-a_{32}^2(\xi^j+\xi^{j'}+\xi^{j+j'})=0.
\end{equation}

We have $1+\xi^j+\xi^{j'}=0$ or $\xi^j+\xi^{j'}+\xi^{j+j'}=0$ 
if, and only if, $k=3p$ and $j,j'\in \{p,2p\}$.
Therefore, if $3\nmid k$ or if $k=3p$ and $j,j'\notin \{p,2p\}$, 
$V(\Omega^k_{j,j'})$ is a codimension 1 algebraic variety in $J^l(2,1)$, for $l\ge 3$. 
For such pairs, 
 we set 
\begin{equation}\label{coeff_alpha_{j,j'}&alpha}
 \alpha_{j,j'} =\frac{\xi^j+\xi^{j'}+\xi^{j+j'}}{(1+\xi^j+\xi^{j'})^2}\quad \mbox{\rm and}\quad
\alpha=\frac{a_{32}^2}{a_{31}a_{33}}.
\end{equation}

We have the following properties of $\Delta_j^k$ and $\Omega_{j,j'}^k$ when 
\mbox{$a_{31}a_{33}\ne 0$}; the case \mbox{$a_{31}a_{33}=0$} 
is dealt with in Propositions \ref{prop:a11=a21=a22=0(4)} and \ref{prop:a11=a21=a22=0(5)}.

\begin{prop}\label{prop:PropertiesDeltaOmega}
Suppose that  $a_{31}a_{33}\ne 0$ and that $3\nmid k$ or $k=3p$ and $j,j'\notin\{p,2p\}$. Then:

{\rm (1)} $\Delta_{j}^k=\xi^{2j}\Delta_{k-j}^k$, and if $\Delta_{j}^k=0$ then 
$\Delta_{s}^k\ne 0 $ for $s\notin \{j,k-j\}$, so the solutions of $\Delta_{j}^k=0$
in the $k^{th}$-roots of unity come in pairs. 

{\rm (2)} For $\Delta_{j}^k$ to vanish requires $\alpha $ in {\rm (\ref{coeff_alpha_{j,j'}&alpha})} to belong to the real semi-line $(-\infty, 3)$.

{\rm (3)} $\alpha_{j,j'}=\alpha_{j',j}$ for all pairs $(j,j')$.
We have $\alpha_{j,j'}=\alpha_{k-j,j'-j}=\alpha_{k-j',j-j'}$ and the pairs $(j,j'), (k-j,j'-j), (k-j',j-j')$ are pairwise distinct.
Furthermore, $\alpha^k_{l,q}=\alpha^k_{j,j'}$ if, and only if, 
$(l,q)$ or $(q,l)$ is one of those $3$ pairs.

{\rm (4)} $\alpha_{j,j'}$  is real if, and only if, $j'=k-j$ or $j'=2j$. 
In that case, by {\rm (3)}, $\alpha_{k-j,2j}$ is also real.
Then,  $\alpha_{j,k-j}=\alpha_{j,2j}=\alpha_{k-j,2j}=(1+\xi^j+\xi^{k-j})^{-1}$.

{\rm (5)} 
If $\alpha$ is real then $\Omega^k_{j,j'}=0$ if, and only if,  $j'=k-j$ or $j'=2j$. Then by {\rm (3)}, we also have 
$\Omega^k_{k-j,2j}=0$.

{\rm (6)}
If $\alpha$ is real, 
then $\Omega^k_{j_0,k-j_0}=0$ implies 
$\Delta_j^k\ne 0$ for all $j$. Conversely, if $\Delta_{j_0}^k=0$, 
then $\Omega^k_{j,k-j}\ne 0$ for all $j$.

{\rm (7)} If $\Delta_j^k\Delta_{j'}^k\ne 0$ and $\Omega^k_{j,j'}\ne 0$ for all $j,j'$ with $j'\ne j$, 
then $\mathcal D_j\cdot \mathcal D_{j'}=4$. 

\end{prop}

\begin{proof} 
(1) As $\xi^{-j}=\xi^{k-j}$, factoring out $\xi^{2j}$ in \eqref{eq:Delta_j} gives $\Delta_{j}^k=\xi^{2j}\Delta_{k-j}^k$.

If $a_{32}^2-4a_{31}a_{33}=0$, then for $\Delta_j^k$ to vanish requires $a_{32}^2-2a_{31}a_{33}=0$. 
This would imply $a_{31}a_{33}=0$. Therefore, under the hypothesis of the proposition, 
we can assume that $a_{32}^2-4a_{31}a_{33}\ne 0$. Then 
$$
\Delta_j^k=0\iff \xi^{2j}+\frac{2(\alpha-2)}{\alpha-4}\xi^j+1=0
$$
with $\alpha$ as in \eqref{coeff_alpha_{j,j'}&alpha}. If $\xi^j$ is a solution of the above quadratic equation, 
then so is $\xi^{k-j}=\xi^{-j}$. Therefore, for $\alpha$ fixed, if $\Delta_j^k=0$, then 
$\Delta_s^k\ne 0$ for $s\notin \{ j,k-j \}$.

(2) When $a_{31}a_{33}\ne 0$, we can write
$
\Delta_j^k=a_{31}a_{33}((1+\xi^j)^2\alpha-4(1+\xi^j+\xi^{2j})).
$
Clearly, $\Delta_j^k\ne 0$ when $\xi^j=-1$. Thus,  
$\Delta_j^k=0$ if, and only if, $\alpha=4(1+\xi^j+\xi^{2j})/(1+\xi^j)^2$, 
which shows that if $\Delta_j^k=0$ then $\alpha$ must be real. The discriminant 
of the quadratic equation $(\xi^j+1)^2\alpha-4(\xi^{2j}+\xi^j+1)=0$ in $\xi^j$ is $4(\alpha-3)$, 
so $\alpha <3$ as the solutions $\xi^j$ are not real.

(3) Clearly as $\alpha_{j,j'}=\frac{\xi^j+\xi^{j'}+\xi^{j+j'}}{(1+\xi^j+\xi^{j'})^2}$, 
we have $\alpha_{j,j'}=\alpha_{j',j}$. Factoring our 
$\xi^{2j}$ (resp. $\xi^{2j'}$) from the numerator and denominator gives
$\alpha_{k-j,j'-j}=\alpha_{k-j,j'-j}=\alpha_{k-j',j-j'}$.

We now seek pairs $(l,q)$ for which $\alpha_{l,q}=\alpha_{j,j'}$. 
We know from the above that we have at least 3 such pairs. To show that these are the only ones, 
we write $\alpha_{j,j'}=c+id$ (so $(c,d)\ne (0,0)$) and represent points on the unit circle, with $-1$ removed,  in the form $z=\frac{1-t^2}{1+t^2}+i\frac{2t}{1+t^2}$ and 
$w=\frac{1-s^2}{1+s^2}+i\frac{2s}{1+s^2}$, with $t,s\in \mathbb R$. We set
$$
\alpha_{z,w}=\frac{z+w+zw}{(1+z+w)^2}.
$$ 
The real and imaginary parts of 
$\alpha_{z,w}-(c+id)$ vanish if, and only if, $P_{(c,d)}(s,t)=Q_{(c,d)}(s,t)=0$, where
$$
{
\small
\begin{array}{rcl}
P_{(c,d)}(s,t)&=&(cs^4+4ds^3+s^4-6cs^2-4ds+2s^2+c+1)t^4+\\
&&4(s^2+1)(ds^2-2cs-d+s)t^3-2(3cs^4-s^4+10cs^2+8ds-c+1)t^2-\\
&&4(s^2+1)(ds^2+2cs+3d-s)t+\\
&&cs^4-4ds^3+s^4+2cs^2-12ds-2s^2+9c-3,\\
Q_{(c,d)}(s,t)&=&-(ds^4-4cs^3-6ds^2+4cs+d)t^4+4(s^2+1)(cs^2+2ds-c+1)t^3+\\
&&2(3ds^4+10ds^2+2s^3-8cs-d+2s)t^2-\\
&&4(s^2+1)(cs^2-2ds+3c-1)t-ds^4-4cs^3-2ds^2+4s^3-12cs-9d+4s.
\end{array}
}
$$

Observe that $P_{(c,d)}$ and $Q_{(c,d)}$ are symmetric polynomials. 
Their resultant with respect to $t$ vanishes if, and only if,
$s^2=3$ or $R_{(c,d)}(s)=0$, with
$$
{\small 
\begin{array}{rcl}
R_{(c,d)}(s)&=&(c^2+d^2-1)^2s^6-(9c^4+18c^2d^2+9d^4+2c^2+2d^2+8c-3)s^4+\\
&&(27c^4+54c^2d^2+27d^4+18c^2+18d^2-16c+3)s^2-\\
&&27c^4-54c^2d^2-27d^4+18c^2+18d^2-8c+1.
\end{array}
}
$$

We have $s^2=3$ if, and only if, $k=3p$ and  $w=\xi^p$ or $w=\xi^{2p}$, and this is excluded from the hypotheses. 
The component $R_{(c,d)}$ of the resultant is a cubic polynomial in $s^2$ provided $c^2+d^2-1\ne 0$, i.e., $|\alpha_{j,j'}|\ne 1$. Suppose that $|\alpha_{j,j'}|\ne 1$. 
Then the discriminant of  $R_{(c,d)}$  vanishes if, and only if, $d=0$ (i.e., $\alpha_{j,j'}$ is real and this is treated in (4) below) or 
$$
\delta_R=27c^4+54c^2d^2+27d^4-18c^2-18d^2+8c-1=0.
$$
For $(c,d)$ in the interior region bounded by the curve $\delta_R=0$ (see Figure \ref{fig:Part}), $R_{(c,d)}$ has a unique solution 
in $s^2$. As we know that there are at least three distinct solution pairs to the problem, it follows that $(c,d)$ must be in the exterior region $(R^+)$ bounded by the curve $\delta_R=0$.

Observe that $s=0$ is a root of $R_{(c,d)}$ if, and only if,  $\delta_R=0$. Therefore, the roots of $R_{(c,d)}$ 
in $s^2$ do not change sign in $(R^+)$. Choosing any point in that region, we find that they are all positive. 
It follows that $R_{(c,d)}$ has six roots $\pm s_1,\pm s_2,\pm s_3$. These correspond to six points on the unit circle 
$w_1,w_2,w_3$ and $\overline{w}_1,\overline{w}_2,\overline{w}_3.$

For each root of $R_{(c,d)}$ we show, by considering the subresultant (see for example 
\cite{Kahoui}) of $P_{(c,d)}$ and $Q_{(c,d)}$ that 
$P_{(c,d)}(\pm s_i,t)$ and $Q_{(c,d)}(\pm s_i,t)$ have only one common root. 
As $P_{(c,d)}$ and $Q_{(c,d)}$  are symmetric polynomials, that common  root is a root of $R_{(c,d)}$. 
Interchanging $w_i$ with $\overline{w}_i$ if necessary, we can set $w_1=\xi^j$, $w_1=\xi^{k-j}$, $w_1=\xi^{k-j'}$,
Then the solutions of $\alpha_{z,w}-c-id=0$ are exactly, up to permutation of $z$ and $w$, $(\xi^j,\xi^{j'})$, $(\xi^{k-j},\xi^{j'-j})$, $(\xi^{k-j'},\xi^{j-j'})$.

We turn now to the case when $|\alpha_{j,j'}|=1$, i.e., $c^2+d^2-1=0$. This occurs if, and only if, $2\mid k$ and 
$j'=j+\frac{k}{2}$ or $j'=\frac{k}{2}$. Suppose that $2\mid k$ and 
$j'=j+\frac{k}{2}$. Then $\alpha_{j,j+\frac{k}{2}}=-\xi^{2j}$, and  $\alpha_{l,l+\frac{k}{2}}= \alpha_{j,j+\frac{k}{2}}$ if,
and only if, $l=j$ or $l=j+\frac{k}{2}$. In both cases, we get only the pair $(j,j+\frac{k}{2})$. 
Now $\alpha_{\frac{k}{2},l}=-\xi^{-2l}$, so $\alpha_{\frac{k}{2},l}=\alpha_{j,j+\frac{k}{2}}$
if, and only if, $l=k-j$ or $l= \frac{k}{2}-j$. 
This shows that 
  $\alpha_{j,j+\frac{k}{2}}=\alpha_{k-j,\frac{k}{2}}=\alpha_{\frac{k}{2}-j,\frac{k}{2}}$ and 
  the equality $\alpha_{j,j+\frac{k}{2}}= \alpha_{l,q}$ holds only for these three pairs.
\begin{figure}
\begin{center}
\includegraphics[scale=0.4]{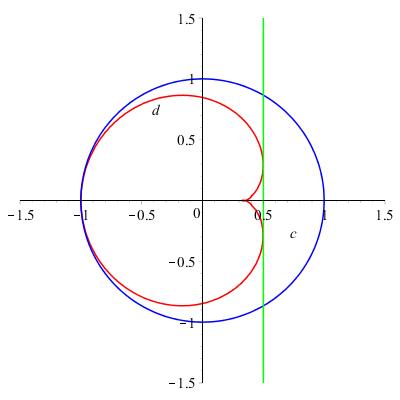}
\end{center}
\caption{The red curve is the discriminant of $R_{(c,d)}$, the blue curve is the unit circle and the green line are the values of $\alpha_{p,w}$ and $\alpha_{2p,w}$ when $k=3p$.}
\label{fig:Part}
\end{figure}

(4) We can write  $\alpha_{j,j'}$ in the form
$$
\alpha_{j,j'}=\frac{6(1+\Re(\xi^j)+\Re(\xi^{j'})+\Re(\xi^{j'-j}))+\xi^{j-2j'}+\xi^{j+j'} + \xi^{j'-2j}}{|1+\xi^j+\xi^{j'}|^2}.
$$

Therefore,  $\alpha_{j,j'}$ is real if, and only if, $\zeta=\xi^{j-2j'}+\xi^{j+j'} + \xi^{j'-2j} $ is real. 
Setting $\theta=j-2j'$ and $\phi=j'-2j$, we get $\Im(\zeta)=\sin(\frac{2\pi \theta}{k})+\sin(\frac{2\pi \phi}{k})-\sin(\frac{2\pi(\theta+\phi)}{k})$. 
Then
$$
\begin{array}{rcl}
\Im(\zeta)=0&\iff& \sin(\frac{2\pi \theta}{k})+\sin(\frac{2\pi \phi}{k})=\sin(\frac{2\pi(\theta+\phi)}{k})\\
&\iff&2 \sin(\frac{\pi (\theta+\phi)}{k})\cos(\frac{\pi (\theta-\phi)}{k})=2\sin(\frac{\pi(\theta+\phi)}{k})\cos(\frac{\pi(\theta+\phi)}{k}).
\end{array}
$$
Now, $\sin(\frac{\pi (\theta+\phi)}{k})=0$ when $j'=k-j$ and 
$\cos(\frac{\pi (\theta-\phi)}{k})=\cos(\frac{\pi(\theta+\phi)}{k})$
when  $j'=2j$ or $j=2j'$. Clearly, $\alpha_{j,k-j}=1/(1+\xi^j+\xi^{k-j})=\alpha_{j,2j}.$

(5) As $a_{31}a_{33}\ne 0$, we can write 
$\Omega_{j,j'}^k=a_{31}a_{33}(1-\alpha_{j,j'}\alpha)$, so $\Omega_{j,j'}^k=0$ if, and only if, $\alpha_{j,j'}=1/\alpha$ 
and the statement follows by (4).
 
(6) 
When $\alpha$ is real, by (4), we have $\Omega_{j_0,k-j_0}=0$ when $\xi^j_0+\xi^{k-j_0}+1=\alpha$.
It follows that $\xi^j_0=\eta_1$ with 
$
\eta_1=\frac{\alpha-1}{2}\pm i\sqrt{1-\left(\frac{\alpha-1}{2}\right)^2}\, .
$
Suppose that there exist a  $j$ for which $\Delta_{j}^k=0$, 
equivalently,
$\xi^{2j}+\frac{2(\alpha-2)}{(\alpha-4)}\xi^j+1=0$.
Then $\xi^j+\xi^{k-j}=-\frac{2(\alpha-2)}{\alpha-4}$ which gives $\xi_j=\eta_2$ with 
$
\eta_2=-\frac{\alpha-2}{\alpha-4}\pm i\sqrt{1-\left(\frac{\alpha-2}{\alpha-4}\right)^2}\, .
$
We have $\eta_1=\eta_2$ if, and only if, $\alpha=0$ or $3$. Then $k=3p$ and $j_0=p$ or $2p$, which is excluded from our hypotheses, so $\eta_1\ne \eta_2$.

The complex number $\eta_1$ (resp. $\eta_2$) is a $k^{th}$-root of unity if, and only if, 
$\alpha$ is a root of the polynomial $P_1(\alpha)$ (resp. $P_2(\alpha)$) of degree $k$ obtained by taking the numerator of $\Re({\eta_1}^k)-1$ (resp. $\Re({\eta_2}^k)-1$).
But if $P_1$ and $P_2$ have one common root, all the other roots must also be common. 
(This follows from the fact that the map $\cos\theta=-(\alpha-2)/(\alpha-4)$ is a bijection for $\alpha\in (-\infty,3]$ and 
$\cos\theta=(\alpha-1)/2$ is also a bijection for $\alpha\in (-1,3]$.)
As $P_1$ and $P_2$ are distinct polynomials, 
it follows that they have no common roots. Consequently, $\Delta_{j}^k\ne 0$ for all $j$. 

The argument for showing that $\Omega^k_{j,k-j}\ne 0$ 
when $\Delta_{j_0}^k=0$ is the same as above.

(7) We have, with the hypothesis, 
$\mathcal D_j\cdot \mathcal D_{j'}=\dim_{\C}\mathcal O_2/\left<xy,x^2+y^2\right>=4$. 
\end{proof}

In view of Proposition \ref{prop:PropertiesDeltaOmega} (7), we give in the rest of this section 
$\mathcal D_j\cdot \mathcal D_{j'}$ for the exceptional branches only, i.e, when one or both branches 
have a singularity more degenerate than $A_1$ or one or both of their components are tangential (which is 
equivalent to $\mathcal D_j\cdot \mathcal D_{j'}>4$).
We start with the case when all  the branches have an $A_1$-singularity.

\begin{prop}\label{prop:a11=a21=a22=0(0)}
Suppose that $a_{11}=a_{21}=a_{22}=0$, $a_{31}a_{33}\ne 0$  and $\Delta_j^k\ne 0$ for all $j$. Any $k$-folding map-germ $F_k$  satisfying 
the additional conditions in {\rm (a)} or {\rm (b)} is finitely $\mathcal A$-determined and is topologically equivalent to one 
of the following
germs:
\[
{\bf U}^k_{l}\colon(x,y)\mapsto(x,y^k,x^2y+2xy^2+y^3+y^{3l-5}),\, l=3,4.\\
\]
Every branch of the double point curve has an $A_1$-singularity and the invariants associated to the germs in these strata are as in {\rm Table \ref{tab:InvariantsParabolicBranch}.}

\begin{itemize}
\item [{\rm (a)} ]If $3\nmid k$ and $\Omega^k_{j,j'}\ne 0$ for all distinct pairs, then 
	$F_k$ is topologically equivalent to 
	${\bf U}^k_{3} $.

\item[{\rm (b)}] Suppose that $k=3p$ and $\Omega^k_{j,j'}\ne 0$ for all distinct pairs with $j,j'\ne \{p,2p\}$. 
\begin{itemize}
	\item[{\rm (b1)}] If $a_{44}\ne 0$, then ${\mathcal D}_p\cdot\mathcal D_{2p}=5$ and $F_k$ is topologically equivalent to 
	${\bf U}^k_{3} $.
	\item[{\rm (b2)}]  If $a_{44}=0$ and $CndUm_8\ne 0$, then ${\mathcal D}_p\cdot\mathcal D_{2p}=8$ and $F_k$ is topologically equivalent to 
	${\bf U}^k_{4} $.
\end{itemize}
\end{itemize}
\end{prop}

\begin{proof} Each branch of the double point curve consists of a transverse 
intersection of two regular curves. In (a) all of these curves are pairwise transverse. 

In (b), $j^2\lambda_s(x,y)=x(a_{31}x+a_{32}\vartheta_{2s}y)$ for $s=p,2p$. 
Observe that $a_{32}\ne 0$ as we supposed $\Delta^k_s\ne 0$, so 
the branches $\mathcal D_p$ and $\mathcal D_{2p}$ have one common line $x=0$ in their 
tangent cone and their associated curves tangent to this line are parametrised by $t\mapsto (\gamma_s(t),t)$ with  
$\gamma_s(t)=-(a_{44}/({\vartheta_{2s} a_{32}}))t^2+O(3)$, 
so they have order of contact 2 when $a_{44}\ne 0$. 
Then $\mathcal D_p\cdot \mathcal D_{2p}=\dim_{\C}\mathcal O_2/\left<xy,x^2+y^3\right>=5$. 

When $a_{44}=0$, the two curves are parametrised by $t\mapsto(\gamma_s(t),t)$ with
$$
\gamma_s(t) = -\frac{a_{55}}{a_{32}}t^3+
\frac{2}{\vartheta_{2s} a_{32}^3}(a_{55}(a_{31}a_{55}-a_{54}a_{32})+a_{77}a_{32}^2)t^5
+O(6),$$
and have contact order 5 when $(a_{55}(a_{31}a_{55}-a_{54}a_{32})+a_{77}a_{32}^2\ne 0$, i.e., $CndUm_8\ne 0$.
Then 
$\mathcal D_p\cdot \mathcal D_{2p}=\dim_{\C}\mathcal O_2/\left<xy,x^2+y^6\right>=8.$

For the topological normal forms, 
we choose $a_{31},a_{32},a_{33}$ real with $\alpha> 3$.
By Proposition \ref{prop:PropertiesDeltaOmega} (2), all $\Delta_j^k$ are non-zero. 
Also, by  Proposition \ref{prop:PropertiesDeltaOmega} (4) and (5) 
$\Omega_{j,k-j}=0$ when $\xi^j+\xi^{k-j}+1=\alpha$, which cannot happen when $\alpha >3$, 
so the $\Omega_{j,j'}^k$ in (a) and (b) are non-zero. 
As the strata are connected sets, we can choose $a_{31}=a_{32}=1,a_{32}=2$
so that the conditions on $\Delta_j^k$ and $\Omega_{j,j'}^k$ are satisfied.

The number of triple points is calculated as in  Proposition \ref{prop:a_{11}=a_{22}=a_{33}=0a_{21}ne0,4ndivk}. 
If $k=3p$, the germ ${\bf U}^k_{3}$ (resp. ${\bf U}^k_{4}$) has $T_{p,2p}=3$ (resp. $T_{p,2p}=6$), 
while all other indices $j$ and $j'$ give $T_{j,j'}=2$.
\end{proof}

\begin{prop}\label{prop:a11=a21=a22=0(1)}
Suppose that $a_{11}=a_{21}=a_{22}=0$, $a_{31}a_{33}\ne 0$, $a_{32}\ne 0$ and 
$\Omega^k_{j_0,j_0'}=0$
for some pair $(j_0,j'_0)$, with $j_0,j'_0\notin\{p,2p\}$ when $k=3p$. 
Then
${\mathcal D}_{s}\cdot\mathcal D_{q}=5$ for $(s,q)\in \{(j_0,j'_0), (k-j_0,j'_0-j_0), (k-j'_0,j_0-j'_0)\}$
if, and only if, 
${CndVm_5}_{j_0,j_0'}\ne 0$. 
For $k=3p$,  we have 
${\mathcal D}_p\cdot \mathcal D_{2p}=5$ if, and only if, $a_{44}\ne 0$. Any $k$-folding map-germ $F_k$ satisfying the above conditions is finitely $\mathcal A$-determined and is topologically equivalent to 
$$
{\bf V}^{k,j_0,j_0'}_{4}\colon(x,y)\mapsto(x,y^k,x^2y+xy^2+\alpha_{j_0,j'_0}y^3+\big(1-\frac{\beta_{j_0,j'_0}}{\alpha^3_{j_0,j'_0}}\big)x^3y+y^4).
$$

The invariants associated to germs in these strata are as in {\rm Table \ref{tab:InvariantsParabolicBranch}.}
\end{prop}

\begin{proof} When $a_{32}=0$ and $k=3p$,  we have $\Delta_p^k=\Delta_{2p}^k=0$. This case is dealt with 
in Proposition \ref{prop:a11=a21=a22=0(2)}. With the hypothesis and Proposition \ref{prop:PropertiesDeltaOmega}(5),
we have $\Omega_{j_0,j'_0}=\Omega_{k-j_0,j'_0-j_0}=\Omega_{k-j'_0,j_0-j'_0}=0.$

We need to consider the order of contact between the two tangential components 
of the pairs  $(\mathcal D_{s},\mathcal D_{q})$ with $(s,q)\in \{(j_0,j'_0), (k-j_0,j'_0-j_0), (k-j'_0,j_0-j'_0)\}$.
Using the fact that if two quadratic equations $x^2+a_ix+b_i=0$, $i=1,2$, have one root in common, then the 
root is given by $x=-(b_1-b_2)/(a_1-a_2)$, we can get the initial terms of parametrisations of the tangential components of the above pairs. 
These are given by 
$t\mapsto (t,\gamma_l(t))$, $l=s,q$, with 
$
\gamma_l(t)=
-\frac {a_{33}}{a_{32}}\big(1+\xi^{s}+\xi^{q}\big)t+\lambda_lt^2+O(3), l=s,q.
$
A calculation shows that $\lambda_s-\lambda_q=0$ if, and only if, 
$$
{CndVm_5}_{s,q}=a_{32}^2a_{44}\beta_{s,q} +a_{32}a_{33}(2a_{33}a_{42}-a_{32}a_{43})\alpha_{s,q}+a_{41}a_{33}^3=0,
$$ with
\begin{equation}\label{eq:betajj'}
\beta_{j,j'}=\frac{\xi^{j+3j'}+\xi^{3j+j'}+\xi^{3j'}+2\xi^{2j+j'}+2\xi^{j+2j'}+\xi^{3j}+2\xi^{j+j'}+\xi^{j}+\xi^{j'} }{(1+\xi^{j}+\xi^{j'})^4}. 
\end{equation}

Observe that $\beta_{j,j'}=\beta_{k-j,j'-j}=\beta_{k-j',j-j'}$, 
so ${CndVm_5}_{s,q}$ has the same value for 
$(s,q)\in \{ (j_0,j'_0), (k-j_0,j'_0-j_0), (k-j'_0,j_0-j'_0)\}.$
It follows that  ${\mathcal D}_{s}\cdot\mathcal D_{q}=\dim_{\C}\mathcal O_2/\left<xy,x^2+y^3\right>=5$ 
if, and only if, ${CndVm_5}_{j_0,j_0'}\ne 0$.

For the topological model, we take $a_{31}=a_{32}=1$ and $a_{33}=\alpha_{j_0,j'_0}$
so $\Omega_{j_0,j'_0}=0$. We also set $a_{44}=1$, $a_{42}=a_{43}=0$, 
then ${CndVm_5}_{j_0,j'_0}=\beta_{j_0,j'_0} +a_{41}\alpha_{j_0,j'_0}^3$. We set $a_{41}=1-\beta_{j_0,j'_0}/\alpha_{j_0,j'_0}^3$, so that 
${CndVm_5}_{j_0,j'_0}\ne 0$. For calculating the triple points, we find that $T_{j,j'}=3$ for all $j$ and $j'$ with $j\ne j'$.
\end{proof}

We deal now with the case when one pair of branches of the double point curve has a singularity more degenerate than $A_1$ (see Proposition \ref{prop:PropertiesDeltaOmega}(1)).

\begin{rem}\label{rem:Unequivalent2}
{\rm 
For $k=3p$,  the singularities ${\bf U}^k_4$ and ${\bf V}^{k,j,j'}_4$ have the same invariants, but  the intersection numbers between their double point branches shows that they are not topologically equivalent.
}
\end{rem}

\begin{prop}\label{prop:a11=a21=a22=0(2)}
Suppose that $a_{11}=a_{21}=a_{22}=0$, $a_{31}a_{33}\ne 0$. The strata below are of codimension $4$, and the invariants associated to germs in the strata are as in {\rm Table \ref{tab:InvariantsParabolicBranch}.}

{\rm (1)} Suppose that $\Delta_{j_0}^k=\Delta_{k-j_0}^k=0$, with $j_0\notin \{ p,2p\}$ when $k=3p$ {\rm (}so $a_{32}\ne 0{\rm )}$. 
The branches 
$\mathcal D_{j_0}$ and $\mathcal D_{k-j_0}$ have an $A_2$-singularity if, and only if,
$CndWA_2\ne 0$. We have
$
{\mathcal D}_{j_0}\cdot\mathcal D_{j}={\mathcal D}_{k-j_0}\cdot\mathcal D_{j}=4$ for all distinct pairs.
When $k=3p$ and  $a_{44}\ne 0$, we have  
${\mathcal D}_p\cdot\mathcal D_{2p}=5$. 
Any $k$-folding map-germ $F_k$ satisfying the above conditions is $\mathcal A$-finitely determined and is topologically equivalent to  
$$
{\bf W}^{k,j_0}_{4}\colon(x,y)\mapsto(x,y^k,x^2y+xy^2+\frac{(\xi^{j_0}+1)^2}{4(\xi^{2j_0}+\xi^{j_0}+1)}y^3+4x^2y^2+y^4).
$$

{\rm (2)} If $k=3p$ and $j_0=p$, then 
${\mathcal D}_p$ and $\mathcal D_{2p}$ have an $A_2$-singularity if, and only if, $a_{44}\ne 0$.
We have ${\mathcal D}_p\cdot\mathcal D_{2p}=8$ 
if, and only if, ${Cnd\widetilde{W}m}_8\ne 0$, and  
${\mathcal D}_s\cdot\mathcal D_j=5$ for $s=p,2p$ and $j\ne s$.
Any $k$-folding map-germ $F_k$ satisfying the above conditions is $\mathcal A$-finitely determined and is topologically equivalent to  
$$
{\bf W}^{3p,p}_{4}\colon(x,y)\mapsto(x,y^k,x^2y+y^3+y^4+y^5).
$$
\end{prop}

\begin{proof} The condition for an $A_2$-singularity of $\mathcal D_{j_0}$ and $\mathcal D_{k-j_0}$
follows by analysing the 3-jets of, respectively, $\lambda_{j_0}$ and $\lambda_{k-j_0}$. 
Observe that when $2\mid k$,  $\xi^{j_0}=-1$ is a solution of 
$\Delta_{j_0}^k=\Delta_{k-j_0}^k=0$ if, and only if, $a_{31}a_{33}=0$. As we are assuming 
$a_{31}a_{33}\ne 0$, we have $j_0\ne k/2$ in (1) when $2\mid k$.

With the hypothesis in (1), 
${\mathcal D}_{j_0}\cdot\mathcal D_{j}={\mathcal D}_{k-j_0}\cdot\mathcal D_{j}=\dim_{\mathbb C}\mathcal O_2/\left<x^2+y^3,x^2-y^2\right >=4$ for any distinct pairs.

For (2), we have 
${\mathcal D}_p\cdot\mathcal D_{2p}=\dim_{\mathbb C}\mathcal O_2/\left<x^2+y^3,y^4\right >=8$ 
when ${Cnd\widetilde{W}m}_8\ne 0$.

For the topological normal forms, we use the fact that $\Delta_{j_0}^k=0$ if and only if 
$\alpha={a_{32}^2}/({a_{31}a_{33}})=4(\xi^{2j_0}+\xi^{j_0}+1)/{(\xi^{j_0}+1)^2}$ (see the proof of Proposition \ref{prop:PropertiesDeltaOmega}) and 
set $a_{31}=a_{32}=1$, so $a_{33}={(\xi^{j_0}+1)^2}/{4(\xi^{2j_0}+\xi^{j_0}+1)}$.
We have $T_{j,j'}=2$ for all $j$ and $j'$ with $j\ne j'$.
\end{proof}

\begin{prop}\label{prop:a11=a21=a22=0(4)}
Suppose that $a_{11}=a_{21}=a_{22}=0$, $a_{31}=0$, $a_{33}a_{32}\ne 0$.
All the branches of the double point curve share only one line in their tangent cones 
except  when $k=3p$ where  the branches $\mathcal D_p$ and $\mathcal D_{2p}$ have the same tangent cone.
Any $k$-folding map-germ $F_k$ satisfying the additional conditions in {\rm (a)}, {\rm (b)}, {\rm (c)} or {\rm (d)} 
is finitely $\mathcal A$-determined and is topologically equivalent to
$$
{\bf X}_4^k\colon (x,y)\mapsto (x,y^k,xy^2+y^3+x^3y+y^4).
$$

The stratum is of codimension $4$, and
the invariants associated to the germs in the strata above are as in {\rm Table \ref{tab:InvariantsParabolicBranch}.}
\begin{itemize}
\item[{\rm (a)}] If $2\nmid k$, $3\nmid k$, all of the branches $\mathcal D_{j}$
of the double point curve have an $A_1$-singularity and
${\mathcal D}_j\cdot\mathcal D_{j'}=5$, $j\ne j'$,  when $a_{41}\ne 0$. 
\item[{\rm (b)}] If $k=2p$, $3\nmid k$, the branches $\mathcal D_{j}$, $j\ne p$, behave as in {\rm (a)}. The branch 
$\mathcal D_p$ has an $A_2$-singularity if, and only if, $a_{41}\ne 0$.
We also have ${\mathcal D}_j\cdot\mathcal D_{p}=5$, $j\ne p$.
\item[{\rm (c)}]  If $k=3p$, $2\nmid k$, the branches $\mathcal D_{j}$ behave as in {\rm (a)} except for 
$\mathcal D_{p}$ and $\mathcal D_{2p}$ where ${\mathcal D}_p\cdot\mathcal D_{2p}=6$ when $a_{41}a_{44}\ne 0$. 
\item[{\rm (d)}]  If $k=6p$, when $a_{41}a_{44}\ne 0$, the exceptional branches of the double point curve 
behave as in {\rm (b)} and {\rm (c)}, the remaining branches behave as in {\rm (a)}. 
\end{itemize}
\end{prop}

\begin{proof}
We have $j^2\lambda_j(x,y)=
y(\vartheta_{2j}a_{32}x+\vartheta_{3j}a_{33}y)$. 
For (a), ${\vartheta_{2j}}\ne 0$ for all $j$ and 
$\frac{\vartheta_{2j}}{\vartheta_{2j}}\ne \frac{\vartheta_{3j'}}{\vartheta_{2j'}}$ for $j\ne j'$, 
so all of the branches of the double point curve have an $A_1$-singularity and one common line 
$y=0$ in their tangent cone. The component of 
${\mathcal D}_j$ with tangent $y=0$ can be parametrised by $t\mapsto (t,\gamma_j(t))$ with  
by $\gamma_j(t)=-{a_{41}}/({(1+\xi^j)a_{32}})t^2+O(3)$. Therefore, when $a_{41}\ne 0$, 
we have ${\mathcal D}_j\cdot\mathcal D_{j'}=\dim_{\mathbb C}
\mathcal O_2/\left<xy,y^2+x^3\right>=5$. 

The remaining parts of the proof follow similarly and are omitted.
\end{proof}

\begin{prop}\label{prop:a11=a21=a22=0(5)}
Suppose that $a_{11}=a_{21}=a_{22}=0$, $a_{33}=0$, $a_{31}a_{32}\ne 0$.
All the branches of the double point curve share only one line in their tangent cones. 
Any  $k$-folding map-germ $F_k$ satisfying the additional conditions in {\rm (a)} or {\rm (b)} 
is finitely $\mathcal A$-determined 
and is topologically equivalent to 
$$
{\bf Y}^{k}_{4}\colon(x,y)\mapsto(x,y^k,-xy^2+x^2y+y^4+y^5).
$$

The stratum is of codimension $4$, and the invariants associated to germs in the stratum 
are as in {\rm Table \ref{tab:InvariantsParabolicBranch}.}
\begin{itemize}
\item[{\rm (a)}] If $2\nmid k$, all of the branches of the double point curve have an $A_1$-singularity 
and ${\mathcal D}_j\cdot\mathcal D_{j'}=5$, $j\ne j'$,  when $a_{44}\ne 0$.

\item[{\rm (b)}]  If $k=2p$, the branch ${\mathcal D}_p$ has an $A_3$-singularity 
if, and only if, ${CndYA}_3\ne 0$.  
Then  ${\mathcal D}_p\cdot\mathcal D_{j}=6$, $j\ne p$, and 
${\mathcal D}_j\cdot\mathcal D_{p+j}=6$, $1\le j<p$, if and only if 
${CndYm_6}_j\ne 0$.
The other pairs of branches behave as in {\rm (a)}.
\end{itemize}
\end{prop}

\begin{proof}
We have $j^2\lambda_j(x,y)=x(a_{31}x+a_{32}\vartheta_{2j} y)$
 so when $2\nmid k$, all of the branches of the double point curve have an $A_1$-singularity as $\vartheta_{2j}\ne 0$ for all $j$. All the branches 
have only $x=0$ as a common line in their tangent cone. 
The components of the branches which are tangent 
to $x=0$ are parametrised by $t\mapsto (\gamma_j(t),t)$, with 
$\gamma_j(t)=-{a_{44}(1+\xi^{2j})}/{a_{32}}t^2+O(3)$.
When $a_{44}\ne 0$, ${\mathcal D}_j\cdot\mathcal D_{j'}=\dim_{\mathbb C}
\mathcal O_2/\left<xy,x^2+y^3\right>=5$. 

When $k=2p$, 
we have $\vartheta_{2p}=\vartheta_{4p}=0$, 
so considering the 4-jet of $\lambda_p$, we find that 
the branch $\mathcal D_p$ has an $A_3$-singularity if, and only if, ${CndYA}_3\ne 0$. When this is the case, 
 ${\mathcal D}_p\cdot\mathcal D_{j}=\dim_{\mathbb C}
\mathcal O_2/\left<xy,x^2+y^4\right>=6$, for $j\ne p$. 

All pairs of branches $\mathcal D_j$, $\mathcal D_{p+j}$,  $1\le j<p$
have tangential component parametrised by $t\mapsto(\gamma_s(t),t)$, 
with 
$\gamma_s(t)=-{a_{44}(1+\xi^2_{j})}/{a_{32}}t^2+\alpha_{s}t^3+O(4)$, $s=j,p+j$, with 
$\alpha_1-\alpha_2\ne 0$ if and only if 
${CndYm_6}_j\ne 0$. When this is the case, 
${\mathcal D}_j\cdot\mathcal D_{p+j}=\dim_{\mathbb C}
\mathcal O_2/\left<xy,x^2+y^4\right>=6$, for $1\le j< p$. 
The other pairs of branches behave as when $2\nmid k$.
\end{proof}

\begin{rem}
{\rm The stratification $\mathcal S_k$ is determined, for each $k\geq 4$, by the conditions in Tables \ref{tab:ClassifPr}, \ref{tab:ClassifAs} and  \ref{tab:ClassifPar} ($\mathcal S_2$ follows from the results in \cite{BruceWilk} and $\mathcal S_3$ is indicated in Table \ref{tab:k=3}). 
Each stratum corresponds to a single topological class and all the topological classes listed in Tables \ref{tab:ClassifPr}, \ref{tab:ClassifAs} and  \ref{tab:ClassifPar} 
are topologically distinct.  
Indeed, as pointed out in \mbox{Remark \ref{rem:Unequivalent}}  (resp. Remark \ref{remCompareQ43AndR44} and resp. Remark \ref{rem:Unequivalent2})  
 ${\bf Q}^k_4$ and $\widetilde {\bf Q}^k_4$ (resp. ${\bf Q}^4_3$ and ${\bf R}^4_4$ and resp. ${\bf U}^k_4$ and ${\bf V}^{k,j,j'}_4$) 
are not topologically equivalent. 
For the remaing classes, it follows by comparing the invariants in Tables \ref{TableInvariantsCase1}, \ref{tab:invAs_strata} and \ref{tab:InvariantsParabolicBranch}, 
that they are all topologically distinct. 
}
\end{rem}


\section{Generic singularities of $k$-folding map-germs}\label{sec:surfaces}

We define, by varying the plane $\pi \in {\rm Graff}(2,3)$,  the family of Whitney $k$-folds 
\mbox{$
\Omega^k:\mathbb C^3\times {\rm Graff}(2,3)\to \mathbb C^3,
$ }
given by $\Omega^k(p,\pi)=\omega^k_{\pi}(p)$, with $\omega^k_{\pi}$ as in Definition \ref{def:k-foldingComplex}. 

Given a complex surface $M$ in $\mathbb C^3$, 
we call the restriction of $\Omega^k$ to $M$ the family of $k$-folding maps on $M$ and denote it by ${\mathcal F}_k$.
We have ${\mathcal F}_k(p,\pi)=F^{\pi}_k(p)=\omega^k_{\pi}(p)$ for all $p\in M$ and $\pi \in {\rm Graff}(2,3)$. 

Recall that a property of surfaces is said to be {\it generic} if it is satisfied in a residual set of embeddings of the surfaces to $\mathbb C^3$. The image of a surface $M$ by an embedding in the residual set is then called generic, or simply that $M$ is generic. When $k$ is large, the $\mathcal A$-singularities of $F^{\pi}_k$ may have high $\mathcal A_e$-codimensions (for the cases in this paper, this means high modality). However, they do occur on generic surfaces. To make sense of this, we follow a similar approach to that in \cite{Bruce84} and proceed as follows. 

As we are interested here in the local singularities of $k$-folding maps, we consider the setting in Remarks \ref{rems:Fkpi}(4) at a point $p_0\in M$ and choose a suitable system of coordinates 
so that $\pi_0:y=0$ and $F_k=F^{\pi_0}_k(x,y)=(x,y^k, f(x,y))$ for $(x,y)$ in  a small enough neighbourhood $U$ of the origin. 
A plane $\pi=\overline{(d,v)}$ near $\pi_0$ is obtained by applying a translation $T_{\pi}$ followed by an orthogonal transformation $R_{\pi}\in U(3)$ to  $\pi_0$. We choose $T_{\pi}$ and $R_{\pi}$ as follows.
The translation $T_{\pi}$ takes the origin to the point of intersection of $\pi$ with the $y$-axis 
 (the point exists because $\pi_0$ is orthogonal to the $y$-axis and $\pi$ is close to $\pi_0$).
The transformation $R_{\pi}$ near the identity (and is taken to be the identity if $v$ is parallel to $v_0$) takes $v/||v||$ to $v_0=(0,1,0)$ and fixes the line through the origin orthogonal to $v_0$ and $v$. 
By varying the planes $\pi$ in a neighbourhood $V$ of $\pi_0$ in ${\rm Graff}(2,3)$, we get the family of $k$-folding maps given by ${\mathcal F}_k:U\times V\to \mathbb C^3$, with
$$
{\mathcal F}_k((x,y),\pi)=(R_{\pi}\circ T_{\pi})\circ \omega_k \circ (R_{\pi}\circ T_{\pi})^{-1}(\phi (x,y)),
$$ 
with $\phi(x,y)=(x,y,f(x,y))$. (We choose $R_{\pi}$ and $T_{\pi}$ to depend analytically on $\pi$.) 
The $\mathcal A$-type of the singularity of $F^{\pi}_k$ at a given point in $U$ is the same as that of the germ of  
$$
\tilde{\mathcal F}_k((x,y),\pi)=\tilde F_{\pi}^k(x,y)=(R_{\pi}\circ T_{\pi})^{-1}\circ {\mathcal F}_k((x,y),\pi)=\omega_k \circ (R_{\pi}\circ T_{\pi})^{-1}(\phi (x,y))
$$ 
at that point. At any point $p=(x,y)\in U$, there exist a bi-holomorphic map-germ  $K:(\mathbb C^2,0)\to (\mathbb  C^2,(x,y))$ such that 
$$
(R_{\pi}\circ T_{\pi})^{-1}(\phi (K(X,Y))=(\bar{x}+X,\bar{y}+Y,\bar{z}+g_{p}(X,Y)),
$$ 
for some germ of a holomorphic function $g_{p}$, where $(\bar{x},\bar{y},\bar{z})=(R_{\pi}\circ T_{\pi})^{-1}(\phi (x,y)).$  Composing $\tilde F_{\pi}^k$ with $K$ gives the germ $(\tilde F_{\pi}^k)_p$ of $\tilde F_{\pi}^k$ at $p$
$$
	\begin{array}{rll}
		(\tilde F_{\pi}^k)_p(X,Y)&=&(\bar{x}+X,(\bar{y}+Y)^k,\bar{z}+g_{p}(X,Y))\\
		&\sim_{\mathcal A}&(X,(\bar{y}+Y)^k,g_{p}(X,Y)).
	\end{array}
$$

Observe that $g_{p}$ depends on the choice of $R_{\pi}$, $T_{\pi}$ and of the coordinates system, 
but the $\mathcal A$-class of the resulting germs $(\tilde F_{\pi}^k)_p$ is independent of these choices.

Clearly, a necessary condition for $(\tilde F^{\pi}_k)_p$ to be singular at the origin is $\bar{y}=0$, equivalently, 
$p\in \pi$. We define the family of maps	$\Phi:U\times V\to \mathbb C\times  J^l(2,1)$ given by	
$$
\Phi(p,\pi)=(\langle v, p\rangle -d,j^lg_{p}(0,0)).
$$
The map $\Phi$ plays a similar role to the Monge-Taylor map in \cite{Bruce84}. Here we include 	the first component to capture the planes through a given point $p\in M$ which can give rise to singular $k$-folding map-germs at $p$.

We stratify $\mathbb C\times  J^l(2,1)$ by $0\times \mathcal S_k$, together with $\mathbb C\times  J^l(2,1)\setminus 0\times  J^l(2,1)$. Following standard transversality arguments (see for example \cite{Bruce84}), one can show that	for a residual set of local embeddings of $M$ in $\mathbb C^3$, the map $\Phi$ is transverse to the strata in $\mathbb C\times J^l(2,1)$. 
As the domain of the family $\Phi$ is of dimension 5,  for a generic local embedding of $M$ in $\mathbb C^3$, 
$\Phi$ intersects a stratum $0\times X$ only when $X$ has codimension $\le 4$ in $J^l(2,1)$. This means that the only singularities of $k$-folding maps that can occur on a generic surface 
are those belonging to the strata listed 
in Tables \ref{tab:k=3}, \ref{tab:ClassifPr}, \ref{tab:ClassifAs}, \ref{tab:ClassifPar}. Furthermore, $\Phi$ is transverse to these strata. Therefore, the locus of points where the $l$-jets of the
$k$-folding map-germs belong  to a stratum $X$ of codimension $3$ (resp. $4$) is a regular curve on $M$ (resp. isolated points on this curve). The above discussion also clarifies why the singularities of $F^{\pi}_k$ occur generically even though they may have high $\mathcal A_e$-codimensions: they belong to strata of low codimension.
 (The strata can be viewed as a gluing of non $\mathcal A$-equivalent orbits which depends on a finite set of moduli.)

It is worth observing that the above discussion is valid for smooth surfaces in $\mathbb R^3$.

\section{$k$-folding map-germs on surfaces in $\mathbb R^3$}\label{sec:hiddensymmetries}

We consider in this section  the geometry of $k$-folding maps-germs on smooth (i.e., regular and of class $C^{\infty}$) surfaces 
in $\mathbb R^3$.

\subsection{Robust features on surfaces}
Robust features on a surface in $\mathbb R^3$ is a terminology introduced by Ian Porteous to indicate special characteristic 
features that can be traced when the surface evolves. 
What is sought after in applications  
are robust features which are represented by curves or points on the surface as these form a ``skeletal structure'' 
of the surface (open regions bounded by robust curves are also robust features). They 
play an important role in computer vision and shape recognition (see for example \cite{Musuvathy}) as they can be  
used to distinguish two shapes  (surfaces) from each other and, in some cases, reconstruct the surface.

We consider here the {\it parabolic, ridge, sub-parabolic and flecnodal curves} on a smooth surface $M$ in $\mathbb R^3$ 
and special points on these curves (see \cite{IFRT} for references on work on these curves from a singularity theory point of view).  
These are robust features on $M$. We recall briefly what they are.

{\it The  parabolic curve} is the locus of points where the Gaussian curvature vanishes. It is captured by the contact of the surface $M$ with planes: it is the locus of points where 
the height function along a normal direction to $M$ has  an $A_{\ge 2}$-singularity. 
The parabolic curve is regular on a generic surface, and the height function has an $A_2$-singularity at its points except at isolated 
ones, called {\it Cusps of Gauss}, where it has an $A_3$-singularity.

{\it  The  ridge} is the locus of points on $M$ 
where a principal curvature is extremal along its associated lines of principal curvature. 
It is also the locus of points on $M$ which correspond to singular points on its focal set.
The ridge is captured by the contact of $M$ with spheres:  
it is the locus of points on $M$ where the distance squared function from a given point 
in $\mathbb R^3$ (the point belong to the focal set) has an $A_{\ge 3}$-singularity. 
Away from umbilic points (i.e., points where the principal curvatures coincide), the ridge is a regular curve on a generic surface and the distance squared function has an $A_3$-singularity 
at its points except at isolated points where it has an $A_4$-singularity. At umbilics, the ridge consist of one regular curve 
or a transverse intersection of three regular curves. (Umbilics and $A_4$-points are used as {\it seed points} for drawing ridges 
on a given shape, see \cite{Musuvathy}.)

The {\it sub-parabolic curve} is the locus of points on $M$ corresponding to parabolic points of its focal set. 
It is the locus of geodesic inflections of the lines of principal curvature; 
it is also the locus of points along which a principal curvature is extremal along the other lines of curvature.
It is captured by the singularities of the 2-folding map on $M$: it is the locus of points where 
some map $F_2^{\pi}$ has an $S_{\ge 2}$-singularity (or one which is adjacent to an $S_2$-singularity).
The singularities of $F_2$ and their geometric characterisation 
on a generic surface parametrised in Monge form $z=f(x,y)$ at a given non-umbilic point, with $f$ as in \eqref{Taylorf},
are as in Table \ref{tab:AlgCdFmp}.
 At umbilics, the sub-parabolic curve 
consist of one regular curve 
or a transverse intersection of three regular curves.
Observe that the ridge is also captured by the singularities of $F_2^{\pi}$ (see Table \ref{tab:AlgCdFmp}).

\begin{table}[tp]
	\begin{center}
		\caption{Geometric characterisation of the singularities of the folding-map $F_{2}$.}
		{\footnotesize
			\begin{tabular}{cp{0.7\textwidth}}
				\hline 
				Name & Algebraic conditions and geometric meaning\\
				\hline 
				{\rm Crosscap}&$a_{21}\ne 0$
				\\
				$B_1=S_1$& $a_{21}=0,\ a_{31}\ne 0, \ a_{33}\ne 0$ \\
				&General smooth point of focal set
				\\
				$B_2$& $a_{21}=0, a_{31}\ne 0, a_{33}=0, CndNA_3=4a_{31}a_{55}-a_{43}^2\ne 0$\\
				&General cusp point of focal set corresponding to a point on the ridge curve
				\\
				$B_3$& $a_{21}=0,  a_{31}\ne 0,  a_{33}=0, CndNA_3=0$,\\
				&${CndNA}_5=8a_{31}^3a_{77}-4a_{65}a_{43}a_{31}^2+ 2a_{53}a_{43}^2a_{31}-
				a_{41}^2a_{43}^3\ne0$\\	&(Cusp) point of focal set in closure of parabolic curve on symmetry set
				\\
				$S_2$& $a_{21}=0,\ a_{31}=0, \  a_{33}\ne 0,\ a_{41}\ne 0$\\
				&Parabolic smooth point of focal set corresponding to a point on the sub-parabolic curve
				\\
				$S_3$& $a_{21}=0,\ a_{31}=0, \  a_{33}\ne 0,\ a_{41}=0,\ a_{51}\ne 0$\\
				&Cusp of Gauss at smooth point of focal set
				\\
				$C_3$& $a_{21}=0,\ a_{31}=0, \  a_{33}=0,\ a_{41}\ne 0,\ a_{43}\ne 0$\\
				&Intersection point of cuspidal-edge and parabolic curve on focal set
				\\
				\hline 
			\end{tabular}
			\label{tab:AlgCdFmp}
		}
	\end{center}
\end{table}

The {\it  flecnodal curve} is the locus of geodesic inflections of the asymptotic curves.
It is captured by the contact of the surface $M$ with lines: it is the locus of points where the 
orthogonal projection of the surface has a singularity of type swallowtail or worse. 
We recall briefly some results on these projections as they are needed for interpreting the 
singularities of $k$-folding maps.
 
The orthogonal projection $P_v$ of $M$ along the direction $v\in S^2$ to the plane $T_vS^2$ 
is given by $P_v(p)=p-(p\cdot v)v$, with $p\in M$. This can be represented locally by a map-germ from the plane to the plane. 
Varying $v$ yields the family of 
orthogonal projections $P:M\times S^2\to TS^2$ given by 
$
P(p,v)=(v,P_v(p)).
$

A transversality theorem asserts that for an open and dense (i.e., generic) set of embeddings $\phi:U\to \mathbb R^3$, 
the surface $M=\phi(U)$ has the following property: for any $v\in S^2$ 
the map-germ $P_v$ has only local singularities $\mathcal A$-equivalent to one in Table \ref{tab:othgprojeSurfaR3Algb} 
at any point on $M$. By translating $p_0\in M$ to the origin and taking $M$ locally at $p_0$ in Monge form,  an open subset of  $M$ is parametrised by 
$\phi(x,y)=(x,y,f(x,y))$, with $f$ having no constant nor linear terms. 
Projecting along the direction $v=(0,1,0)$ gives 
$P_v(x,y)=(x,f(x,y))$, which is singular at the origin.
Table \ref{tab:othgprojeSurfaR3Algb}  shows the conditions on the coefficients of $f$ 
for the map-germ $P_v$ to have a singularity at the origin of 
$\mathcal A_e$-codimension $\le 2$, where $d_e(G,\mathcal A)$ denotes the $\mathcal A_e$-codimension of $G$ (see for example \cite{IFRT}), 

\begin{table}[tp]
\begin{center}
\caption{$\mathcal A_e$-codimension $\le 2$ singularities of map-germs $G:(\mathbb R^2,0)\to (\mathbb R^2,0)$.}
{\footnotesize
\begin{tabular}{llcl}
\hline 
Name & Normal form & $d_e(G,{\mathcal{A}})$ & Algebraic conditions on $f$ in (\ref{Taylorf}) \\
&&&for the singularities of $G(x,y)=(x,f(x,y))$\\
\hline
Fold          & $(x,y^2)$ &0 &$a_{22}\ne 0$\\
\hline
Cusp          & $(x,  xy+y^3)$&0 &$a_{22}=0$, $a_{21}\ne 0$, $a_{33}\ne 0$\\
\hline
Swallowtail   & $(x, xy+y^4)$&1 &$a_{22}=0$, $a_{33}=0$, $a_{21}\ne 0$, $a_{44}\ne 0$\\
\hline
Lips/beaks    & $(x,  y^3\pm x^2y)$&1 &$a_{22}=0$, $a_{21}=0$, $a_{33}\ne 0$, $a_{32}^2-3a_{31}a_{33}\ne 0$\\
\hline
Goose         & $(x,  y^3+ x^3y)$&2 &$a_{22}=0$, $a_{21}=0$, $a_{32}^2-3a_{31}a_{33}=0$, $a_{33}\ne 0$, \\
& &&$27a_{41}a_{33}^3-18a_{42}a_{32}a_{33}^2+9a_{43}a_{32}^2a_{33}-4a_{44}a_{32}^3\ne 0$\\
\hline
Butterfly    & $(x, xy+y^5\pm y^7)$& 2&$a_{22}=0$, $a_{33}=0$, $a_{44}=0$, $a_{21}\ne 0$, $a_{55}\ne 0$,\\
&&&$ (8a_{55}a_{77}-5a_{66}^2)a_{21}^2+2a_{55}(a_{32}a_{66}-20a_{43}a_{55})a_{21}+$\\
&&&$35a_{32}^2a_{55}^2\ne 0$\\
\hline
Gulls         & $(x,xy^2+y^4+y^{5})$ & 2&$a_{22}=0$, $a_{21}=0$, $a_{33}=0$, $a_{32}\ne 0$, $a_{44}\ne 0$,\\
&&& $a_{55}a_{32}^2-2a_{43}a_{44}a_{32}+4a_{31}a_{44}^2\ne 0$\\
\hline
\end{tabular}
}
\label{tab:othgprojeSurfaR3Algb}
\end{center}
\end{table}

The projection $P_v$ has a fold singularity at $p$ if, and and only if, $v$ is a non-asymptotic tangent direction to $M$ at $p$.
The singularity at $p$ is of type cusp or worse  if, and only 
if, $v$ is an asymptotic direction at $p$. For a generic surface $M$, the closure of the set of points where 
$P_v$ has a swallowtail (resp. lips/beaks) singularity is a precisely the flecnodal (resp. parabolic) curve. 
The flecnodal curve meets tangentially the parabolic curve at the cusps of the Gauss map, which 
are the gulls singularities of $P_v$ \cite{banchoffgaffneymccrory}. 
We call a point where this happens a \emph{gulls-point} of $M$

The goose (resp. butterfly) singularities of $P_v$ appear at special points on the parabolic (resp. flecnodal) curve. We call a point on $M$ where these singularities occur a \emph{goose-point}  (resp. \emph{butterfly-point}) on $M$.

The above robust features are defined in terms of the principal curvatures $\kappa_1$ and $\kappa_2$. 
We can suppose $\kappa_1<\kappa_2$ (away from umbilic curves) and give different colors to the robust features associated to 
each principal curvature. For instance, we can have a blue ridge (associated to $\kappa_1$) and a red ridge 
(associated to $\kappa_2$).

\subsection{The geometry of $k$-folding map-germs}\label{sec:surfaces}
We obtain here the robust features determined by the singularities of $k$-folding map-germs, for $k\geq 3$. 
We start with the case where $p$ is not an umbilic point and treat the $k=3$ case separately. 

\begin{theo} \label{theo:k=3}
Let $M$ be a generic surface in $\mathbb R^3$, $p$ a non umbilical point on $M$ 
and $\pi$ a plane through $p$ and orthogonal to $v$. The map-germ $F_3^{\pi}$ at $p$ is an immersion if, and only if, $v$ is not tangent to $M$ at $p$. 
It has an $S_1$-singularity
if, and only if, $v\in T_pM$  but is neither a principal nor an asymptotic direction. We have the following when $v$ is 
a principal or an asymptotic direction at $p$.
\begin{enumerate}
\item[{\rm (1)}] 
If $v$ is a principal direction at $p$, then the possible singularities of $F_3^{\pi}$ at $p$ and their geometric characterisations are as follows:

\begin{tabular}{cp{0.8\textwidth}}
$S_3$:& $p$ is not on	the sub-parabolic curve associated to $v$.\\
$S_5$:& $p$ is a generic point on the sub-parabolic curve associated to $v$.\\
$S_7$:&  $p$ is an $S_3$-point on the sub-parabolic curve associated to $v$.
\end{tabular}

\item[{\rm (2)}] Suppose that $v$  is an asymptotic  and that $p$ is a hyperbolic point. 
Then $F_3^{\pi}$ has a singularity at $p$ of type $H_{k},k=2,3,4$.
The $H_{3}$-singularities occur on a regular curve on $M$.
We call its closure  the {\rm $H_3$-curve}. 
The $H_4$-singularities occur at isolated points on this curve.

\item[{\rm (3)}]If $p$  is a parabolic point and $v$ is the unique asymptotic direction at $p$ 
which is also a principal direction, then the possible topological classes of  $F_3^{\pi}$ at $p$ and  their geometric characterisations are as follows:

\begin{tabular}{cp{0.8\textwidth}}
$X_4$:&$p$ is a generic point on the parabolic curve.\\
${\bf U}^3_4$:& $p$ is an $A_2^*$-point {\rm (}see {\rm \cite{GiblinJaneszko,GiblinWarderZak})}; it is on the closure of the $H_3$-curve.\\
${\bf X}^3_4$:& $p$ is on the sub-parabolic curve associated to $v$; the principal map with value $v$ at $p$ has a beaks singularity at $p$ {\rm (\cite{BTframes})}.\\
${\bf W}^{3,1}_4$:& $p$ is on the sub-parabolic curve associated to the other principal direction $v^{\perp}$; it is  on the closure of the $H_3$-curve; the frame map has 
a cross-cap singularity at $p$ {\rm (\cite{BTframes})}.
\end{tabular}
\end{enumerate}
\end{theo}

\begin{proof}
The proof follows by considering the defining equations and the open conditions of the strata of $3$-folding map-germs 
in Theorem \ref{theo:k=3SingReflmaps} 
and the geometric interpretation of the algebraic conditions 
in Table \ref{tab:k=3}  given in Tables \ref{tab:AlgCdFmp} and \ref{tab:othgprojeSurfaR3Algb}. 
\end{proof}

\begin{prop}\label{prop:surfacesk=3}
The $H_3$-curve of a generic surface 
is a regular curve and meets the parabolic curve tangentially at $A_2^*$
and  ${\bf W}^{3,1}_4$-points.
\end{prop}

\begin{proof}
The regularity of the $H_3$-curve follows by a transversality argument. We compute the 1-jets 
of the parabolic and $H_3$-curves and find that they are tangential at their points of intersection.
This is expected as the $H_3$-singularities of $F_3^{\pi}$ occur when $v$ is an asymptotic direction, so the 
$H_3$-curve lies 
 in the closure of the hyperbolic region of the surface.
\end{proof}

We turn now to the case $k\ge 4$, still assuming $p$ not to be an umbilical point. 
We recall that some topological classes in Tables \ref{tab:ClassifPr}, \ref{tab:ClassifAs}, \ref{tab:ClassifPar}
come with divisibility conditions on $k$. For example, the ${\bf N}^k_3$ class requires $2\mid k$ (see Remarks \ref{rems:GeomkFoldNoneUmb}).

\begin{theo} \label{theo:general k}
Let $k\geq 4$ and let $M$ be a generic surface in $\mathbb R^3$, $p$ a non umbilical point on $M$ 
and $\pi$ a plane through $p$ and orthogonal to $v$. The map-germ $F_k^{\pi}$ at $p$ is an immersion if, and only if, $v$ is not tangent to $M$ at $p$.
It has a singularity which is topologically equivalent to  $M^k_1$
if, and only if, $v\in T_pM$ but is neither a principal nor an asymptotic direction at $p$. 
We have the following when $v$ is 
a principal or an asymptotic direction at $p$. 
\begin{enumerate}

\item[{\rm (1)}]  If $v$ is a principal direction at $p$, then $F_k^{\pi}$ belongs to a stratum in 
{\rm Table \ref{tab:ClassifPr}}. The possible topological classes of $F_k^{\pi}$ 
and their geometric interpretations are as follows: 

\begin{tabular}{cp{0.83\textwidth}}
${\bf M}^k_2$:& $p$ is not on the sub-parabolic curve associated to $v$.\\
${\bf M}^k_3$:&$p$ is a generic point on the sub-parabolic curve associated to $v$.\\
${\bf M}^k_4$:&$p$ is an $S_3$-point on the sub-parabolic curve associated to $v$.\\
${\bf N}^k_3$:&$p$ is a generic point on the ridge curve associated to $v$.\\
${\bf N}^k_4$:&$p$ is a $B_3$-point on the ridge curve associated to $v$.\\
${\bf O}^k_4$:&$p$ is a $C_3$-point {\rm (}intersection point of the ridge and sub-parobolic curves associated to $v${\rm)}.
\end{tabular}

\item[{\rm (2)}] If $p$ is a hyperbolic point and $v$ is an asymptotic direction at $p$, then $F_k^{\pi}$ belongs to a stratum in 
{\rm Table \ref{tab:ClassifAs}}. The possible topological classes of $F_k^{\pi}$ and their geometric interpretations are as follows: 

\begin{tabular}{cp{0.83\textwidth}}
${\bf P}^k_2$:& $p$ is not on the flecnodal curve associated to $v$, and is 
not on the $H_3$-curve associated to $v$ when $3\mid k$.\\
${\bf P}^k_3$:&$p$ is a generic point on the $H_3$-curve associated to $v$.\\
${\bf P}^k_4$:&$p$ is an $H_4$-point on the $H_3$-curve associated to $v$.\\
${\bf Q}^k_3$:&$p$ is a generic point on the flecnodal curve associated to $v$.\\
${\bf Q}^k_4$:&$p$ is a point of intersection of the flecnodal and $H_3$-curves associated to $v$.\\
$\widetilde{\bf Q}^k_4$:&$p$ is a special point on the flecnodal curve associated to $v$.\\
${\bf R}^k_4$:& $p$ is a butterfly-point on  the flecnodal curve associated to $v$.
\end{tabular}

\item[{\rm (3)}] If $p$ is a parabolic point and $v$ is an asymptotic {\rm(}and a principal{\rm)} direction at 
$p$, then $F_k^{\pi}$ belongs to a stratum in {\rm Table \ref{tab:ClassifPar}}. The possible topological classes of $F_k^{\pi}$ 
and their geometric interpretations are as follows: 

\begin{tabular}{cp{0.83\textwidth}}
${\bf U}^k_3$:& $p$ is a generic point on the parabolic curve.\\
${\bf U}^k_4$:&$p$ is an $A^*_2$-point.\\
${\bf V}^{k,j,j'}_4$:&$p$ is a special point on the parabolic curve.\\
${\bf W}^{k,j}_4$:& $p$ is a special point on the parabolic curve.\\
${\bf W}^{3q,q}_4$:&$p$ is on the sub-parabolic 
curve associated to $v^{\perp}$; it is on the closure of the $H_3$-curve 
associated to $v$; the frame map has 
a cross-cap singularity at $p$ {\rm (\cite{BTframes})}.\\
${\bf X}^k_4$:& $p$ is on the intersection of the parabolic and sub-parabolic curves associated to $v$; 
the principal map with value $v$ at $p$ has a beaks singularity at $p$ {\rm (\cite{BTframes})}.\\
${\bf Y}^k_4$:& $p$ is a cusp of Gauss point and a gulls-point.
\end{tabular}

\end{enumerate}
\end{theo}

\begin{proof} With the setting as in the proof of Theorem \ref{theo:k=3}, the results follow by interpreting 
the conditions in Tables \ref{tab:ClassifPr}, \ref{tab:ClassifAs} and \ref{tab:ClassifPar} using 
Tables \ref{tab:AlgCdFmp} and \ref{tab:othgprojeSurfaR3Algb}.
\end{proof}

\begin{rems}\label{rems:GeomkFoldNoneUmb}
	{\rm 
		1. The singularities in Branches 1 and 2 are associated to principal directions, those in Branch 3 to asymptotic directions 
		at hyperbolic points and those in Branch 4 to asymptotic directions at parabolic points.
		
		2. Theorems \ref{theo:k=3} and \ref{theo:general k}  show clearly that  $k$-folding maps 
		capture the robust features obtained by $2$-folding maps and by the contact of the surface with lines, planes and spheres, giving thus new geometric characterisations of these features and a unified approach to study them. We also obtain a new 1-dimensional robust feature (the $H_3$-curve) and several $0$-dimensional ones: the $H_4$, ${\bf Q}^k_4$, $\widetilde{\bf Q}^k_4$, ${\bf V}^{k,j,j'}_4$ and ${\bf W}^{k,j}_4$ points. 
		
	 3. The $H_3$-curve is captured by $k$-folding maps when $k$ is divisible by $3$. 
		The sub-parabolic and flecnodal curves are captured by $k$-folding maps for any $k$ while 
		the ridge curve is captured by $k$-folding maps when $k$ is even. 
		Regarding the ridge, it is the locus of points where the surface has more infinitesimal 
			symmetry with respect to planes (\cite{BruceWilk}). The map-germs $(x,y)\mapsto (x,y^{2p},f(x,y))$ identify the pair of points $(x,y)$ and $(x,-y),$ 
		which explains why all the $F_{2p}$ folding maps capture the ridge curve.

		4. The condition $\Delta_j^k=0$ can be satisfied for any $k,j$ when the coefficients $a_{sl}$ are real, so 
		${\bf W}^{k,j}_4$-points can occur on surfaces in $\mathbb R^3$ for all $k,j$. 
		For the ${\bf V}^{k,j,j'}_4$ topological class, it follows from Proposition \ref{prop:PropertiesDeltaOmega} that 
		only ${\bf V}^{k,j,k-j}_4$, ${\bf V}^{k,j,2j}_4$ and ${\bf V}^{k,k-j,2j}_4$-points can occur on surfaces in $\mathbb R^3$.
		
		5. Theorem \ref{theo:k=3} gives a new
		geometric interpretation for the $A^*_2$-points in \cite{GiblinJaneszko,GiblinWarderZak}.
		
		6. Observe that the open conditions (those involving expressions $\ne 0$) for the 
		singularities of the $k$-folding map in Tables \ref{tab:k=3}, \ref{tab:ClassifPr}, \ref{tab:ClassifAs}, \ref{tab:ClassifPar} and 
		their associated ones in Tables \ref{tab:AlgCdFmp} and \ref{tab:othgprojeSurfaR3Algb} are not always identical. 
		For the $0$-dimensional robust features, the open conditions in both tables are satisfied on a generic surface. 
		For the $1$-dimensional robust features, this means that some special points in one setting are not special in the other. For example, the $3$-folding map does not distinguish between a $C_3$-point and a generic point on the sub-parabolic curve.
	}
\end{rems}

We consider now the situation at umbilic points. For a generic surface $M$, these occur at isolated points in its elliptic region,
and every direction in the tangent plane of $M$ at such points can be considered a principal direction.
We take $M$ locally in Monge form $z=f(z,y)$, consider the origin to be an umbilic point and write 
$
f(x,y)=\frac{\kappa}{2}(x^2+y^2)+C(x,y)+ O_4(x,y)
$
where $C$ is a homogeneous cubic form in $x,y$. We can take 
$C(x,y)$ to be  the real part of $z^2+\beta z^2\overline{z}$, with $z=x+iy$ and $\beta=s+it$ (see for example \cite{BruceWilk}). 
Then, $
C=(1+s)x^3-tx^2y+(s-3)xy^2x-ty^3. $

\begin{theo}\label{theo:umb}
Let $k\ge 3$ and let $M$ be a generic smooth surface in $\mathbb R^3$, $p$ an umbilic point  on $M$ 
and $\pi$ a plane through $p$ and orthogonal to $v\in T_{p}M$.   
\begin{enumerate}
\item If $2\nmid k$, then for almost all directions $v$ in $T_{p}M$ the singularity of $F_k^{\pi}$ at $p$ is 
of type $S_3$ when $k=3$ and of type ${\bf M}^k_2$ when $k\ge 4$.
There are three  directions {\rm (}resp. one direction{\rm )}  where the singularity is of type $S_5$ when $k=3$ and of type ${\bf M}^k_3$ when $k\ge 4$ 
if $\beta$ is inside {\rm (}resp. outside{\rm )} the outer hypocycloid $\beta=-3(2e^{2i\theta}+e^{-4i\theta})$ 
in \mbox{\rm Figure \ref{fig:PartUmb}}.

\item If $2\mid k$, then for almost all directions $v$ in $T_{p}M$ the singularity of $F_k^{\pi}$ at $p$ is 
of type ${\bf M}^k_2$.
There are three  directions {\rm (}resp. one direction{\rm )}  where the singularity is of type ${\bf M}^k_3$ 
if $\beta$ is inside {\rm (}resp. outside{\rm )} the hypocycloid $\beta=-3(2e^{2i\theta}+e^{-4i\theta})$.
There are also three directions {\rm (}resp. one direction{\rm )}  where the singularity is of  type ${\bf N}^k_3$ 
when $\beta$ is inside {\rm (}resp. outside{\rm )}  the inner hypocycloid $\beta=2e^{2i\theta}+e^{-4i\theta}$ in \mbox{\rm Figure \ref{fig:PartUmb}}.
\end{enumerate}
\end{theo}

\begin{figure}
\begin{center}
	\includegraphics[scale=0.35]{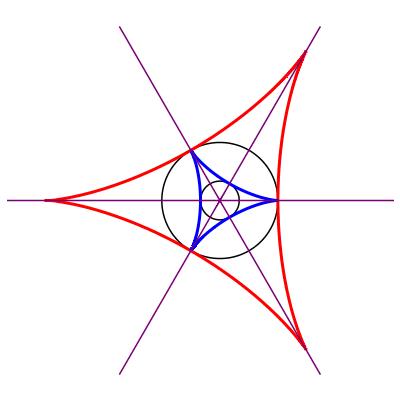}
\end{center}
\caption{Partition of the space of cubic forms.}
\label{fig:PartUmb}
\end{figure}

\begin{proof}
We take $v=(\cos(\theta),\sin(\theta),0)$, $\theta\in[0,2\pi]$ and consider the rotation 
$$
R=\left(
\begin{array}{ccc}
\sin(\theta)&\cos(\theta)&0\\
-\cos(\theta)&\sin(\theta)&0\\
0&0&1
\end{array}
\right),
$$
which takes the direction $(0,1,0)$ to $v$. Then, 
$$
F_k^{\pi}\circ R^{-1}(x,y)=(x\sin(\theta)-y\cos(\theta),(x\cos(\theta)+y\sin(\theta))^k,f(x,y)).
$$

Changes of coordinates in the source give
$$
F_k^{\pi}\circ R^{-1}(X,Y)=(X,Y^k,f(X\sin(\theta)+Y\cos(\theta),-X\cos(\theta)+Y\sin(\theta))).
$$
We denote by $\bar{a}_{lj}$ the coefficient of $X^{l-j}Y^j$ in 
the Taylor expansion of $f(X\sin(\theta)+Y\cos(\theta),-X\cos(\theta)+Y\sin(\theta))$.
The proof follows then considering the conditions for the singularities of $F_k^{\pi}$ in Tables 
\ref{tab:k=3} and \ref{tab:ClassifPr}. 

(1) We have 
$$
\bar{a}_{31}=(s-3)\cos(\theta)^3-t\cos(\theta)^2\sin(\theta)+(9+s)\sin(\theta)^2\cos(\theta)-t\sin(\theta)^3.
$$

When $2\nmid k$ and $\bar{a}_{31}\ne 0$, the singularity of $F_k^{\pi}$ at the origin is 
of type $S_3$ when $k=3$ or of type $M^k_2$ when $k\ge 4$.

The coefficient $\bar{a}_{31}$ is a cubic form in $\cos(\theta)$ and $\sin(\theta)$. Its discriminant is 
the hypocycloid $\beta=-3(2e^{2i\theta}+e^{-4i\theta})$. The cubic has 
three roots for $\beta$ inside the hypocycloid and one root when it is outside.
For $v$ corresponding to one of these roots,  the singularity of $F_k^{\pi}$ is of type 
$S_5$ when $k=3$ and of type ${\bf M}^k_3$ when $k\ge 4$, provided $\bar{a}_{41}\ne 0$. 
The condition $\bar{a}_{41}\ne 0$ is satisfied at umbilic points on generic surfaces.
 
(2) We have 
$$
\bar{a}_{33}=(1+s)\cos(\theta)^3-t\cos(\theta)^2\sin(\theta)+(s-3)\cos(\theta)\sin(\theta)^2-t\sin(\theta)^3.
$$

When $2\mid k$  (so $k\ge 4$)  and $\bar{a}_{33}\ne 0$, the singularity of $F_k^{\pi}$ at the origin is 
of type ${\bf M}^k_2$. We also get the ${\bf M}^k_3$ singularities as in (1) when $\bar{a}_{31}=0$.  

The coefficient $\bar{a}_{33}$ is also a cubic form in $\cos(\theta)$ and $\sin(\theta)$. Its discriminant is 
the hypocycloid $\beta=2e^{2i\theta}+e^{-4i\theta}$. The cubic has 
three roots for $\beta$ inside the hypocycloid and one root when it is outside.
For $v$ corresponding to one of these roots,  the singularity of $F_k^{\pi}$ is 
of type ${\bf N}^k_3$ if $\bar{a}_{31}\ne 0$. 

We have $\bar{a}_{33}=\bar{a}_{31}=0$ if, and only if, $\beta$ 
is on one of the tangent lines $t(3s^2-t^2)=0$ 
to the hypocycloids at their cusp points, see Figure \ref{fig:PartUmb}. (On these lines, 
the singularity is of type ${\bf O}^k_4$ or more degenerate. This singularity does not occur at umbilic points on a generic surface.)
\end{proof}

\begin{rem}\label{rems:umb}
{\rm 
1. Theorem \ref{theo:umb} is merely another interpretation of the results 
in \cite{BruceWilk,wilkinson} when using the geometric characterisations of the singularities of $k$-folding 
maps in  Theorems \ref{theo:k=3} and \ref{theo:general k}. We know from \cite{Bruce84, BruceWilk,wilkinson} that 
there are one or three ridge curves and   one or three sub-parabolic curves at umbilic points on a generic surface.
These curves meet transversally and change colour at the umbilic point. 
	
2. Figure \ref{fig:PartUmb} is first obtain in \cite{BruceWilk} when considering 2-folding maps. 
In that case both hypocycloids are present, whereas when $k\ge 3$ only one of them 
is present when $k$ is odd (both are present when $k$ is even). 
Also in \cite{BruceWilk} is considered the circle $|\beta|=3$  
which corresponds to the Monge-Taylor map failing to be transverse to the umbilics stratum. 
The circle $|\beta|=1$ is also exceptional and corresponds to two ridges through the umbilic being tangential, see \cite{BGT}.
As these conditions are geometric, the circles $|\beta|=1$ and $|\beta|=3$ can also be considered exceptional for 
$k$-folding maps and are added to  Figure \ref{fig:PartUmb}.

3. Umbilic points on a generic surface occur at elliptic points, that is why we do not get flecnodal curves or $H_3$-curves 
through such points. 
}
\end{rem}

\begin{acknow}
GPS was partially supported by the Basque Government through the BERC 2018-2021 program and Gobierno Vasco Grant IT1094-16, by the Spanish Ministry of Science, Innovation and Universities: BCAM Severo Ochoa accreditation SEV-2017-0718, by the ERCEA Consolidator Grant 615655 NMST, and by Programa de Becas Posdoctorales en la UNAM, DGAPA, Instituto de Matem\'aticas, UNAM.

FT was partially supported by the 
FAPESP Thematic project grant 2019/07316-0  and the CNPq research grant 303772/2018-2.
\end{acknow}


\noindent
GPS: Departament de Matem\`atiques - Universitat de València, Calle Dr. Moliner 50,
46100, Burjassot (València), Spain. \\
E-mail: guillermo.penafort@uv.es
\\

\noindent
FT: Instituto de Ci\^encias Matem\'aticas e de Computa\c{c}\~ao - USP, Avenida Trabalhador s\~ao-carlense, 400 - Centro,
CEP: 13566-590 - S\~ao Carlos - SP, Brazil.\\
E-mail: faridtari@icmc.usp.br

\end{document}